\documentclass[twoside,12pt]{amsart}
\usepackage[utf8]{inputenc}
\usepackage{amsmath,amsthm,amsfonts,amssymb,verbatim,enumerate}
\usepackage{array,charter}
\usepackage{longtable}
\usepackage{verbatim}
\usepackage{pdfpages}

\usepackage{latexsym,mathptm, times,mathcomp,lscape}
   \usepackage{tikz,stmaryrd}
\usepackage{calligra}
\usepackage{calrsfs}
\usepackage[mathscr]{euscript}
\usepackage{dutchcal}
\makeatletter
\@namedef{subjclassname@2020}{
  \textup{2020} Mathematics Subject Classification}
\makeatother

\input amssym.def
\input amssym 
\input xypic
\input xy
\xyoption{all}

\usepackage[right=1.5in,left=1.5in,top=1in,bottom=1in]{geometry}
\usepackage{amssymb}

\def\fakeht{\vphantom{E^{E_E}_{E_E}}}
\def\bs{\bigskip}
\def\AA{\frak A}
\def \transpose{^{\rm T}}
\def \ms{\medskip}
\def\maxm{\mathfrak m}
\def\ev{\operatorname{ev}}
\def\Z{\mathbb Z}
\def\ts{\textstyle}
\def\blop{q}
\def\Br{\operatorname{Br}}
\def\la{\langle}
\def\ra{\rangle}
\def\B{\mathbb B}
\def\g{\gamma}
\def\bs{\bigskip}
\def\mA{\mathcal A}
\def\mB{\mathcal B}
\def\mC{\mathcal C}
\def\mD{\mathcal D}
\def\mf{\mathcal f}
\def\CA{\operatorname{CA}}
\def\Sym{\operatorname{Sym}}
\def\Ext{\operatorname{Ext}}
\def\kk{{\pmb k}}
\def\im{\operatorname{im}}
\def\w{\wedge}
\def\SM{\operatorname{SM}}
\def\ann{\operatorname{ann}}
\def\p{\oplus}
\def\mL{\mathfrak L}
\def\Hom{\operatorname{Hom}}
\def\t{\otimes}
\def\w{\wedge}

\def\HH{\operatorname{H}}
\def\C{\mathfrak C}

\def\pd{\operatorname{pd}}
\def\grade{\operatorname{grade}}
\def\a{\alpha}
\def\Pbar{\overline{P}}
\def\bw{\bigwedge}
\def\blip{\theta}
\def\map{\operatorname{map}}
\def\rank{\operatorname{rank}}
\def\proj{\operatorname{proj}}
\def\Ass{\operatorname{Ass}}
\newcounter{nameOfYourChoice}
\newtheorem{theorem}{Theorem}[section]

\newtheorem{lemma}[theorem]{Lemma}
\newtheorem{observation}[theorem]{Observation}
\newtheorem{proposition}[theorem]{Proposition}
\newtheorem{corollary}[theorem]{Corollary}
\newtheorem{claim-no-advance}[equation]{Claim}
\newtheorem{lemma-no-advance}[equation]{Lemma}

\theoremstyle{definition}
\newtheorem{data}[theorem]{Data}
\newtheorem{background}[theorem]{Background}

\newtheorem{definition}[theorem]{Definition}
\newtheorem{convention}[theorem]{Convention}
\newtheorem{conventions}[theorem]{Conventions}
\newtheorem{definition-no-advance}[equation]{Definition}
\newtheorem{chunk}[theorem]{}
\newtheorem{chunk-no-advance}[equation]{}
\newtheorem{example}[theorem]{Example}

\newtheorem{remark}[theorem]{Remark}
\newtheorem{remarks}[theorem]{Remarks}
\newtheorem*{Remark}{Remark}

\newtheorem{Earlier Notes}[theorem]{Earlier Notes}

\numberwithin{equation}{theorem}

\begin{document}

\baselineskip=16pt

\title[Artinian Gorenstein algebras]{Artinian Gorenstein algebras of embedding dimension four and socle degree three
over an arbitrary field.}

\date{\today}

\author[Sabine El Khoury and Andrew R. Kustin]
{Sabine El Khoury and Andrew R. Kustin}

\address{Mathematics Department,
American University of Beirut,
Riad el Solh 11-0236,
Beirut,
Lebanon}
\email{se24\@aub.edu.lb}

\address{Department of Mathematics\\ University of South Carolina\\\newline
Columbia SC 29208\\ U.S.A.} \email{kustin\@math.sc.edu}



\thanks{We made extensive use of   Macaulay2 \cite{M2}. We are very appreciative of this computer algebra system.}

\begin{abstract} Let $\kk$ be an arbitrary field,   $\AA$ be  a standard graded
Artinian Gorenstein  
$\kk$-algebra of embedding dimension four and  socle degree three, and $\pi:P\to \AA$ be a surjective graded homomorphism from a polynomial ring with four variables over $\kk$ onto $\AA$. We give the minimal generators of the kernel of $\pi$ and the minimal homogeneous resolution of $\AA$ by free $P$-modules.
We give formulas for the entries in the matrices in the resolution in terms of the coefficients of the Macaulay inverse system for $\AA$.  We have implemented these formulas in Macaulay2 scripts.
  
The ideal $\ker \pi$ has either $6$, $7$, or $9$ minimal generators. The number of minimal generators and the precise form of the minimal resolution are determined by the rank of a $3\times 3$ symmetric matrix of constants that we call $\SM$. 

If  
$\ker \pi$ requires more than six generators, then we prove that $\ker\pi$ is the sum of two linked perfect ideals of grade three.  
If the $\ker \pi$ is six-generated, then we prove that $\AA$ is a hypersurface section of a codimension three Gorenstein algebra.

Our approach is based on the structure of Gorenstein-linear resolutions and the theorem that, except for exactly one exception, $\AA$ has the weak Lefschetz property.
 \end{abstract}

\maketitle

\tableofcontents

\section{Introduction.}

Let $\kk$ be an arbitrary field, $U$ be a vector space over $\kk$ of dimension four, $U^*$ be $\Hom_{\kk}(U,\kk)$, $P$ be the polynomial ring $P=\Sym_{\bullet} U$, and $D_{\bullet}U^*$ be the corresponding divided power algebra.

Marques, Veliche, and Weyman \cite{MVW} describe all Artinian Gorenstein algebras of embedding dimension four and socle degree three over an algebraically closed field of characteristic zero. Their description uses a classification of homogeneous forms of degree three over an algebraically closed field of characteristic zero. (In characteristic zero,  the Macaulay inverse system of an Artinian Gorenstein $\kk$-algebra of socle degree three  may as well be a homogeneous cubic form because $D_\bullet U^* $ is isomorphic to $\Sym_\bullet U$.)

The paper \cite{MVW} expresses the defining ideal of each algebra under consideration as the generalized sum of two linked perfect ideals of grade three. There are many cases, many ad hoc arguments, and there  is no uniform answer. 
The paper \cite{MVW} includes an example that indicates that the conclusion might fail 
in characteristic two.

The characteristic two example \cite[Ex.~5.6]{MVW} is very interesting! After changing the variables, one sees that it is the unique  Artinian Gorenstein $\kk$-algebra of embedding dimension four and socle degree three   which does not have the weak Lefschetz property. This algebra has a very pretty resolution and the defining ideal is indeed the sum of two linked perfect grade three ideals; see Section~\ref{Frob}.
 
We proved in \cite{K23}
that, with one exception,  every standard graded Artinian Gorenstein $\kk$-algebra of embedding dimension four and socle degree three has a weak Lefschetz element. The proof in \cite{K23} is fairly complicated: it involves many cases and many ad hoc arguments. 
However, once one has the result of \cite{K23}, then one is able to prove the results of \cite{MVW} over {\bf an arbitrary field} in a uniform manner. 

In addition to the result of \cite{K23}, the other main ingredient in the present paper  is the 
structure of Gorenstein-linear resolutions which is given in \cite{EKK3}.
In a sequence of three papers, we gave explicit formulas for the differentials in the minimal homogeneous resolution of $P/I$ by free $R$-modules, where $P=\kk[x_1,\dots,x_d]$ is a standard-graded polynomial ring over the field $\kk$ and $I$ is a grade $d$ homogeneous Gorenstein ideal, provided the resolution is Gorenstein-linear in the sense that the resolution has the form
$$0\to F_d\xrightarrow{f_d}F_{d-1}\xrightarrow{f_{d-1}}F_{d-2}\xrightarrow{f_{d-2}}\dots\xrightarrow{f_3}F_{2}\xrightarrow{f_2}F_{1}\xrightarrow{f_1}P$$
and all the entries of all of the matrices $f_i$, with $2\le i\le d-1$, are linear. The formulas are given in terms of the coefficients of a Macaulay inverse system for $I$. 
The  paper \cite{EKK1} proves that the project can be done;  \cite{EKK2} carries out the project when $d=3$, and \cite{EKK3} carries out the project for all $d$.

In the present paper,   
 $\AA$ 
is an 
  Artinian Gorenstein
$\kk$-algebra of embedding dimension four and socle degree three. 
Let $\Phi_3\in D_3U^*$ be a Macaulay inverse system for $\AA$. In other words,
$\AA$ is isomorphic to $P/\ann \Phi_3$. Fix an element $x$ in $U$ which is a weak Lefschetz element on $\AA$. Then $P/\ann x\Phi_3$ has a Gorenstein-linear resolution; consequently, the resolution of $P/\ann x\Phi_3$ is given by the formulas of \cite{EKK3}.  One quickly obtains a resolution of $\AA$. 
Furthermore, there is a symmetric $3\times 3$ matrix of constants $\SM$ which controls the precise form of the minimal resolution  of $\AA$. If $\SM$ has rank zero, then the defining ideal of $\AA$ has nine minimal generators and the minimal resolution of $\AA$ is given in Theorem \ref{9.1}.
If $\SM$ has rank one, then the defining ideal of $\AA$ has seven minimal generators. One may change variables to put $\SM$ in the form
\begin{equation}
\label{form}
\bmatrix \a&0&0\\0&0&0\\0&0&0\endbmatrix,\end{equation}for some nonzero $\a\in\kk$. Once this is done, then
 the minimal resolution of $\AA$ is given in Theorems \ref{7.1} and \ref{7.2}. If $\SM$ has rank at least two, then the defining ideal of $\AA$ has six minimal generators, $\AA$ is a hypersurface section of a codimension three Gorenstein ring, and the minimal resolution of $\AA$ is given in Theorems \ref{6.1} and \ref{6.2}.

One consequence of the above results is
\begin{corollary}
\label{1.1} 
If $\kk$ is an arbitrary field,  $\AA$ is  a standard graded
Artinian Gorenstein  
$\kk$-algebra of embedding dimension four and  socle degree three, and $\pi:P\to \AA$ is a surjective graded homomorphism from a polynomial ring with four variables over $\kk$ onto $\AA$, then the defining ideal of $\AA$ has $6$, $7$, or $9$ minimal generators.
\end{corollary}

Corollary~\ref{1.1} is already known 
in \cite{AS22} and  \cite{MVW} when $\kk$ has characteristic zero. Indeed, these two papers predict the exact graded Betti numbers of $\AA$. We  
obtain the exact same graded Betti numbers in the general case as are obtained in the characteristic zero case.

\section{Notation, conventions, and elementary results.}

\begin{convention}
In this paper, $\kk$ is an arbitrary, but fixed, field. Every unadorned operation is a functor in the category of $\kk$-modules. In particular, we write  $\otimes$, $\w$, $\bigwedge^i$, $\operatorname{Sym}_j$, $D_j$, $L^{i}_{j}$, $K^{i}_{j}$, to mean $\otimes_{\kk}$, $\w_{\kk}$, $\bigwedge^i_{\kk}$, $\operatorname{Sym}_j^{\kk}$, $D_j^{\kk}$, ${L^i_j}{}^{\kk}$ and ${K_i^j}{}^{\kk}$, respectively. 
We 
write $U^*$ to mean $\operatorname{Hom}_{{\kk}}(U,{\kk})$, $DU^*$ to mean $D(U^*)$, and $x^{*(n)}$ to mean $(x^*)^{(n)}$. 
\end{convention}

\begin{chunk} 
If $M$ is a matrix with entries in the ring $P$ and $r$ is a positive integer, then $I_r(M)$ is the ideal of $P$ generated by the $r\times r$ minors of $M$.\end{chunk} 

\begin{chunk}
Let $I$ be a proper ideal in a commutative Noetherian ring $R$.
The {\it grade} of   
$I$   
is the length of a maximal  $R$-regular sequence in $I$. 
\end{chunk}

\begin{chunk}Let $M$ be a nonzero finitely generated module over a Noetherian ring $R$, 
$\ann(M)$ be the annihilator of $M$, and $\pd_RM$ be the projective dimension of $M$. 
It is always true that
 $$\operatorname{grade}\ann (M)\le \operatorname{pd}_R M.$$
When equality holds, then $M$ is called a {\it perfect $R$-module}.
Let $I$ be a proper ideal in a commutative Noetherian ring $R$.
The ideal $I$ is called a {\it perfect ideal} if $R/I$ is a perfect $R$-module.
\end{chunk}

\begin{convention}
 \label{2.5}
For any set of variables $\{x_1,\dots,x_r\}$ and any degree $s$, we write
$$\binom{x_1,\dots,x_r} s$$ for   the set of monomials of degree $s$ in the variables $x_1,\dots,x_r$.\end{convention}

\begin{conventions}
\begin{enumerate}[\rm(a)] 
 \item The graded algebra $\AA=\bigoplus_{0\le i}\AA_i$ is a {\it standard graded $\AA_0$-algebra} if $\AA_1$ is finitely generated as an $\AA_0$-module  and 
$\AA$ is generated as an $\AA_0$-algebra by $\AA_1$.

\item Let $\kk$ be a field and $\AA=\bigoplus \AA_i$ be a standard graded $\kk$-algebra. Then $\AA$ has the {\it weak Lefschetz property}  if there exists a linear form $\ell$ of $\AA_1$ such that the $\kk$-module homomorphism
$$\mu_{\ell}:\AA_i\to \AA_{i+1}$$ has maximal rank for each index $i$, where $\mu_{\ell}$ is multiplication by $\ell$. 
(A homomorphism $\xi:V\to W$ of finitely generated $\kk$-modules has {\it maximal rank} if $\rank \xi$ is equal to $\min\{\dim V,\dim W\}$.)

\item If $\AA=\bigoplus_{i=0}^\sigma \AA_i$ 
 is an Artinian standard-graded $\kk$-algebra, then $\AA$ is {\it Gorenstein with socle degree $\sigma$} if $\AA_\sigma$ is a one dimensional vector space and every nonzero ideal of $\AA$ contains $\AA_\sigma$.
\end{enumerate}
\end{conventions}
\begin{chunk}
 Let  $\AA=\bigoplus_{i=0}^\sigma \AA_i$ 
 be an  Artinian standard-graded $\kk$-algebra.
Let $U$ be the vector space $\AA_1$, $P$ be the polynomial ring $P=\Sym_\bullet U$, $I$ be a homogeneous ideal of $P$ with $\AA$ isomorphic to $P/I$, $U^*$ be the dual space $\Hom_{\kk}(U,\kk)$ of $U$,  and $D_\bullet U^*$ be the divided power $\kk$-algebra $\bigoplus_{0\le i} D_iU^*$, with $$D_iU^*=\Hom_{\kk}(\Sym_i U,\kk).$$
The rules for a divided power algebra are recorded in \cite[Section~7]{GL} or \cite[Appendix 2]{Eis}. (In practice these rules say that $\nu^{(n)}$ behaves like $\nu^n/(n!)$ would behave if $n!$ were a unit in $\kk$, for each element $\nu\in U^*$.)

The algebras $D_\bullet U^*$ and  $\Sym_\bullet U$ are modules over one another.
If $u_a$ is in $\Sym_aU$ and $\nu_b\in D_b U^*$, then $u_a\nu_b$ is the element of $$D_{b-a}U^*=\Hom_\kk(\Sym_{b-a}U,\kk)$$ with $$u_a\nu_b(u_{b-a})=\nu_b(u_au_{b-a}).
$$

Macaulay duality guarantees that 
$$\ann_{D_\bullet U^*}I$$ is a cyclic $P$-submodule of $D_\bullet U^*$ generated by an element 
in $D_\sigma U^*$. (Any generator 
$\ann_{D_\bullet U^*}I$
in $D_\sigma U^*$ 
is called 
a Macaulay inverse system for $\AA$.) Furthermore, if $\Phi_\sigma$ is a Macaulay inverse system for $\AA$, then
$$I=\ann_P \Phi_\sigma.$$
\end{chunk}

\ms

All of the results which are established in this paper use the following data and notation.
\begin{data}
\label{2.1} 
\begin{enumerate}
[\rm(a)] \item\label{2.1.a} Let $\kk$ be an arbitrary field, $U$ be a vector space over $\kk$ of dimension four, $U^*$ be $\Hom_{\kk}(U,\kk)$, $P$ be the polynomial ring $P=\Sym_{\bullet} U$, and $D_{\bullet}U^*$ be the corresponding divided power algebra. Let $\Phi_3$ be an element of $D_3U^*$ and $\AA$ be the algebra $P/\ann \Phi_3$. Assume that $x$ in $U$ is a weak Lefschetz element in $\AA$ and $\ann \Phi_3$ does not contain any nonzero linear elements.

\item Let $p:U\to U^*$ be the homomorphism $p(u)=u(x\Phi_3)$.

\item Fix an element  
$x^*\in U^*$ with $x^*(x)=1$. 
 Let $U_0=\ker(x^*:U\to \kk)$. Observe that  $U_0$ is a $\kk$-vector space 
of dimension 
 $d-1$ and  $U=\kk x\oplus U_0$.  
Let $$U_0^*=\{f\in U^*\mid f(x)=0\}.$$ Observe that $U_0^*=\kk x^* \p U_0^*$. 
\item\label{2.1.d} Define the symmetric bilinear form $\la -,-\ra:\Sym_2U_0\times \Sym_2U_0\to \kk$ by
$$\la u_{2,0},u_{2,0}'\ra= [u_{2,0}p^{-1}(u_{2,0}'\Phi_3)](\Phi_3),$$for elements $u_{2,0}$ and $u_{2,0}'$ in $\Sym_2U_0$.
\item Fix bases $\omega_{U_0}$ and $\omega_{U_0^*}$ for $\bw^3U_0$ and $\bw^3U_0^*$ with 
$$\omega_{U_0}(\omega_{U_0^*})=1.$$
Let $$\omega_U=x\w \omega_{U_0}\quad \text{and}\quad \omega_{U^*}=\omega_{U_0^*}\w x^*$$ be bases for $\bw^4 U$ and $\bw^4 U^*$.

\item If a basis $y,z,w$ is chosen for $U_0$, then $y^*,z^*,w^*$ is automatically chosen to be the basis for $U_0^*$ which is dual to the basis $y,z,w$ for $U_0$. Furthermore, $\omega_{U_0}$ and $\omega_{U_0^*}$ are automatically chosen to be
 $$\omega_{U_0}=y\w z\w w\quad\text{and}\quad \omega_{U_0^*}=w^*\w z^*\w y^*.$$

\item Let $\operatorname{proj}:\operatorname{Sym}_{\bullet}U\to\operatorname{Sym}_{\bullet}U_0$ the projection map induced  by $U=\kk x_1\oplus U_0$.

\end{enumerate}
\end{data}

\begin{Remark}Please notice that the final hypothesis in Data~\ref{2.1}.(\ref{2.1.a}), that ``$\ann \Phi_3$ does not contain any nonzero linear elements'' is equivalent to assuming that $\AA$ has embedding dimension four. \end{Remark} 

\begin{chunk}
\label{ev*} Let $U$ be a finite dimensional vector space over the field $\kk$. The evaluation map $D_NU_0^*\otimes \operatorname{Sym}_NU_0\xrightarrow{\ \operatorname{ev}\ }\kk$ is a coordinate-free map; consequently, the dual
$$\kk\xrightarrow{\ \operatorname{ev}^*\ } \operatorname{Sym}_NU_0 \otimes D_NU_0^* $$is also a coordinate-free map. Thus, $$\operatorname{ev}^*(1) =\sum\limits_{m\in \{m\}}m\otimes m^*$$ is an element of $  \operatorname{Sym}_NU_0 \t D_NU_0^*$
 which is completely coordinate-free, where $\{m\}$ is a basis for $U$ and $\{m^*\}$ is the corresponding dual basis for $U^*$.
\end{chunk}

\begin{observation} 
\label{GLR} Adopt Data~{\rm\ref{2.1}}.
The $P$-algebra $P/\ann x\Phi_3$ has a Gorenstein-linear resolution and the homomorphism $p:U\to U^*$ is an isomorphism.\end{observation}

\begin{proof}The fact that $\AA$ has embedding dimension four guarantees that zero is the only linear element in $\ann \Phi$.  
The fact that $x$ is a weak Lefschetz element on $P/\ann \Phi_3$ guarantees that zero is the only linear element of  $\ann x\Phi_3$. The socle degree of $P/x\ann \Phi_3$ is two. Apply \cite[Prop.~1.8]{EKK1} to see that $P/\ann x\Phi_3$ has a Gorenstein-linear resolution.
\end{proof}

\begin{chunk}
\label{L-K} Let $U_0$ be a finite dimensional vector space over the field $\kk$. The vector spaces  $L^a_bU_0$ and $K^a_bU_0$ represent the Schur and Weyl functors 
for a hook
as described, for example,  in \cite[Data~2.1]{EKK1}. (The translation between our notation and the notation of Buchsbaum and Eisenbud and the notation of Weyman is contained in \cite[2.1 and 2.5]{EKK1}.)
In particular, 
$$\textstyle \begin{array}{lll}
L^a_bU_0&=&\ker \left(\bigwedge^aU_0\otimes{\operatorname{Sym}}_bU_0\xrightarrow{\kappa}
\bigwedge^{a-1}U_0\otimes{\operatorname{Sym}}_{b+1}U_0\right)\quad\text{and}\vspace{5pt}\\
K^a_bU_0&=&\ker \left(\bigwedge^aU_0\otimes D_bU_0^*\xrightarrow{\eta}
\bigwedge^{a-1}U_0\otimes D_{b-1}U_0^*\right),\end{array}$$where $\kappa$ is a Koszul complex map and $\eta$ is an Eagon-Northcott complex map.
If $x_1,\dots, x_d$ and $x_1^*,\dots,x_d^*$ are a pair of dual bases for $U$ and $U^*$, then  
\begin{equation}\label{kappa-eta}\kappa(\theta\otimes u)=\sum\limits_{i=2}^dx_i^*(\theta)\otimes x_iu\quad \text{and}\quad \eta(\theta\otimes \nu)=\sum\limits_{i=2}^dx_i^*(\theta)\otimes x_i(\nu)\end{equation} 
for $\theta\in \bigwedge^aU_0$, $u\in \operatorname{Sym}_bU_0$, and $\nu\in D_{b}U_0^*$. In light of (\ref{ev*}), the description of $\kappa$ and $\eta$ given in (\ref{kappa-eta}) is coordinate-free.
\end{chunk}

\begin{chunk} 
We often write ``after a change of variables'' some expression becomes some other expression. Indeed, the paper \cite{K23} guarantees that (with one exception) every Gorenstein algebra $\AA$ under consideration (as described in Data~\ref{2.1}) contains a weak Lefschetz element. This paper always assumes   (except in Section~\ref{Frob}) that $x$ is a weak  Lefschetz element. Often it is necessary to change the variables in order to make $x$ be a weak Lefschetz element. Also, when we find the precise form of the minimal resolution when $\SM$ has rank one we assume that one has made a change of variables to put ``$\SM$'' in the form (\ref{form}). It is not difficult to change variables. We show the procedure here.

Suppose that $x,y,z,w$ and $x^*,y^*,z^*,w^*$ is a pair of dual bases for $U$ and $U^*$, respectively and that  $X,Y,Z,W$ and  $X^*,Y^*,Z^*,W^*$ is another  pair of dual bases for $U$ and $U^*$.
If $M$ is an invertible matrix with entries in $\kk$ and  $$\bmatrix X&Y&Z&W\endbmatrix 
 =\bmatrix x&y&z&w\endbmatrix M,$$ then
$$\bmatrix X^*&Y^*&Z^*&W^*\endbmatrix M\transpose 
 =\bmatrix x^*&y^*&z^*&w^*\endbmatrix.$$
\end{chunk}

\section{The structure of Gorenstein-linear  resolutions, general case.}
We record the main result  
of \cite{EKK3}. We employ  
a different point of view than the one used in \cite{EKK3}. We work over a field $\kk$ rather than over the ring of integers $\Z$. The coefficients of our Macaulay inverse system $\Phi_{2n-2}$ are elements of $\kk$ rather than variables. We deal with the inverse of $p$ rather than ``$q$'', which is essentially the classical adjoint of $p$. We do not have reason to write $\det p$. In \cite{EKK3}, $\pmb \delta$ represents the determinant of $p$; in other words, $\pmb \delta$ is a polynomial of high degree in the coefficients of $\Phi_{2n-2}$. The results of  \cite{EKK3} hold after the coefficients of $\Phi_{2n-2}$ have been specialized to elements of $\kk$, provided the image of  $\pmb \delta$ is a unit.

There is one other difference between the language of \cite{EKK3} and the language we use. In both situations, it is necessary to consider an ``integral'' of $\Phi_{2n-2}$; in other words, we must consider an element $\Phi_{2n-1}$ with \begin{equation}\label{integral}x(\Phi_{2n-1})=\Phi_{2n-2}\end{equation} The plan of \cite{EKK3} was to specify everything and leave no choices. In \cite{EKK3} we chose $\Phi_{2n-1}$ to have a special property in addition to
(\ref{integral}). We insisted that when $\Phi_{2n-1}$ was written as an element of $$D_{2n-1} U_0^*\p D_{2n-2}U_0^*x^*\p D_{2n-3} U_0^*x^{*(2)}\p\dots\p 
D_{0}U_0^* x^{*(2n-1)},$$(where the decomposition $U=\kk x\p U_0$ had been previously made), then the ``constant term'' in $D_{2n-1} U_0^*$ is zero.  (See, for example, \cite[Rem.~3.4.(a)]{EKK-qp} or \cite[Data 2.5.(d)]{EKK3}.) Unfortunately, we later realized that this choice limited the potential applications of \cite{EKK3}. If one uses an arbitrary $\Phi_{2n-1}$ rather than the special $\widetilde{\Phi}_{2n-1}$, with ``constant term'' zero, one gets a different resolution; however, the new resolution is isomorphic to the old resolution. When $d=3$, the argument for this isomorphism is given explicitly as \cite[Thm.~3.6]{EKK-qp}. The exact same argument works for all $d$. 

\begin{data}
\label{Opening-Data} Let $\kk$ be a field,  $U$ be a vector space over $\kk$ of dimension $d$, 
and
 $n\ge 2$ be an integer.  Let $\Phi_{2n-2}$ be an element of $D_{2n-2}U^*$. Assume that the homomorphism $p:\Sym_{n-1}U\to D_{n-1}U^*$, given by $p(u_{n-1})=u_{n-1}\Phi_{2n-2}$, is an isomorphism.

\begin{enumerate}[(a)]

\item Fix elements $x_1\in U$ and $x_1^*\in U^*$ with $x_1^*(x_1)=1$. 
\item Let $U_0=\ker(x_1^*:U\to \kk)$. Observe that  $U_0$ is a $\kk$-vector space 
of dimension 
 $d-1$ and  $U=\kk x_1\oplus U_0$. 
\item View $U_0^*$ as the submodule of $U^*$ with $U_0^*(x_1)=0$.
\item Let $\Phi_{2n-1}$ be an element of $P\otimes D_{2n-1}U^*$ with $x_1(\Phi_{2n-1})=\Phi_{2n-2}$ in $D_{2n-2}U^*$. 
\item Let $x_2,\dots,x_d$ be a basis for $U_0$ and $x_2^*,\dots,x_d^*$ be the corresponding dual basis for $U_0^*$.

\item Let $\operatorname{proj}:\operatorname{Sym}_{n-1}U\to\operatorname{Sym}_{n-1}U_0$ the projection map induced  by $U=\kk x_1\oplus U_0$.

\end{enumerate}
\end{data}

The following result is the main theorem of \cite{EKK3}. The complex $(F,f)$ is defined in \cite[Def.~3.1]{EKK3}. The complex $(F,f)$ is shown to be a resolution in \cite[Cor.~6.17.g]{EKK3}. The universal property of $(F,f)$ is established in \cite[Cor.~6.18]{EKK3}. The vector spaces $K^a_bU_0$ and $L^a_bU_0$ are defined in \ref{L-K}. 

\begin{theorem}
\label{MTEKK3} 
Adopt Data~{\rm\ref{Opening-Data}}. Then the minimal homogeneous resolution of 
$$P/\ann \Phi_{2n-2}$$ by free $P$-modules is given by
 $$(F,f): \quad 0\to F_d\xrightarrow{f_d} F_{d-1}\xrightarrow{f_{d-1}}\dots \xrightarrow{f_2} F_{1}\xrightarrow{f_{1}}F_0,$$ as described below. 
\begin{enumerate}[\rm (a)]

\item\label{def1-a} Define the free $P$-module $F_r$ by
$$F_r=\begin{cases}
P,&\text{if $r=0$},\\P\otimes (K^{r-1}_{n-1}U_0\oplus L^{r-1}_{n}U_0),&\text{if $1\le r\le d-1$},\\
P\otimes \bigwedge^{d-1}U_0,&\text{if $r=d$}.\end{cases}$$
\item\label{diff-gen}  The $P$-module homomorphism  $f_1:F_1\to F_0=P$ is defined by  $$\begin{array}{ll} f_1(\nu_{n-1,0})=x_1p^{-1}(\nu_{n-1,0}),&\text{for $\nu_{n-1,0}\in K^0_{n-1}U_0$,}\vspace{5pt}\\f_1(u_{n,0})=u_{n,0}-x_1p^{-1}(u_{n,0}(\Phi_{2n-1})),&\text{for $u_{n,0}\in  L^{0}_{n}U_0$.}\end{array}$$

\item\label{diff-a}  For each  integer $r$, with $2\le r\le d-1$, the $P$-module homomorphism $$f_r:\underbrace{P\otimes K^{r-1}_{n-1}U_0}_{\subseteq F_r}\to F_{r-1}=\left\{\begin{array}{c} P\otimes K^{r-2}_{n-1}U_0\\\oplus\\P\otimes L^{r-2}_{n}U_0\end{array}\right.$$ 
is defined by $f_r({\textstyle\sum_i}\theta_{r-1,0,i}\otimes \nu_{n-1,0,i})$ is equal to
\begin{equation}\label{9-27-15-z}
\begin{cases}-x_1\otimes \sum_i\eta\left(\theta_{r-1,0,i}\otimes [p^{-1}(\nu_{n-1,0,i})](\Phi_{2n-1})\right)
\vspace{5pt}\\
+ \sum_{j=2}^dx_j\t x_j^*(\theta_{r-1,0,i})\otimes \nu_{n-1,0,i}
\\\hline
-x_1\otimes \sum_i\kappa (\theta_{r-1,0,i}\otimes \operatorname{proj}(p^{-1}(\nu_{n-1,0,i}))),
\end{cases}\end{equation}
for $\theta_{r-1,0,i}\in \bigwedge^{r-1}U_0$, $\nu_{n-1,0,i}\in D_{n-1}U_0^*$,  and $\sum_i\theta_{r-1,0,i}\otimes \nu_{n-1,0,i}\in K^{r-1}_{n-1}U_0$.
\item\label{diff-b}  For each  integer $r$, with $2\le r\le d-1$, the $P$-module homomorphism  $$f_r:\underbrace{P\otimes L^{r-1}_{n}U_0}_{\subseteq F_r}\to F_{r-1}=\left\{\begin{array}{c} P\otimes K^{r-2}_{n-1}U_0\\\oplus\\P\otimes L^{r-2}_{n}U_0\end{array}\right.$$  is defined by $f_r({\textstyle\sum_i}\theta_{r-1,0,i}\otimes u_{n,0,i})$ is equal to
\begin{equation}\label{9-27-15-a}
\begin{cases}\phantom{+}x_1\otimes \sum_i\eta \left(\theta_{r-1,0,i}\otimes [p^{-1}(u_{n,0,i}(\Phi_{2n-1}))](\Phi_{2n-1})\right)
\vspace{5pt}\\\hline
+x_1\otimes \sum_i\kappa \left(\theta_{r-1,0,i}\otimes \operatorname{proj}(p^{-1}(u_{n,0,i}(\Phi_{2n-1})))\right)\vspace{5pt}\\
+ \sum_{j=1}^d\sum_i x_j\t x_j^*(\theta_{r-1,0,i})\otimes u_{n,0,i}, 
\end{cases}\end{equation} for $\theta_{r-1,0,i}\in \bigwedge^{r-1}U_0$, $u_{n,0,i}\in \operatorname{Sym}_{n}U_0$, and $\sum_i\theta_{r-1,0,i}\otimes u_{n,0,i}\in L^{r-1}_{n}U_0$.

\item The $P$-module homomorphism  $f_d:F_d\to F_{d-1}$  is defined by 
$$f_d: P\otimes {\textstyle\bigwedge^{d-1}}U_0=F_d\longrightarrow F_{d-1}=\left\{\begin{array}{c} P\otimes K^{d-2}_{n-1}U_0\\\oplus\\ P\otimes L^{d-2}_{n}U_0,\end{array}\right.$$
with
$$F_d(\omega_{d-1,0})=\begin{cases} 
\phantom{+}\sum\limits_{m\in\binom{x_2,\dots,x_d}{n}} [ m-x_1p^{-1}(m(\Phi_{2n-1}))]\otimes \eta(\omega_{d-1,0}\otimes m^*)\vspace{5pt}\\\hline -\sum\limits_{m\in\binom{x_2,\dots,x_d}{n-1}} x_1p^{-1}(m^*)\otimes \kappa(\omega_{d-1,0}\otimes m),
\end{cases}$$ for $\omega_{d-1,0}\in \bigwedge^{d-1} U_0$.
\end{enumerate}
\end{theorem}

\begin{remark} 
\label{7.3}
One key feature of the differential given in Theorem~\ref{MTEKK3} is that the expressions
\begin{align*}
\sum_i\eta\left(\theta_{r-1,0,i}\otimes [p^{-1}(\nu_{n-1,0,i})](\Phi_{2n-1})\right)\quad\text{and}\\ 
\sum_i\eta \left(\theta_{r-1,0,i}\otimes [p^{-1}(u_{n,0,i}(\Phi_{2n-1}))](\Phi_{2n-1})\right)\end{align*}
from
(\ref{9-27-15-z}) and (\ref{9-27-15-a}) are automatically in $K^{r-1}_{n-1} U_0$, no projection is necessary. Of course, these expressions are automatically in $K^{r-1}_{n-1} U$. One could apply a projection map to look at the component in  $K^{r-1}_{n-1} U_0$. However this projection map would be ``in the way'' as one verifies that $(F,f)$ is a complex. It is fortunate that no projection map is needed. 
To show that 
$$\ts\sum_i\eta\left(\theta_{r-1,0,i}\otimes [p^{-1}(\nu_{n-1,0,i})](\Phi_{2n-1})\right)$$
is in $K^{r-1}_{n-1} U_0$ it suffices to show that
\begin{equation}\ts(1\t x_1)\left(\sum_i\eta\left(\theta_{r-1,0,i}\otimes [p^{-1}(\nu_{n-1,0,i})](\Phi_{2n-1})\right)\right)=0.\label{7.3.1}\end{equation} 
 We prove (\ref{7.3.1}) 
 mainly to give the flavor of how the differential of $(F,f)$ works.
Observe that the left side of (\ref{7.3.1}) is equal to

\begin{align*}
&\ts\sum_i\eta\left(\theta_{r-1,0,i}\otimes x_1[p^{-1}(\nu_{n-1,0,i})](\Phi_{2n-1})\right)\\
{}={}&\ts\sum_i\eta\left(\theta_{r-1,0,i}\otimes [p^{-1}(\nu_{n-1,0,i})](x_1\Phi_{2n-1})\right)\\
{}={}&\ts\sum_i\eta(\theta_{r-1,0,i}\otimes p[p^{-1}(\nu_{n-1,0,i})])
=\sum_i\eta(\theta_{r-1,0,i}\t \nu_{n-1,0,i})=0.
\end{align*}
The first equality holds because the action of $1\t x_1$ on $\ts\bw^\bullet U\t D_\bullet U^*$ commutes with the Eagon-Northcott map $\eta$. The second equality holds because $D_\bullet U^*$ is a module over $\Sym_\bullet U$. The third equality uses the definition of $p$, which is given in Data~\ref{Opening-Data}.
The conclusion follows from the hypothesis that $\sum_i \theta_{r-1,0,i}\t \nu_{n-1,0,i}$ is in $K^{r-1}_{n-1}U_0$.
\end{remark}

\begin{remark}
 We also demonstrate that the composition $$P\t L_n^1U_0\xrightarrow{f_2}F_1
\xrightarrow{f_1}P$$ from the resolution $(F,f)$ of Theorem~{\rm \ref{MTEKK3}}
is zero. Our motivation for recording this argument here is the same as our motivation for recording Remark~{\rm\ref{7.3}}. We want to convey some of the tricks involved in studying the resolution $(F,f)$. 

Take  $\theta_{1,0,i}\in \bigwedge^{1}U_0$ and $u_{n,0,i}\in \Sym_{n}U_0$, with $\sum_i\theta_{1,0,i}\otimes u_{n,0,i}\in L^{1}_{n}U_0$.
Observe that $(f_1\circ f_2)(\sum_i\theta_{1,0,i}\otimes u_{n,0,i})$ is $S_1+S_2+S_3+S_4+S_5$ with
\begin{align*} S_1={}&\ts x_1^2\cdot  p^{-1}\Big(\sum_i\theta_{1,0,i}\Big[ [p^{-1}(u_{n,0,i}(\Phi_{2n-1})](\Phi_{2n-1})\Big]\Big),
\\
S_2={}&\ts x_1\cdot \Big[\sum_i \Big(\theta_{1,0,i}\cdot \proj\big(p^{-1}[u_{n,0,i}(\Phi_{2n-1})]\big)\Big)\Big],\\
S_3={}&\ts -x_1^2\cdot 
p^{-1}\Big(\Big[\sum_i \Big(\theta_{1,0,i}\cdot \proj\big(p^{-1}[u_{n,0,i}(\Phi_{2n-1})]\big)\Big)\Big](\Phi_{2n-1})\Big),
\\
S_4={}&\ts-\sum_i x_1\theta_{1,0,i}\cdot 
p^{-1}[(u_{n,0,i})(\Phi_{2n-1})],\text{ and}\\
S_5={}&\ts\sum_i \theta_{1,0,i}\cdot (u_{n,0,i})=0.
\\
\end{align*}Observe further that
\begin{align*}
S_1+S_3={}&\ts x_1^2\cdot 
p^{-1}\Big(\Big[\sum_i \Big(\theta_{1,0,i}\cdot (1-\proj)\big(p^{-1}[u_{n,0,i}(\Phi_{2n-1})]\big)\Big)\Big](\Phi_{2n-1})\Big)
\\
S_2+S_4={}&\ts -x_1\cdot \Big[\sum_i \Big(\theta_{1,0,i}\cdot (1-\proj)\big(p^{-1}[u_{n,0,i}(\Phi_{2n-1})]\big)\Big)\Big].
\end{align*}
Apply Observation~\ref{8K} to $$\ts u_n=\Big[\sum_i \Big(\theta_{1,0,i}\cdot (1-\proj)\big(p^{-1}[u_{n,0,i}(\Phi_{2n-1})]\big)\Big)\Big]$$
to see that 
$$S_1+S_3=x_1\cdot \ts 
\Big[\sum_i \Big(\theta_{1,0,i}\cdot (1-\proj)\big(p^{-1}[u_{n,0,i}(\Phi_{2n-1})]\big)\Big)\Big]=-(S_2+S_4).
$$ Thus, the restriction of $f_1\circ f_2$ to $L_n^1U_0$ is identically zero. 
\end{remark}

\begin{observation}
\label{8K}
Adopt Data~{\rm\ref{Opening-Data}}. If $u_n\in x_1\Sym_{n-1}U$, then $x_1p^{-1}(u_n(\Phi_{2n-1}))$ is equal to $u_n$. \end{observation}
\begin{proof}
Let $u_n=x_1u_{n-1}$. Observe that
$$x_1p^{-1}(u_n(\Phi_{2n-1}))=x_1p^{-1}(x_1u_{n-1}(\Phi_{2n-1}))
=x_1p^{-1}p(u_{n-1})=x_1u_{n-1}=u_n
.$$
\end{proof}

Proposition~7.10 of \cite{EKK3} describes the self-duality of $(F,f)$. We will not describe this self-duality in the situation of Theorem~\ref{MTEKK3}; however, we exploit the duality in Section~\ref{d=4,n=2} where we record $(F,f)$ when $d=4$ and $n=2$. The duality means that we need only to record $f_1$ and $f_2$ because $f_3$ and $f_4$ are obtained from $f_1$ and $f_2$ by transposing and rearranging.

\section{The structure of Gorenstein-linear  resolutions, when $d=4$ and $n=2$.}\label{d=4,n=2}
We record the resolution $(F,f)$ of Theorem~\ref{MTEKK3} when $d=4$ and $n=2$. We take advantage of the duality of \cite[Prop.~7.10]{EKK3}. We use names for the modules of $F$ which are less cumbersome than the Schur and Weyl modules $L^a_bU_0$ and $K^a_bU_0$. The resolution is written using maps in Theorem~\ref{resol} and in terms of matrices in Example~\ref{record}.
\begin{observation} 
\label{cleaner bases}
Adopt Data~{\rm\ref{2.1}}. Then 
\begin{align*}K^{0}_{1}U_0={}&U_0^*,&L^0_{2}U_0={}&\Sym_2 U_0,&
K^{1}_1U_0={}&\ker (U_0\t U_0^*\xrightarrow{\ev}\kk),\\
L^{1}_{2}U_0\cong {}& \frac {U_0\t U_0^*}{\ev^*(1)},&
K^2_1U_0\cong {}&D_2U_0^*,\text{ and}&
L_2^2U_0\cong {}&U_0.\end{align*}
\end{observation}
\begin{proof} Observe that 
\begin{align*}
K^{0}_{1}U_0={}&\ts\ker (\bw^0 U_0\t D_1U_0^*\to 0)=U_0^*,\\
L^0_2U_0={}&\ts\ker(\bw^0\t \Sym_2U_0\to 0)=\Sym_2U_0,\text{ and}\\
K^{1}_1U_0={}&\ts\ker (\bw^1U_0\t D_1U_0^*\xrightarrow{\ev}\kk).\\\end{align*}
The Koszul complex 
$$\ts 0\to \bw^3U_0\to \bw^2U_0\t U_0\to \bw^1U_0\t \Sym_2U_0\to \bw^0U_0\t \Sym_3U_0\to 0$$ is split exact;
so,
$$\ts L_2^1U_0=\ker(\bw^1U_0\t \Sym_2U_0\to \bw^0U_0\t \Sym_3U_0)\cong \frac{\bw^2U_0\t U_0}{\im \bw^3U_0}.$$ It follows that
\begin{equation}\label{map1}\map_1:\frac{U_0\t U_0^*}{\ev^*(1)}\to L_2^1U_0,\end{equation} which is induced by
$$\map_1(u_{1,0}\t \nu_{1,0})=-\kappa (\nu_{1,0}(\omega_{U_0})\t u_{1,0}),$$
for $u_{1,0}\in U_0$ and $\nu_{1,0}\in U_0^*$, is an isomorphism.
The Eagon-Northcott complex 
$$\ts 0\to \bw^3U_0\t D_2U_0^*\xrightarrow{\eta} 
\bw^2U_0\t D_1U_0^*\xrightarrow{\eta}
\bw^1U_0\t D_0U_0^*\xrightarrow{\eta} 0$$ is also split exact.
It follows that 
 \begin{equation}\label{map2}\map_2:D_2U_0^*\to K^2_1U_0,\end{equation} which is given by
$$\map_2(\nu_{2,0})=\eta (\omega_{U_0}\t \nu_{2,0}),$$
for $\nu_{2,0}\in D_2U_0^*$, is an isomorphism. In a similar manner, 
  \begin{equation}\label{map3}\map_3:U_0\to L^2_2U_0,\end{equation} which is given by
$$\map_3(u_{1,0})=\kappa (\omega_{U_0}\t u_{1,0}),$$
for $u_{1,0}\in U_0$, is an isomorphism.
\end{proof}

We use the isomorphisms of Observation~\ref{cleaner bases} and the duality of \cite[Prop.~7.3]{EKK3} to record the resolution $(F,f)$ of Theorem~\ref{MTEKK3} in the following form when $d=4$ and $n=2$.
\begin{theorem} 
\label{resol}
Adopt Data~{\rm\ref{2.1}}. The minimal homogeneous resolution of $P/\ann x\Phi_3$ by free $P$ modules is given by
$$(\frak F, \frak f):\quad 0\to \frak F_4\xrightarrow{\frak f_4}
\frak F_3
\xrightarrow{\frak f_3}\frak F_2
\xrightarrow{\frak f_2}\frak F_1
\xrightarrow{\frak f_1}\frak F_0,$$
where 
\begin{align*}\frak F_4={}& P,&\frak F_3={}&(P\t U_0)\p (P\t D_2U_0^*),\\
\frak F_2={}&\begin{cases}P\t \ker(\ev:U_0\t U_0^*\to \kk)\\ \hskip.5in\p\\P\t \frac{U_0\t U_0^*}{\ev^*(1)},\end{cases}&
\frak F_1={}&(P\t U_0^*)\p (P\t \Sym_2 U_0),\\
\frak F_0={}&P,\\
\frak f_4={}&\bmatrix \mf_{1,1}^*\\\mf_{1,2}^*\endbmatrix,&
\frak f_3={}&\bmatrix x\t\mB^*&\mD ^*\\\mA ^*&x\t\mC ^*\endbmatrix,\\
\frak f_2={}&\bmatrix \mA &x\t\mB \\x\t\mC &\mD \endbmatrix,\text{ and}&
\frak f_1={}&\bmatrix \mf_{1,1}&\mf_{1,2}\endbmatrix.\end{align*}
In the following discussion, $u_{i,0}$ is an element of $\Sym_iU_0$ and $\nu_{i,0}$ is an element of $D_iU_0^*$.
\begin{itemize}
\item The map $\mf_{1,1}:U_0^* \to P$ sends $\nu_{1,0}
$ to $xp^{-1}(\nu_{1,0})
$.
\item The map $\mf_{1,2}:\Sym_2U_0\to P$ sends $u_{2,0}$ 
 to 
$u_{2,0}-xp^{-1}(u_{2,0}\Phi_3)$. 
\item The map $\mA : \ker(U_0\t U_0^*\to \kk)\to P\t U_0^*$ is induced by the map
$$U_0\t U_0^*\to P\t U_0^*$$
 which sends $$u_{1,0}\t \nu_{1,0}\mapsto -x\t [u_{1,0}p^{-1}(v_{1,0})](\Phi_3)+u_{1,0}\t \nu_{1,0}.$$
\item   
The map $\mB:\frac{U_0\t U_0^*}{\ev^*(1)}\to U_0^*$ sends the class of $u_{1,0}\t \nu_{1,0}$ to 
$$\begin{cases}
- \big[y p^{-1}\big((\nu_{1,0}(z\w w)u_{1,0})(\Phi_3)\big)\big](\Phi_3)
+\big[z p^{-1}\big((\nu_{1,0}(y\w w)u_{1,0})(\Phi_3)\big)\big](\Phi_3)\\
-\big[w p^{-1}\big((\nu_{1,0}(y\w z)u_{1,0})(\Phi_3)\big)\big](\Phi_3).\\
\end{cases}$$

\item The map $\mC : \ker(U_0\t U_0^*\to \kk)\to \Sym_2U_0$ is induced by the map
$$U_0\t U_0^*\to \Sym_2U_0$$
 which sends $u_{1,0}\t \nu_{1,0}$ to $-u_{1,0}\cdot (\proj\circ p^{-1})(\nu_{1,0})$.

\item The map
$\mD:\frac{U_0\t U_0^*}{\ev^*(1)}\to P\t \Sym_2U_0$
 sends the class of $u_{1,0}\t \nu_{1,0}$ 
 to
\begin{equation}
\label{mD}
\begin{cases}
-x\t w\cdot (\proj\circ p^{-1})\Big([\nu_{1,0}(y\w z)u_{1,0}](\Phi_3)\Big)\\
+x\t z\cdot (\proj\circ p^{-1})\Big([\nu_{1,0}(y\w w)u_{1,0}](\Phi_3)\Big)\\
-x\t y\cdot (\proj\circ p^{-1})\Big([\nu_{1,0}(z\w w)u_{1,0}](\Phi_3)\Big)\\
-w\t \nu_{1,0}(y\w z)u_{1,0}+z\t \nu_{1,0}(y\w w)u_{1,0}
-y\t \nu_{1,0}(z\w w)u_{1,0}.
\end{cases}
\end{equation}
\end{itemize}
\end{theorem}
\begin{remarks} 
\begin{enumerate}[\rm(a)]
\item The entries in the matrices for $\mf_{1,1}$ and $\mf_{1,2}$ are homogeneous of degree two, for $\mA $ and $\mD$ the entries are homogeneous of degree one, and for $\mB $ and $\mC $ the entries are homogeneous of degree zero.
\item Theorem~\ref{resol} is implemented in the Macaulay2 script ``BasicSocDeg3''; see Section~\ref{14.A.1}. \end{enumerate}\end{remarks} 

\begin{proof}
The maps $\mf_{1,1}$, $\mf_{1,2}$, $\mA $, and $\mC $ are obviously correct.
The map $\mD$ is the composition
$$P\t \frac{U_0\t U_0^*}{\ev^*(1)}\xrightarrow{\map_1} P\t L^1_2U_0\xrightarrow{\map_4} P\t\Sym_2U_0,$$
where $\map_1$ is given in (\ref{map1}) and $\map_4$ 
 is
the composition
$$P\t L^1_2U_0\xrightarrow{\text{inclusion}}F_2\xrightarrow{f_2}F_1\xrightarrow{\text{projection}} P\t L_2^0U_0=P\t \Sym_2U_0.$$
If $u_{1,0}\in U_0$ and $\nu_{1,0}\in U_0^*$, then $\map_1$ of the class of $u_{1,0}\t \nu_{1,0}$ in $\frac{U_0\t U_0^*}{\ev^*(1)}$ is equal to
\begingroup\allowdisplaybreaks
\begin{align}
\notag{}={}& 
-\kappa
(\nu_{1,0}(y\w z\w w)\t u_{1,0})\\
\notag{}={}& 
\kappa
\big(
 -\nu_{1,0}(y)\cdot z\w w\t u_{1,0}
+\nu_{1,0}(z)\cdot y\w w\t u_{1,0}
-\nu_{1,0}(w)\cdot y\w z\t u_{1,0}
\big)\\
\notag{}={}&
\begin{cases}
- \nu_{1,0}(y)\cdot w\t zu_{1,0}+\nu_{1,0}(z)\cdot w\t yu_{1,0}\\
+\nu_{1,0}(y)\cdot z\t wu_{1,0}-\nu_{1,0}(w)\cdot z\t yu_{1,0}\\
-\nu_{1,0}(z)\cdot y\t wu_{1,0}+\nu_{1,0}(w)\cdot y\t zu_{1,0}\\\end{cases}\\
\label{8.3.1}{}={}&
- w\t \nu_{1,0}(y\w z)u_{1,0} 
+z\t \nu_{1,0}(y\w w)u_{1,0} 
-y\t \nu_{1,0}(z\w w)u_{1,0}.
\end{align}
Hence, $\mD$ of   the class of $u_{1,0}\t \nu_{1,0}$ is
\begin{align*}
{}={}&\begin{cases}
-x\t w\cdot (\proj\circ p^{-1})\Big([\nu_{1,0}(y\w z)u_{1,0}](\Phi_3)\Big)\\
+x\t z\cdot (\proj\circ p^{-1})\Big([\nu_{1,0}(y\w w)u_{1,0}](\Phi_3)\Big)\\
-x\t y\cdot (\proj\circ p^{-1})\Big([\nu_{1,0}(z\w w)u_{1,0}](\Phi_3)\Big)\\
-w\t \nu_{1,0}(y\w z)u_{1,0}
+z\t \nu_{1,0}(y\w w)u_{1,0}
-y\t \nu_{1,0}(z\w w)u_{1,0}.
\end{cases}
\end{align*} 
\endgroup

The map $\mB$ is the composition
$$\frac{U_0\t U_0^*}{\ev^*(1)}\xrightarrow{\map_1}  L^1_2U_0\xrightarrow{\map_5} U_0^*,$$
where  
 \begin{equation}\label{map5}
\map_5({\textstyle\sum_i}u_{1,0,i}\otimes u_{2,0,i})
=
\sum_i \big[u_{1,0,i} [p^{-1}(u_{2,0,i}(\Phi_{3}))]\big](\Phi_{3})\end{equation}
for $u_{1,0,i}\in \bigwedge^{1}U_0$ and  $u_{2,0,i}\in \operatorname{Sym}_{2}U_0$ with $\sum_iu_{1,0,i}\otimes u_{2,0,i}\in L^{1}_{2}U_0$.

If $u_{1,0}\in U_0$ and $\nu_{1,0}\in U_0^*$, then apply (\ref{8.3.1}) to see that $\mB$ sends the class of $u_{1,0}\t \nu_{1,0}$ in $\frac{U_0\t U_0^*}{\ev^*(1)}$  to
\begingroup\allowdisplaybreaks
\begin{align*}
&\map_5\begin{cases}
-\nu_{1,0}(y)\cdot w\t zu_{1,0}+\nu_{1,0}(z)\cdot w\t yu_{1,0}
+\nu_{1,0}(y)\cdot z\t wu_{1,0}\\-\nu_{1,0}(w)\cdot z\t yu_{1,0}
-\nu_{1,0}(z)\cdot y\t wu_{1,0}+\nu_{1,0}(w)\cdot y\t zu_{1,0}\\\end{cases}\\
{}={}&\begin{cases}
-\nu_{1,0}(y)\cdot \big[w[p^{-1}(zu_{1,0})](\Phi_3)\big](\Phi_3)+\nu_{1,0}(z)\cdot  \big[w[p^{-1}(yu_{1,0})](\Phi_3)\big](\Phi_3)\\
+\nu_{1,0}(y)\cdot  \big[z[p^{-1}(wu_{1,0})](\Phi_3)\big](\Phi_3)-\nu_{1,0}(w)\cdot  \big[z[p^{-1}(yu_{1,0})](\Phi_3)\big](\Phi_3)\\
-\nu_{1,0}(z)\cdot  \big[y[p^{-1}(wu_{1,0})](\Phi_3)\big](\Phi_3)+\nu_{1,0}(w)\cdot \big[y[p^{-1}(zu_{1,0})](\Phi_3)\big](\Phi_3)\\\end{cases}\\
{}={}&\begin{cases}
 -\big[y p^{-1}\big((\nu_{1,0}(z\w w)u_{1,0})(\Phi_3)\big)\big](\Phi_3)\\
+\big[z p^{-1}\big((\nu_{1,0}(y\w w)u_{1,0})(\Phi_3)\big)\big](\Phi_3)\\
-\big[w p^{-1}\big((\nu_{1,0}(y\w z)u_{1,0})(\Phi_3)\big)\big](\Phi_3).\\
\end{cases}
\end{align*} 
\endgroup
When $d=4$ and $n=2$, then the self-duality of $(F,f)$ as described in \cite[Prop. 7.10]{EKK3} is a consequence of the natural perfect pairings
\begin{align*}
&K_1^0U_0\t L_2^2U_0\to \kk,\\
&K_1^1U_0\t L_2^1U_0\to \kk,\text{ and}\\
&K_1^2U_0\t L_2^0U_0\to \kk.\end{align*}
These perfect pairings are compatible with the isomorphisms of Observation~\ref{cleaner bases} and become the natural perfect pairings
\begin{align*}
&U_0^*\t U_0\to \kk,\\
&\ker(U_0\t U_0^*\to \kk) \t \frac{U_0\t U_0^*}{\ev^*(1)}\to \kk,\text{ and}\\
&D_2U_0^*\t \Sym_2U_0\to \kk.\end{align*}
When we record matrices for the resolution $(\frak F,\frak f)$ in Example~\ref{record}, we use the bases \begin{equation}\begin{array}{|l|l|}\hline
\ker (U_0\t U_0^*\to \kk) & \frac{U_0\t U_0^*}{\ev^*(1)}\\\hline\hline
y\t y^*-w\t w^*& y\t y^*\\\hline
y\t z^*& z\t y^*\\\hline
y\t w^*& w\t y^*\\\hline
z\t y^*& y\t z^*\\\hline
z\t z^*-w\t w^*& z\t z^* \\\hline
z\t w^*&w\t z^* \\\hline
w\t y^*&y\t w^* \\\hline
w\t z^*&z\t w^* \\\hline
\end{array}
\hskip.5in \text{and}\hskip.5in \begin{array}{|l|l|}\hline
\frak F_1&\frak F_3\\\hline\hline
y^*&y\\\hline
z^*&z\\\hline
w^*&w\\\hline
y^2&y^{*(2)}\\\hline
yz&y^*z^*\\\hline
yw&y^*w^*\\\hline
z^2&z^{*(2)}\\\hline
zw&z^*w^*\\\hline
w^2&w^{*(2)}\\\hline
\end{array}\, .\label{bases}\end{equation}Each basis vector is listed next to its dual. The matrix for $\mf_{1,1}$ is denoted $f_{1,1}$; for $\mf_{1,2}$ is denoted $f_{1,2}$; for $\mA$ is denoted $A$; for $\mB$ is denoted $B$; for $\mC$ is denoted $C$; and for $\mD$ is denoted $D$. 
The matrices for $\mA ^*$, $\mB ^*$, $\mC ^*$, $\mD ^*$, $f_{1,1}^*$, and $f_{1,2}^*$ automatically become $A\transpose$, $B\transpose$, $C\transpose$, $D\transpose$, $f_{1,1}\transpose$, and $f_{1,2}\transpose$, respectively.
\end{proof}

\begin{remark}
\label{5.4}
 An alternate description of the map $\mB :\frac{U_0\t U_0^*}{\ev^*(1)}\to U_0^*$ is available. If 
$u_{1,0}'$,  $u_{1,0}''$ and $u_{1,0}'''$
are elements of $U_0$, then 
$\mB $ of the class of $u_{1,0}'\t (u_{1,0}''\w u_{1,0}''')(\omega_{U_0^*})$
in $\frac{U_0\t U_0^*}{\ev^*(1)}$ is the element of $U_0^*$ with
$$\big[\mB \big(u_{1,0}'\t (u_{1,0}''\w u_{1,0}''')(\omega_{U_0^*})\big)
\big](u_{1,0})=
\la u_{1,0}u_{1,0}'',u_{1,0}'u_{1,0}'''\ra-\la u_{1,0}u_{1,0}''',u_{1,0}'u_{1,0}''\ra,$$for all $u_{1,0}\in U_0$. The map $\la-,-\ra$ is defined in Data~\ref{2.1}.
\begin{proof} We compute 
\begin{align*}
&\big[\mB \big(u_{1,0}'\t (u_{1,0}''\w u_{1,0}''')(\omega_{U_0^*})\big)
\big](u_{1,0})\\
{}={}&\phantom{{}-{}}\big[(\map_5\circ \map_1)\big(u_{1,0}'\t (u_{1,0}''\w u_{1,0}''')(\omega_{U_0^*})\big)\big](u_{1,0})\\
{}={}&-\big[(\map_5\circ \kappa)\big((u_{1,0}''\w u_{1,0}''')\t u_{1,0}' \big)\big](u_{1,0})&\text{by (\ref{map1})}\\
{}={}&-\big[\map_5
\big( u_{1,0}'''\t u_{1,0}''u_{1,0}' 
-u_{1,0}''\t u_{1,0}'''u_{1,0}' \big)
\big](u_{1,0})\\
{}={}&
-\big( u_{1,0}'''[p^{-1}(u_{1,0}''u_{1,0}'\Phi_3)](\Phi_3) 
-u_{1,0}''[p^{-1}(u_{1,0}'''u_{1,0}'\Phi_3)](\Phi_3) \big)
(u_{1,0})&\text{by (\ref{map5})}\\
{}={}&
- u_{1,0}u_{1,0}'''[p^{-1}(u_{1,0}''u_{1,0}'\Phi_3)](\Phi_3) 
+u_{1,0}u_{1,0}''[p^{-1}(u_{1,0}'''u_{1,0}'\Phi_3)](\Phi_3) \\
{}={}&\phantom{{}+{}}\la u_{1,0}u_{1,0}'',u_{1,0}'''u_{1,0}'\ra- \la u_{1,0}u_{1,0}''',u_{1,0}''u_{1,0}'\ra.\end{align*}
\end{proof}
\end{remark}

\begin{remark}
\label{4.4+} The map $\mB ^*:U_0\to \ker(U_0\t U_0^*\to \kk)$ is given by
$$\mB ^*(u_{1,0}')=\begin{cases}
\phantom{+}w\t [zp^{-1}(yu_{1,0}')(\Phi_3)](\Phi_3)\\
-z\t [wp^{-1}(yu_{1,0}')(\Phi_3)](\Phi_3)\\
-w\t [yp^{-1}(zu_{1,0}')(\Phi_3)](\Phi_3)\\
+y\t [wp^{-1}(zu_{1,0}')(\Phi_3)](\Phi_3)\\
+z\t [yp^{-1}(wu_{1,0}')(\Phi_3)](\Phi_3)\\
-y\t [zp^{-1}(wu_{1,0}')(\Phi_3)](\Phi_3),\\
\end{cases}
$$for all $u_{1,0}'$ in $U_0$.\end{remark}
\begin{proof}
A routine calculation shows that the formula given above satisfies
$$[\mB (u_{1,0}\t\nu_{1,0})](u_{1,0}')=(\nu_{1,0}\t u_{1,0})(\mB^* (u_{1,0}'))$$
for all $u_{1,0}$ and  $u_{1,0}'$ in $U_0$ and all $\nu_{1,0}\in U_0^*$.
\end{proof}

\begin{example}
\label{record}We record the matrices of the resolution $(\frak F,\frak f)$ from Theorem~\ref{resol} explicitly. For each matrix, we list the basis of the domain as the top row and the the target as the left most column. 
The bases are listed in \ref{bases}. The matrix for $X^*$ is the transpose of $X$ for each matrix $X$. 
 
For example, the first column of $B$ should be read to say:
$$\mB (y\t y^*)=0 y^*+(\la yw,z^2\ra -\la yz,zw\ra) z^*+(\la yw,zw\ra-\la yz,w^2\ra)w^*.$$
\begin{landscape}
{\scriptsize
$$f_{1,1}=\begin{array}{|l||l|l|l|l|}\hline
&y^*&z^*&w^*\\\hline\hline
1&xp^{-1}(y^*)&xp^{-1}(z^*)&xp^{-1}(w^*)\\\hline\end{array}\, ,$$

$$f_{1,2}=\begin{array}{|l||l|l|l|l|l|l|}\hline
&y^2&yz&yw&z^2&zw&w^2
\\\hline\hline
1&y^2-xp^{-1}(y^2\Phi_3)&yz-xp^{-1}(yz\Phi_3)
&yw-xp^{-1}(yw\Phi_3)&z^2-xp^{-1}(z^2\Phi_3)&zw-xp^{-1}(zw\Phi_3)&w^2-xp^{-1}(w^2\Phi_3)\\\hline\end{array}\, ,$$}

{\tiny 
$$\hskip-70pt A=\begin{array}{|l||l|l|l|l|l|l|l|l|}
\hline
&y\t y^*&y\t z^*&y\t w^*&z\t y^*& z\t z^*&z\t w^*&w\t y^*& w\t z^*
\\
&-w\t w^*&&&&-w\t w^*&&&
\\\hline
\hline
y^*&-[y^2p^{-1}(y^*)](\Phi_3)\cdot x&-[y^2p^{-1}(z^*)](\Phi_3)\cdot x&-[y^2p^{-1}(w^*)](\Phi_3)\cdot x&-[yzp^{-1}(y^*)](\Phi_3)\cdot x&
-[yzp^{-1}(z^*)](\Phi_3)\cdot x&
-[yzp^{-1}(w^*)](\Phi_3)\cdot x&
-[ywp^{-1}(y^*)](\Phi_3)\cdot x& 
-[ywp^{-1}(z^*)](\Phi_3)\cdot x
 \\
&+[ywp^{-1}(w^*)](\Phi_3)\cdot x&&&+z&+[ywp^{-1}(w^*)](\Phi_3)\cdot x
&&+w&\\
&+y&&&&&&&\\
\hline
z^*&-[yzp^{-1}(y^*)](\Phi_3)\cdot x&-[yzp^{-1}(z^*)](\Phi_3)\cdot x&-[yzp^{-1}(w^*)](\Phi_3)\cdot x&-[z^2p^{-1}(y^*)](\Phi_3)\cdot x&-[z^2p^{-1}(z^*)](\Phi_3)\cdot x&-[z^2p^{-1}(w^*)](\Phi_3)\cdot x&-[zwp^{-1}(y^*)](\Phi_3)\cdot x&-[zwp^{-1}(z^*)](\Phi_3)\cdot x
\\
&+[wzp^{-1}(w^*)](\Phi_3)\cdot x&+y&&&+[wzp^{-1}(w^*)](\Phi_3)\cdot x&&&+w
\\
&&&&&+z&&&
\\\hline

w^*&-[ywp^{-1}(y^*)](\Phi_3)\cdot x&-[ywp^{-1}(z^*)](\Phi_3)\cdot x&-[ywp^{-1}(w^*)](\Phi_3)\cdot x&-[zwp^{-1}(y^*)](\Phi_3)\cdot x&-[zwp^{-1}(z^*)](\Phi_3)\cdot x&-[zwp^{-1}(w^*)](\Phi_3)\cdot x&-[w^2p^{-1}(y^*)](\Phi_3)\cdot x&-[w^2p^{-1}(z^*)](\Phi_3)\cdot x
\\
&+[w^2p^{-1}(w^*)](\Phi_3)\cdot x&&+y&&+[w^2p^{-1}(w^*)](\Phi_3)\cdot x&+z&&
\\
&-w&&&&-w&&&
\\
\hline
\end{array}\,,$$
}

$$B=\begin{array}{|l||l|l|l|l|l|l|l|l|}\hline
&y\t y^*&z\t y^*&w\t y^*&y\t z^*&z\t z^*&w\t z^*&y\t w^*&z\t w^*\\\hline\hline
y^*&0&  \la yz,zw\ra&\la yz,w^2\ra&0&\la yw,yz\ra&\la yw,yw\ra&0&\la y^2,z^2\ra\\
   & &-\la yw,z^2\ra&-\la yw,zw\ra& &-\la y^2,zw\ra&-\la y^2,w^2\ra&&-\la yz,yz\ra\\\hline
z^*&\la yw,z^2\ra&0&\la z^2,w^2\ra&\la y^2,zw\ra&0&\la zw, yw\ra&\la yz,yz\ra&0\\
   &-\la yz,zw\ra&&-\la zw,zw\ra&-\la yw,yz\ra&&-\la w^2,yz\ra&-\la y^2,z^2\ra&\\\hline
w^*&\la yw,zw\ra&\la zw,zw\ra&0&\la y^2,w^2\ra&\la yz,w^2\ra&0&\la yz,yw\ra&
\la z^2,yw\ra
\\
 &-\la yz,w^2\ra&-\la z^2,w^2\ra&&-\la yw,yw\ra&-\la zw,yw\ra&&-\la y^2,wz\ra& -\la yz,zw\ra  \\\hline \end{array}\, ,$$

{\scriptsize
$$
C=\begin{array}{|l||l|l|l|l|l|l|l|l|}
\hline
&y\t y^*&y\t z^*&y\t w^*&z\t y^*& z\t z^*&z\t w^*&w\t y^*& w\t z^*
\\
&-w\t w^*&&&&-w\t w^*&&&
\\\hline
\hline
y^2&-[y^*p^{-1}(y^*)]&-[y^*p^{-1}(z^*)]&-[y^*p^{-1}(w^*)]&0&0&0&0&0
\\\hline
yz&-[z^*p^{-1}(y^*)]&-[z^*p^{-1}(z^*)]&-[z^*p^{-1}(w^*)]&-[y^*p^{-1}(y^*)]&-[y^*p^{-1}(z^*)]&-[y^*p^{-1}(w^*)]&0&0
\\\hline
yw&0 
&-[w^*p^{-1}(z^*)]&-[w^*p^{-1}(w^*)]&0&+[y^*p^{-1}(w^*)]&0&-[y^*p^{-1}(y^*)]&-[y^*p^{-1}(z^*)]
\\
\hline
z^2&0&0&0&-[z^*p^{-1}(y^*)]&-[z^*p^{-1}(z^*)]&-[z^*p^{-1}(w^*)]&0&0
\\
\hline
zw&[z^*p^{-1}(w^*)]&0&0&-[w^*p^{-1}(y^*)]&
0&-[w^*p^{-1}(w^*)]&-[z^*p^{-1}(y^*)]&-[z^*p^{-1}(z^*)]
\\
\hline
w^2&+[w^*p^{-1}(w^*)]&0&0&0&+[w^*p^{-1}(w^*)]&0&-[w^*p^{-1}(y^*)]&-[w^*p^{-1}(z^*)]
\\
\hline
\end{array}\,,$$}
and

{\tiny
$$\hskip-70pt
D=
\begin{array}{|l|l|l|l|l|l|l|l|l|}
\hline
&y\t y^*&z\t y^*&w\t y^*&y\t z^*&z\t z^*&w\t z^*&y\t w^*&z\t w^*
\\
\hline\hline
y^2
&0&0&0&-[y^*p^{-1}(wy\Phi_3)]\cdot x&-[y^*p^{-1}(wz\Phi_3)]\cdot x&-[y^*p^{-1}(w^2\Phi_3)]\cdot x&[y^*p^{-1}(yz\Phi_3)]\cdot x&+[y^*p^{-1}(z^2\Phi_3)]\cdot x
\\
&&&&+w&&&-z&\\
\hline
yz
&[y^*p^{-1}(yw\Phi_3)]\cdot x&[y^*p^{-1}(zw\Phi_3)]\cdot x&[y^*p^{-1}(w^2\Phi_3)]\cdot x&-[z^*p^{-1}(yw\Phi_3)]\cdot x&-[z^*p^{-1}(zw\Phi_3)]\cdot x&
-[z^*p^{-1}(w^2\Phi_3)]\cdot x&-[y^*p^{-1}(y^2\Phi_3)]\cdot x&-[y^*p^{-1}(yz\Phi_3)]\cdot x
\\
&-w&&&&+w&&+[z^*p^{-1}(zy\Phi_3)]\cdot x
&+[z^*p^{-1}(z^2\Phi_3)]\cdot x\\
&&&&&&&+y&-z\\\hline
yw
&-[y^*p^{-1}(yz\Phi_3)]\cdot x&-[y^*p^{-1}(z^2\Phi_3)]\cdot x&-[y^*p^{-1}(zw\Phi_3)]\cdot x&+[y^*p^{-1}(y^2\Phi_3)]\cdot x&[y^*p^{-1}(yz\Phi_3)]\cdot x&[y^*p^{-1}(yw\Phi_3)]\cdot x&+[w^*p^{-1}(yz\Phi_3)]\cdot x&
+[w^*p^{-1}(z^2\Phi_3)]\cdot x\\
&+z&&&-[w^*p^{-1}(wy\Phi_3)]\cdot x&-[w^*p^{-1}(wz\Phi_3)]\cdot x
&-[w^*p^{-1}(w^2\Phi_3)]\cdot x&&\\
&&&&-y&&+w&&
\\\hline
z^2
&[z^*p^{-1}(yw\Phi_3)]\cdot x&[z^*p^{-1}(zw\Phi_3)]\cdot x&[z^*p^{-1}(w^2\Phi_3)]\cdot x&0&0&0&-[z^*p^{-1}(y^2\Phi_3)]\cdot x&-[z^*p^{-1}(yz\Phi_3)]\cdot x
\\
&&-w&&&&&&+y
\\\hline
zw
&-[z^*p^{-1}(yz\Phi_3)]\cdot x&-[z^*p^{-1}(z^2\Phi_3)]\cdot x&-[z^*p^{-1}(zw\Phi_3)]\cdot x&+[z^*p^{-1}(y^2\Phi_3)]\cdot x&+[z^*p^{-1}(yz\Phi_3)]\cdot x&+[z^*p^{-1}(yw\Phi_3)]\cdot x&-[w^*p^{-1}(y^2\Phi_3)]\cdot x&-[w^*p^{-1}(yz\Phi_3)]\cdot x
\\
&+[w^*p^{-1}(yw\Phi_3)]\cdot x&+[w^*p^{-1}(zw\Phi_3)]\cdot x&+[w^*p^{-1}(w^2\Phi_3)]\cdot x&&-y&&&
\\
&&+z&-w&&&&&
\\\hline
w^2
&-[w^*p^{-1}(yz\Phi_3)]\cdot x&-[w^*p^{-1}(z^2\Phi_3)]\cdot x
&-[w^*p^{-1}(zw\Phi_3)]\cdot x
&+[w^*p^{-1}(y^2\Phi_3)]\cdot x&[w^*p^{-1}(yz\Phi_3)]\cdot x&[w^*p^{-1}(yw\Phi_3)]\cdot x&0&0
\\
&&&+z&&&-y&&
\\
\hline
\end{array}\,.$$
}
\end{landscape}
\end{example}

\begin{observation}
\label{x=0} Let $(\frak F,\frak f)$ be the resolution of Theorem~{\rm \ref{resol}} and $\Pbar$ be the quotient ring $\frac P{(x)}$. Then   $$\HH_j\left((\frak F,\frak f)\t_P \Pbar\right)=0, \quad \text{for $2\le j$}.$$ Indeed, $(\frak F,\frak f)\t_P \Pbar$  is isomorphic to the mapping cone of the zero map $$\Hom_{\Pbar}(E,\Pbar)\to E,$$ where $E$ is the Eagon-Northcott resolution of $\Pbar/I_2(M)$, and $M$ is the matrix $$M=\bmatrix y&z&w&0\\0&y&z&w\endbmatrix.  $$
\end{observation}
\begin{proof}Let 
$\bar{\phantom{x}}$ be the functor $-\t_P\Pbar$.
One can read Example~\ref{record} (or Theorem~\ref{resol}) to see that $(\frak F,\frak f)\t_P \Pbar$ is the mapping cone of 
{\small$$\xymatrix{\Hom_{\Pbar}(E,\Pbar):\ar[d]^{0}&\quad
0\ar[r]&\Pbar(-6)^1\ar[r]^{\bar f_{12}\transpose}\ar[d]^{0}&\Pbar(-4)^6\ar[r]^{\bar D\transpose}\ar[d]^{0}&\Pbar(-3)^8\ar[r]^{\bar A}\ar[d]^{0}&\Pbar(-2)^3\ar[d]^{0}\\
E&0\ar[r]&\Pbar(-4)^3\ar[r]^{\bar A\transpose}&\Pbar(-3)^8\ar[r]^{\bar D}&\Pbar(-2)^6\ar[r]^{\bar f_{1,2}}&\Pbar,}$$}where
$$\bar A\transpose=\bmatrix
y&0&-w\\0&y&0\\0&0&y\\z&0&0\\0&z&-w\\0&0&z\\w&0&0\\0&w&0\endbmatrix,\quad
\bar D=\bmatrix 0&0&0&w&0&0&-z&0\\-w&0&0&0&w&0&y&-z\\z&0&0&-y&0&w&0&0\\0&-w&0&0&0&0&0&y\\0&z&-w&0&-y&0&0&0\\0&0&z&0&0&-y&0&0\endbmatrix,$$
and 
$$\bar f_{1,2}=\bmatrix y^2&yz&yw&z^2&zw&w^2\endbmatrix.$$
(Actually, the complex $\frak F\t_P\Pbar$ is studied already in \cite{EKK3}, where it is called the ``skeleton'' of $\frak F$.)
It is clear that $E$ is the Eagon-Northcott resolution of the quotient $\Pbar/I_2(M)$ and that $\Hom_{\Pbar}( E,\Pbar)$ is the Eagon-Northcott resolution of 
$$\Ext^3_{\Pbar}(\Pbar/I_2(M),\Pbar),$$ which is the canonical module of $\Pbar/I_2(M)$. It follows that 
$(\frak F,\frak f)\t_P \Pbar$ has nonzero homology in positions $0$ and $1$; but otherwise, the homology is zero.
\end{proof}

\section{A resolution of $P/\ann \Phi_3$ by free $P$-modules.}
A few slight modifications turn the minimal homogeneous resolution of Theorem~\ref{resol} of $P/\ann x\Phi_3$ by free $P$-modules into a resolution of $P/\ann \Phi_3$. 
\begin{theorem} 
\label{a res}
Adopt Data~{\rm\ref{2.1}}. Then the ideal $\ann \Phi_3$ is generated by
$$\{x^2 p^{-1}(\nu_{1,0})\mid \nu_{1,0}\in U_0^*\}\cup\{u_{2,0}-xp^{-1}(u_{2,0}\Phi_3)|u_{2,0}\in \Sym_2U_0\}.$$Furthermore,
$$\left(\frak F, \widetilde {\frak f}\,\,\right):\quad 0\to \frak F_4\xrightarrow{\widetilde{\frak f_4}}\frak F_3
\xrightarrow{\widetilde{\frak f_3}}\frak F_2
\xrightarrow{\widetilde{\frak f_2}}\frak F_1
\xrightarrow{\widetilde{\frak f_1}}\frak F_0$$
is a homogeneous resolution of $P/\ann \Phi_3$ by free $P$-modules where
\begin{align*}
\widetilde{\frak f_4}={}&\bmatrix x\mf_{1,1}^*\\\mf_{1,2}^*\endbmatrix,&
\widetilde{\frak f_3}={}&\bmatrix \mB ^*&\mD ^*\\\mA ^*&x^2\t\mC ^*\endbmatrix,\\
\widetilde{\frak f_2}={}&\bmatrix \mA &\mB \\x^2\t\mC &\mD \endbmatrix,\text{ and}&
\widetilde{\frak f_1}={}&\bmatrix x\mf_{1,1}&\mf_{1,2}\endbmatrix.\end{align*}
The resolution $\left(\frak F, {\widetilde {\frak f}}\,\,\right)$ is minimal if and only if $B=0$. Indeed, the minimal number of generators of $\ann\Phi_3$ is equal to $9-\rank B$.
\end{theorem}

\begin{Remark}
Theorem~\ref{a res} is implemented in the Macaulay2 script ``BasicSocDeg3''; see Section~\ref{14.A.1}. \end{Remark}

\begin{proof} We first show that 
\begin{equation}\label{gens}\la x^2 p^{-1}U_0^*\ra+\la u_{2,0} -x p^{-1}(u_{2,0}\Phi_3)\mid u_{2,0}\in \Sym_2U_0\ra = \ann \Phi_3.\end{equation}
The inclusion $\subseteq$ is obvious. Indeed,
$$xp^{-1}(U_0^*)\subseteq \ann x\Phi_3 \implies x^2p^{-1}(U_0^*)\subseteq \ann \Phi_3$$
and 
 \begin{align*}&[u_{2,0}-xp^{-1}(u_{2,0} \Phi_3)](\Phi_3)=u_{2,0}(\Phi_3)-[p^{-1}(u_{2,0}\Phi_3)](x\Phi_3)\\
{}={}&u_{2,0}(\Phi_3)-p[p^{-1}(u_{2,0}\Phi_3)]=0.\end{align*}
The inclusion $\supseteq$ requires more work.
\begin{claim-no-advance}
\label{13.1}The elements 
$$xp^{-1}(x^*), \quad
xp^{-1}(y^*),\quad
xp^{-1}(z^*),\quad  \text{and}\quad xp^{-1}(w^*)$$
of $(P/\ann \Phi_3)_2$ are linearly independent. The element  $x^2p^{-1}(x^*)$ of $(P/\ann \Phi_3)_3$ is nonzero.\end{claim-no-advance} 
\ms\noindent{\it Proof of Claim~{\rm \ref{13.1}}.} Suppose that  $a,b,c,d$ are constants with  $$axp^{-1}(x^*)+ 
bxp^{-1}(y^*)+cxp^{-1}(z^*)+dxp^{-1}(w^*)=0$$ in $P/\ann \Phi_3$. In this case,
$$
\ell=ap^{-1}(x^*)+ 
bp^{-1}(y^*)+cp^{-1}(z^*)+dp^{-1}(w^*)$$ is a linear form in $P/\ann \Phi_3$ with $x\ell=0$.  
However, $x$ is a weak Lefschetz element on  $P/\ann \Phi_3$; hence, $\ell =0$ in 
$P/\ann \Phi_3$. The $\kk$-algebra $P/\ann \Phi_3$ has embedding dimension four; hence, $\ell$ is zero in $P$. 
The elements $x^*$, $y^*$, $z^*$, $w^*$ of $U^*$ are linearly independent and  $p^{-1}$ is an isomorphism. It follows that  $a$, $b$, $c$, and $d$ are all zero.

For the other assertion, observe that 
$[x^2p^{-1}(x^*)](\Phi_3)=x(x^*)=1$, which is   not equal to $0$; hence, $x^2p^{-1}(x^*)\neq 0$ in $(P/\ann \Phi_3)_3$. This completes the proof of Claim \ref{13.1}. 

We finish the proof of (\ref{gens}) by showing that 

\begin{equation}\label{13.1.1}\ann \Phi_3\subseteq \Big(\la x^2 p^{-1}U_0^*\ra+\la u_{2,0} -x p^{-1}(u_{2,0}\Phi_3)\mid u_{2,0}\in \Sym_2U_0\ra\Big).\end{equation}
The Hilbert series of $\AA=P/\ann \Phi_3$ is $(1,4,4,1)$. Indeed the Hilbert series exhibits symmetry, has $\dim \AA_1=4$ and has a socle of dimension 1 in degree $3$.
Thus, we see from Claim~\ref{13.1}, that\begin{equation}\label{13.1.2}1,\ x,\ y,\ z,\ w, \  xp^{-1}(x^*),\  xp^{-1}(y^*),\  xp^{-1}(z^*),\  xp^{-1}(w^*),\  x^2p^{-1}(x^*)\end{equation} is a basis for $P/\ann \Phi_3$. We  show (\ref{13.1.1}) by showing
that $P_0\p P_1\p P_2\p P_3$ is contained in the vector space spanned by (\ref{13.1.2}) 
and
\begin{equation}\label{13.1.3} \Big(\la x^2 p^{-1}U_0^*\ra+\la u_{2,0} -x p^{-1}(u_{2,0}\Phi_3)\mid u_{2,0}\in \Sym_2U_0\ra\Big).\end{equation}

We see that $P_0$ and $P_1$ are both in the span of (\ref{13.1.2}). Observe that
$P_2$ is spanned by 
\begin{equation}\label{gb}xp^{-1}(x^*),\ 
xp^{-1}(y^*),\ xp^{-1}(z^*), xp^{-1}(w^*)\end{equation}
and 
\begin{equation}\label{gb2}\begin{array}{l}y^2-xp^{-1}(y^2\Phi_3),\ 
yz-xp^{-1}(yz\Phi_3),\ 
yw-xp^{-1}(yw\Phi_3),\\ 
z^2-xp^{-1}(z^2\Phi_3),\ 
zw-xp^{-1}(zw\Phi_3),\ \text{and}\ 
w^2-xp^{-1}(w^2\Phi_3).\end{array}\end{equation}
 The first four are a basis for $xP_1$ (by Claim~\ref{13.1}) and the last $6$ are a basis for $P_2/xP_1$. 
The first four are in (\ref{13.1.2}) and the last six are in (\ref{13.1.3}).

Use the  basis for $P_2$ which is the union (\ref{gb}) and (\ref{gb2})  
 to see  that $P_3$ is spanned by
\begin{enumerate}[\rm(a)]
\item\label{13.1.a} $(x,y,z,w) xp^{-1}(x^*)$,
\item\label{13.1.b} $(x,y,z,w)(xp^{-1}(y^*),\ xp^{-1}(z^*), xp^{-1}(w^*))$,
\item \label{13.1.c} $(x,y,z,w) \la u_{2,0}-xp^{-1}(u_{2,0} \Phi_3)\mid u_{2,0}\in \Sym_2U\ra$.
\end{enumerate}
Observe that the basis elements of (\ref{13.1.c}) are in (\ref{13.1.3}) and the basis elements $$x(xp^{-1}(y^*),\ xp^{-1}(z^*), xp^{-1}(w^*))$$ from (\ref{13.1.b}) are in $x^2p^{-1}U_0^*$; hence these are in (\ref{13.1.3}).
It is clear that $x^2p^{-1}(x^*)$ from (\ref{13.1.a}) is in \ref{13.1.2}. 
Let $\ell$ be an element of $U_0$. Consider $\ell\cdot xp^{-1}(x^*)$ from (\ref{13.1.a}). Observe that
$$[\ell p^{-1}(x^*)](x\Phi_3)=\ell(x^*)=0;$$therefore,
$$\ell p^{-1}(x^*)\in \ann(x\Phi_3)=x p^{-1}U_0^*+\la u_{2,0}-xp^{-1}(u_{2,0} \Phi_3)\mid u_{2,0} \in \Sym_2U_0\ra.$$ It follows that
$$x\ell p^{-1}(x^*)\in x^2 p^{-1}U_0^*+\la u_{2,0}-xp^{-1}(u_{2,0} \Phi_3)\mid u_{2,0} \in \Sym_2U_0\ra, $$which is (\ref{13.1.3}). 

Let $\ell_1$ and $\ell_2$ be different elements selected from $y$, $z$, and $w$.
Consider $\ell_1 x p^{-1}(\ell_2^*)$ from (\ref{13.1.b}).
Observe that 
$$[\ell_1 p^{-1}(\ell_2^*)](x\Phi_3)=\ell_1(\ell_2^*)=0.$$
Therefore, $$\ell_1 p^{-1}(\ell_2^*)\in \ann x \Phi_3= xp^{-1}U_0^*+\la u_{2,0}-xp^{-1}(u_{2,0} \Phi_3)\mid u_{2,0} \in \Sym_2U_0\ra $$ and
$$x\ell_1 p^{-1}(\ell_2^*)\subseteq x^2p^{-1}U_0^*+\la u_{2,0}-xp^{-1}(u_{2,0} \Phi_3)\mid u_{2,0} \in \Sym_2U_0\ra, $$ which is (\ref{13.1.3}). 

Let $\ell$ equal $y$, $z$, or $w$. Consider $\ell x p^{-1}\ell^*$ from (\ref{13.1.b}). Observe that
$$\Big(\ell p^{-1}(\ell^*)-xp^{-1} (x^*)\Big)(x\Phi_3)=\ell(\ell^*)-x(x^*)=1-1=0.$$ Thus, $$\ell p^{-1}(\ell^*)-xp^{-1} (x^*)\in \ann x\Phi_3=xp^{-1}U_0^*+\la u_{2,0}-xp^{-1}(u_{2,0} \Phi_3)\mid u_{2,0} \in \Sym_2U_0\ra. $$ In other words,
$\ell p^{-1}(\ell^*)$ is equal to
$$xp^{-1} (x^*)+ \text{ an element of }xp^{-1}U_0^*+\la u_{2,0}-xp^{-1}(u_{2,0} \Phi_3)\mid u_{2,0} \in \Sym_2U_0\ra $$ and
$x\ell p^{-1}(\ell^*)$ is equal to $$x^2p^{-1} (x^*)+ \text{ an element of }x^2p^{-1}U_0^*+\la u_{2,0}-xp^{-1}(u_{2,0} \Phi_3)\mid u_{2,0} \in \Sym_2U_0\ra. $$
The element $x^2p^{-1} (x^*)$ is in (\ref{13.1.2}) and an element of $$x^2p^{-1}U_0^*+\la u_{2,0}-xp^{-1}(u_{2,0} \Phi_3)\mid u_{2,0} \in \Sym_2U_0\ra$$ is in (\ref{13.1.3}). This completes the proof of (\ref{gens}).

We prove that the complex $(\frak F, \widetilde{\frak f})$ is acyclic by studying the short exact sequence of complexes
$$0\to (\frak F, \widetilde{\frak f})\xrightarrow{\g} (\frak F, \frak f) \to \Hom_P(E,\bar P)[-1] \to 0.$$ (The  ``$[-1]$'' refers to a shift of homological position.) 
The complex $(\frak F, \frak f)$ is introduced in Theorem~\ref{resol} and the complex $\Hom_P(E,\bar P)$
may be found  in Observation~\ref{x=0}. The map $\g_i:(\frak F, \widetilde{\frak f})_i\to (\frak F, \frak f)_i$ is
$$\g_0=1, \quad \g_1=\g_2=\bmatrix x&0\\0&1\endbmatrix,\quad \g_3=\bmatrix 1&0\\0&x\endbmatrix,\quad \g_4=x.$$ 
The complexes  $(\frak F, \frak f)$ and  $\Hom_P(E,\bar P)$ are acyclic; thus,  $(\frak F, \widetilde{\frak f})$ is also acyclic.

The resolution $(\frak F, \widetilde{\frak f})$ is homogeneous with graded Betti numbers 
$$0\to P(-7)^1\to \begin{matrix} P(-4)^3\\\p\\ P(-5)^6\end{matrix}
\to \begin{matrix}P(-4)^8\\\p\\ P(-3)^8\end{matrix}\to \begin{matrix}P(-3)^3\\\p\\ P(-2)^6\end{matrix}\to P.$$
All of the nonzero maps have positive degree, except possibly $$B:P(-3)^8\to P(-3)^3\quad\text{and}\quad B\transpose: P(-4)^3\to P(-4)^8.$$ It follows that 
$(\frak F, \widetilde{\frak f})$ is a minimal resolution if and only if $B$ is zero. Indeed, the minimal number of generators of $\ann\Phi_3$ is equal to $9-\rank B$. \end{proof}

\section{The matrix $\SM$.}\label{section SM}

Retain Data~\ref{2.1}.
In this section, we introduce the $3\times 3$ symmetric matrix $\SM$ which controls the rank of the matrix $B$ from Theorems \ref{resol} and \ref{a res} and therefore controls the precise form of the minimal resolution of $P/\ann \Phi_3$. Furthermore, we introduce the complex $\C$ which relates $\SM$ and $B$.

\subsection{Complexes associated to a symmetric $3\times 3$ matrix.}\label{2.9}

$ $

Contrary to the usual convention \ref{2.1}.(\ref{2.1.a}) in this paper, throughout Subsection~\ref{2.9},
$(-)^*$ represents $\Hom_R(-,R)$.
\begin{background}
\label{2.9.1} Let $R$ be a commutative Noetherian ring, $U_0$ be a finitely generated free $R$-module, and $(-)^*$ represent $\Hom_R(-,R)$.
\begin{enumerate}[\rm(a)]
\item If  $\Gamma: U_0^*\to U_0$ is an $R$-module homomorphism, then $\Gamma$ is {\it symmetric} if $$[\Gamma(w_1)](w_1')=[\Gamma(w_1')](w_1)$$ for  $w_1,w_1'\in U_0^*$.
\item\label{2.9.1.b} If $\omega_{U_0^*}$ is a fixed basis for $\bigwedge^{\rank U_0^*}U_0^*$ and $\Gamma: U_0^*\to U_0$ is an $R$-module homomorphism, then the {\it classical adjoint} of $\Gamma$ is the $R$-module homomorphism  $\Gamma^\vee:U_0\to U_0^*$ with 
$$\Gamma^\vee(u_1)=\left[\ts(\bigwedge^{\rank U_0^*-1}\Gamma)[u_1(\omega_{U_0^*})]\right](\omega_{U_0^*}),$$ for $u_1\in U_0$. (Of course, if $\Gamma$ is invertible, then $\Gamma^\vee$ is equal to a unit times $\Gamma^{-1}$.)  
\item If $\Gamma: U_0^*\to U_0$ is a symmetric homomorphism, then the classical adjoint $$\Gamma^\vee:U_0\to U_0^*$$ of $\Gamma$ is also a symmetric homomorphism.
\item If $\beta: U_0\to U_0^*$ is a symmetric homomorphism, then there is an element 
$\delta $ in $D_2U_0^*$ with $\beta (u_1)= u_1(\delta )$ for all $u_1\in U_0$.
\end{enumerate}
\end{background}

\begin{definition}
\label{24.2}Let $R$ be a commutative Noetherian ring, 
 $U_0$ be a free $R$-module of rank $3$, $(-)^*$ represent $\Hom_R(-,R)$,
   ${\Gamma: U_0^*\to  U_0}$ be a symmetric homomorphism of $R$-modules, $\Gamma^\vee:U_0\to U_0^*$ be the classical adjoint of $\Gamma$,
and $\delta $ be  the element of $D_2U_0^*$ with $\Gamma^\vee (u_1)= u_1(\delta )$ for  $u_1\in U_0$.
Let $\mathfrak C_{\Gamma}$ be the following sequence of $R$-module homomorphisms:
\begin{equation}\notag 
\mathfrak C_{\Gamma}:\quad 0\to R\xrightarrow{\ \ \mathfrak C_3\ \ }D_2U_0^*\xrightarrow{\ \ \mathfrak C_2\ \ }\frac{U_0\t_{R} U_0^*}{(\ev^*(1))}\xrightarrow{\ \ \mathfrak C_1\ \ }\ts \bigwedge^2 U_0\to 0,\end{equation}
\begin{align*}\mathfrak C_3(1)&=\delta \in D_2U_0^*,&&\text{for $1\in R$},\\
\mathfrak C_2(\nu_1^{(2)})&=\Gamma (\nu_1)\t \nu_1,&&\text{for $\nu_1\in U_0^*$, and}\\
\mathfrak C_1(u_1\t \nu_1)&=u_1\w \Gamma(\nu_1),&&\text{for $u_1\in U_0$ and $\nu_1\in U_0^*$.}
\end{align*}
\end{definition}

\begin{lemma}
\label{24.4}Adopt the notation of Definition~{\rm\ref{24.2}}. Then the following statements hold.
\begin{enumerate}[\rm(a)]
\item\label{24.4.a} The maps of $\mathfrak C$ form a complex of free $R$-modules.
\item\label{24.4.b} If $I_2(\Gamma)=R$, then $\mathfrak C$ is split exact.
\item\label{24.4.c} If $I_2(\Gamma)$ is a proper ideal of $R$ of grade at least $3$, then $\mathfrak C$ and its dual $\mathfrak C^*$ are acyclic.
\item\label{24.4.d}When the hypotheses of {\rm (\ref{24.4.c})} are in effect, then $\mathfrak C^*$ is a resolution of $R/I_2(\Gamma)$ by free $R$-modules and $\mathfrak C$ is a resolution of $\Ext^3_R(R/I_2(\Gamma),R)$ by free $R$-modules. 
\item\label{24.4.e} The ideals $I_2(\Gamma)$, $I_3(\C_1)$, $I_5(\C_2)$, and $I_1(\C_3)$ all have the same radical.
\end{enumerate}\end{lemma}
\begin{proof}
See \cite[Thm.~3.1]{J78} and  \cite{K74,Ko91}  for the proof of (\ref{24.4.a}), (\ref{24.4.b}), (\ref{24.4.c}), (\ref{24.4.d}). We give an argument for (\ref{24.4.e}).
Let $\widetilde{R}$ be the polynomial ring $R[\{X_{i,j}|1\le i\le j\le 3\}]$,  $\widetilde{\Gamma}$ be the generic symmetric homomorphism $\widetilde{R}\to \widetilde{R}$ whose matrix is  
$$
\bmatrix X_{1,1}&X_{1,2}&X_{1,3}\\X_{1,2}&X_{2,2}&X_{2,3}\\X_{1,3}&X_{2,3}&X_{3,3}\endbmatrix,$$ and $\widetilde{\frak C}$ be the complex of Definition~\ref{24.2} built using $\widetilde{\C}$ in place of $\C$.  The ideal $I_2(\widetilde{\Gamma})$  is perfect of grade $3$ by \cite{K74}; hence $\widetilde{\mathfrak C}$ is exact by (\ref{24.4.c}); thus, 
$I_3(\widetilde{\C}_1)$, $I_5(\widetilde{\C}_2)$, and $I_1(\widetilde{\C}_3)$ 
all have the same radical by \cite[Props.~6.8 and 6.3.(c)]{Ho75}. 
Let $f: \widetilde{R}\to R$ be the $R$-algebra homomorphism which sends $\widetilde{\Gamma}$ to $\Gamma$. Observe that 
$f(I_i(\widetilde{\C}_j))=I_i(\C_j)$ for all $i$ and $j$.
The proof is complete because,  $I_1(\C_3)=I_2(\Gamma)$.
\end{proof}

\subsection{The complex $\mathfrak C$ that is used throughout the paper.}

$ $

We return to the Data of \ref{2.1}. 
Recall the  symmetric bilinear form           
$$\la-,-\ra:\Sym_2 U_0\times \Sym_2 U_0\to \kk,$$ given by 
$$\la u_2,u_2'\ra=[(p^{-1}u_2\Phi_3)u_2'](\Phi_3),$$ which is introduced in Data~\ref{2.1}.(\ref{2.1.d}).  

\begin{definition}
\label{8.1}
 Adopt Data~{\ref{2.1}}.
\begin{enumerate}
[\rm (a)]\item Define \begin{align*}
&F:\ts \bw^2U_0\t \bw^2 U_0\to \kk,\\ 
&F':U_0^*\t U_0^* \to \kk,\quad\text{and}\quad \\
&\Gamma:U_0^*\to U_0\end{align*} by
\begin{align}\notag F(u_{1,0}\w u_{1,0}'\t u_{1,0}''\w u_{1,0}''')={}&\la u_{1,0}u_{1,0}'',u_{1,0}'u_{1,0}'''\ra-\la u_{1,0}u_{1,0}''',u_{1,0}'u_{1,0}''\ra,
\\ 
\notag F'(\nu_{1,0}\t \nu_{1,0}')={}&F(\nu_{1,0}(\omega_{U_0})\t 
\nu_{1,0}'(\omega_{U_0})),\quad \text{and}\\
\label{6.1.1}[\Gamma(\nu_{1,0})](\nu_{1,0}')={}&F'(\nu_{1,0}\t \nu_{1,0}')
.\end{align}

\item The homomorphism $\Gamma^{\vee}:U_0\to U_0^*$ is symmetric; so there exists $\delta $ in $D_2U_0^*$ with
$$
\Gamma^{\vee}(u_{1,0})=u_{1,0}\delta \quad\text{for all $u_{1,0}$ in $U_0$}. 
$$In the same manner, if $\Gamma$ is invertible, then there exists
a unit $\lambda$ in $\kk$ with $$
\Gamma^{-1}(u_{1,0})=u_{1,0}\lambda \delta \quad\text{for all $u_{1,0}$ in $U_0$}. 
$$

 \item Let $\SM$ be the matrix for $\Gamma$; use the basis $y^*$, $z^*$, $w^*$ for $U_0^*$ and $y$, $z$, $w$ for $U$:
\begin{equation} \SM=\label{SM}
\begin{array}{|c||c|c|c|}\hline
         &y^*&z^*&w^*\\\hline\hline
y& \la z^2, w^2\ra-\la zw, zw\ra&\la zw, yw\ra-\la zy, w^2\ra&\la yz, zw\ra-\la z^2, yw\ra\\
\hline
z&\la yw, zw\ra-\la yz, w^2\ra&\la y^2, w^2\ra-\la yw, yw\ra&\la yz, yw\ra-\la y^2, wz\ra\\\hline
w&\la yz, zw\ra-\la yw, z^2\ra&\la yw, yz\ra-\la y^2, zw\ra&\la y^2, z^2\ra -\la yz, yz\ra\\\hline
\end{array}\,.\end{equation}
\item Let $\mathfrak C$ be the complex $\mathfrak C_{\Gamma}$:
\begin{equation} \label{6.4.3} 
\mathfrak C: \quad 0\to \kk \xrightarrow{\mathfrak C_3} D_2U_0^*\xrightarrow{\mathfrak C_2} \frac{U_0\t U_0^*}{\ev^*(1)}\xrightarrow{\mathfrak C_1} {\ts\bw^2}U_0\to 0,\end{equation}
built as described in Definition~\ref{24.2} using the data $\Gamma$ of (\ref{6.1.1}).
\end{enumerate}
\end{definition}

\begin{observation}
\label{8.2}
Adopt the data of 
{\rm\ref{2.1}} and the language of Definition~{\rm\ref{8.1}}. Let $B$ be the matrix  of Theorems {\rm\ref{resol}} and {\rm\ref{a res}}. Then 
\begin{enumerate}
[\rm (a)]
\item\label{8.2.a} $I_1(\C_3)=I_2(\SM)$, 
\item\label{8.2.b} the matrices $\C_1$ and $B$ have the same rank, and
\item\label{8.2.c} if $2\le \rank \SM$, then the complex $\mathfrak C$ of {\rm(\ref{6.4.3})} is split exact.\end{enumerate} 
\end{observation}

\begin{proof} 
(\ref{8.2.a}) Let $y,z,w$ be a basis for $U_0$. One basis for $D_2U_0^*$ is $\{
m^*\mid m\in \binom{y,z,w}2\}$. Observe that
$$\C_3(1)=\sum_{m\in\binom{y,z,w}2}[\Gamma^{\vee}(m)]\cdot m^*.$$If 
$\C_3$ is expressed as a $6\times 1$ matrix (see, for example, (\ref{delta-mat})), then the entries of this matrix are precisely the collection of two by two minors of the matrix for $\Gamma$, which is $\SM$.  We conclude that $I_1(\C_3)=I_2(\SM)$. 

\ms \noindent (\ref{8.2.b})
Let  
$\psi_2:{\ts\bigwedge^2} U_0\to U_0^*$ be the isomorphism with $\psi_2(\theta_2)= \theta_2(\omega_{U_0^*})$. We prove that 
the  diagram  
\begin{equation}\label{B=C(Gamma)sub1!}\xymatrix{ \frac{U_0\t U_0^*}{(\ev^*(1))}\ar[rr]^{\C_1}
\ar[rrd]^{\mB}&& \bigwedge^2U_0\ar[d]^{\cong}_{\psi_2}\\
&&U_0^*}\end{equation}
commutes.
Take $u_{1,0}, u_{1,0}',u_{1,0}'',u_{1,0}'''\in U_0$. We compare the elements
\begin{align*}
S_1={}&\Big[\mB\Big(u_{1,0}'\t (u_{1,0}''\w u_{1,0}''')(\omega_{U_0^*})\Big)\Big](u_{1,0})
\quad \text{and}\\ 
S_2={}&\Big[\big(\psi_2\circ \C_1\big)\Big(u_{1,0}'\t (u_{1,0}''\w u_{1,0}''')(\omega_{U_0^*})\Big)\Big](u_{1,0})\end{align*}
of $\kk$.
Apply Remark~\ref{5.4} to see that $$S_1=\la u_{1,0}u_{1,0}'',u_{1,0}'u_{1,0}'''\ra-\la u_{1,0}u_{1,0}''',u_{1,0}'u_{1,0}''\ra.$$ 
On the other hand,
\begingroup \allowdisplaybreaks
\begin{align*}
 S_2={}& \Big(\psi_2\big(u_{1,0}'\w \Gamma\big[(u_{1,0}''\w u_{1,0}''')(\omega_{U_0^*})\big]\big)\Big)(u_{1,0}),
&&\text{by Def. \ref{24.2}}, \\
{}={}& 
\Big[\Big(u_{1,0}'\w \Gamma\big[(u_{1,0}''\w u_{1,0}''')(\omega_{U_0^*})\big]\Big)(\omega_{U_0^*})\Big](u_{1,0}) \\
{}={}& 
\Big(\Gamma\big[(u_{1,0}''\w u_{1,0}''')(\omega_{U_0^*})\big]\Big)\Big((u_{1,0}\w u_{1,0}')(\omega_{U_0^*})\Big)\\
{}={}&F((u_{1,0}''\w u_{1,0}''')\t(u_{1,0}\w u_{1,0}')),&&\text{by Def.~\ref{8.1}},\\
{}={}&\la u_{1,0}''u_{1,0},u_{1,0}'''u_{1,0}'\ra -\la u_{1,0}''u_{1,0}', u_{1,0}'''u_{1,0}\ra,&&\text{by Def.~\ref{8.1}}.
\end{align*}\endgroup 
Observe that $S_1=S_2$.

\smallskip
\noindent(\ref{8.2.c}) Recall that $\SM$ is the matrix for $\Gamma$. It follows from
Lemma~\ref{24.4}.(\ref{24.4.b}) that $\mathfrak C$ is split exact.
\end{proof}

\begin{theorem}
\label{6.3} 
Adopt Data~{\rm\ref{2.1}}. The matrix $B$ 
of Theorems {\rm\ref{resol}} and {\rm\ref{a res}} has 
 rank $0$, $2$, or $3$. The number of minimal generators of the ideal $\ann\Phi_3$ is $9$, $7$, or $6$. 
\end{theorem}

\begin{proof}
Every nonzero entry in the matrix $B$ is plus or minus an entry of the symmetric matrix (\ref{SM}). If $\SM$ is zero, then $B=0$. If $\SM$ has rank one, then
the usual theory of symmetric matrices guarantees that one  can choose a pair of dual bases for $U_0$ and  $U_0^*$  so that $\SM$ becomes 
\begin{equation}\label{rankSM1}\bmatrix \a&0&0\\ 0&0&0\\0&0&0\endbmatrix,\text{ for some unit $\a$ in $\kk$.}\end{equation} 
In this case,
\begin{equation}\label{specialB}B=\bmatrix 
0&0&0&0&0&0&0&0\\
0&0&\a&0&0&0&0&0\\
0&-\a&0&0&0&0&0&0\endbmatrix,\end{equation}which has rank $2$.
The matrix $\SM$ is a matrix of constants. If $2\le \rank \SM$, then $I_2(\SM)=P$; hence, $I_1(\C_3)=P$ and $I_3(B)=P$ by Lemma~\ref{24.4}.(\ref{24.4.e}) and Observation~\ref{8.2}.

We saw in Theorem~\ref{a res} that the minimum number of generators of $\ann \Phi_3$ is $9-\rank B$. 
\end{proof}

\section{The generalized sum of linked perfect ideals.}
In this section
$$(-)^* \text{ is the functor $\Hom_P(-,P)$}.$$
\begin{observation} 
\label{dec23}
Let $J$ be a grade $g$ perfect ideal in the Gorenstein ring $P$. Suppose $\theta:\Ext^g_P(P/J,P)\to P/J$ is an injection. Let $K$ be an ideal of $P$ with $\im \theta=K/J$. If $g+1\le \grade K$, then $K$ is a perfect Gorenstein in $P$ of grade $g+1$.\end{observation}

\begin{proof} Let $0\to F_g\to \dots \to F_0=P$ be a resolution of $P/J$ by free $P$-modules,  and 
\begin{equation}\label{gsli-mc}\xymatrix{
0\ar[r]&F_0^*\ar[r]\ar[d]&\dots\ar[r]&F_{g-1}^*\ar[r]\ar[d]&F_g^*\ar[d]\\ 
0\ar[r]&F_g\ar[r]&\dots\ar[r]&F_1\ar[r]&F_0}\end{equation} be a map of complexes which extends
$$\xymatrix{F_g^*\ar[r]&\Ext^g_P(P/J,P)\ar[r]\ar[d]^{\theta}&0\\
F_0\ar[r]&P/J\ar[r]&0.}$$ 
Observe that 
the mapping cone of  (\ref{gsli-mc}) is a resolution of $P/K$ of length $g+1$ and the last Betti number in this resolution is one.

The fact that $$\grade \ann M\le \pd M,$$ for all finitely generated $P$-modules $M$, ensures that $\pd P/K=g+1$, and  that the minimal resolution of $(P/K)_{\mathfrak p}$ has length $g+1$ and 
 last Betti number equal to one, for all prime ideals $\mathfrak p$ of $P$ which contain $K$. Thus $K$ is a Gorenstein ideal of grade $g+1$.
\end{proof}

\begin{definition}
\label{gsli} The ideal $K$ of Observation~{\rm \ref{dec23}} is called
a {\it generalized sum of linked ideals}.\end{definition}

We offer two examples of  generalized sums of linked ideals. The first example 
is the  
sum of geometrically linked perfect ideals. The second example is a hypersurface section of a Gorenstein ideal.

\begin{example}
\label{GL}
Peskine and Szpiro introduced \cite[Rem.~1.4]{PS} the idea of adding two geometrically linked ideals of grade $g$ to produce a Gorenstein ideal of grade $g+1$. Also, Ulrich's work \cite{U90} along these lines is of significant interest. Recall that the grade $g$ ideals $I$ and $J$ of the Gorenstein (not necessarily local) ring $P$ are {\it linked}, if there exists a regular sequence $\a=\a_1,\dots,\a_g$ in $I\cap J$ with $(\a):I=J$ and $(\a):J=I$. The ideals $I$ and $J$ are {\it geometrically linked} if they are linked and any of the following three equivalent conditions also occur:
\begin{enumerate}[\rm(1)]
\item $\Ass P/I \cap \Ass P/J=\emptyset$,
 \item $I\cap J=(\a_1,\dots,\a_g)$, 
\item $g+1\le\grade (I+J)$.\end{enumerate} 
We think of writing a grade 4 Gorenstein ideal as the sum of two grade three Gorenstein ideals as a labor-saving device. We are able to describe an entire length four resolution in terms of a length three resolution (and connecting maps). In particular, if $I$ is a perfect grade three ideal geometrically linked to $J$ and
\begin{equation}\label{G}0\to G_3\xrightarrow{g_3} G_2\xrightarrow{g_2} G_1\xrightarrow{g_1} P\to P/I\end{equation}is a resolution of $P/I$, then 
$$0\to P^*\xrightarrow{g_1^*} G_1^*\xrightarrow{g_2^*} G_2^* \xrightarrow{g_3^*} G_3^* $$ is a resolution of 
$$\ts \Ext^3_P(\frac PI,P)=\Hom_P(\frac PI,P/(\a))=\frac{\a:I}{\a}=\frac{J}{I\cap J}=\frac{I+J}I$$
and the mapping cone of 
$$\xymatrix{
0\ar[r]& P^*\ar[r]^{g_1^*}\ar[d]^{c_3}& G_1^*\ar[r]^{g_2^*}\ar[d]^{c_2}& G_2^* \ar[r]^{g_3^*}\ar[d]^{c_1}& G_3^*\ar[d]^{c_0} \\
0\ar[r]&G_3\ar[r]^{g_3} &G_2\ar[r]^{g_2}&G_1\ar[r]^{g_1}&P
}$$
 is a resolution of 
$P/(I+J)$, where $c_0:G_3^*\to P$ is a map for which 
$$\xymatrix{G_3^*\ar[r]\ar[d]^{c_0}&\frac{I+J}I\ar@{^{(}->}[d]\\P_0\ar[r]&\frac PI}$$commutes.
\end{example}

\begin{example}
\label{7.4}
The same trick works  when $I$ is a grade three Gorenstein ideal and $a$ is a regular element on $P/I$. The resolution of $P/(I,a)$ is the mapping cone of $a:G^*\to G$ which covers the injection $a:P/I\to P/I$, where $G\cong G^*$ both resolve $P/I$. Of course, $(I,a)$ is not literally the sum of the linked ideals $I$ and $I$.
\end{example}

In practice the resolution ``$G$'' of (\ref{G}) is already well 
 understood. 
 Indeed, for us $G$ is the resolution of a grade 3 Gorenstein ideal when $\ann \Phi_3$ has six minimal generators; $G$ is the resolution of a grade three perfect ideal of type 2 with one minimal Koszul relation as studied by Anne Brown \cite{B87} (see also \cite{K22}) when $\ann \Phi_3$
has seven minimal generators; and $G$ is a deformation of the Eagon-Northcott resolution of the monomials of degree 2 in three variables when $\ann \Phi_3$ has nine generators. 

The paper \cite{MVW} refers to the technique of forming the mapping cone of $c:G^*\to G$, where $\HH_0 (G^*)\hookrightarrow  \HH_0(G)$, as ``doubling''.

\section{The algebra with no weak Lefschetz element.}\label{Frob}

Let $\kk$ be an arbitrary field and   $\AA$ be  a standard graded
Artinian Gorenstein  
$\kk$-algebra of embedding dimension four and  socle degree three. If $\AA$ does not have the weak Lefschetz property, then it is shown in \cite{K23} that
$\kk$ has characteristic two and $\AA$ is isomorphic to
\begin{equation}\label{No}\frac{\kk[x,y,z,w]}{(xy,xz,xw,y^2,z^2,w^2,x^3+yzw)}.\end{equation}  
The Macaulay inverse system of $\AA$ is 
\begin{equation}\label{MIS}
\Phi_3=x^{*(3)}+y^*z^*w^*.\end{equation}
In this section, we show that the defining ideal of $\AA$ is the sum of linked ideals. We also show 
the algebra of \cite[Ex.~5.6]{MVW} is isomorphic to (\ref{No}).
 
\begin{observation}
Let $P$ be the ring $\kk[x,y,z,w]$, where $\kk$ has characteristic two. The mapping cone of the following map of complexes is the minimal homogeneous resolution of $\AA$ from {\rm (\ref{No})}{\rm:}

$$\xymatrix{
0\ar[r]&P\ar[r]^{K_1\transpose}\ar[d]^{L_0\transpose}&P^5\ar[r]^{K_2\transpose}\ar[d]^{L_1\transpose}&P^6\ar[r]^{K_3\transpose}\ar[d]^{L_1}&P^2\ar[d]^{L_0}\\
0\ar[r]&P^2\ar[r]^{K_3}&P^6\ar[r]^{K_2}&P^5\ar[r]^{K_1}&P^1
},$$
where 
\begin{align*}K_1={}&\bmatrix y^2&z^2&xy&xz&xw\endbmatrix, &K_2={}&\bmatrix  
        z^2 &0 &x &0 &0 &0 \\
       y^2 &x &0 &0 &0 &0 \\
       0  &0 &y &0 &w &z \\
       0  &z &0 &w &0 &y \\
      0  &0 &0 &z &y &0 \endbmatrix,\\
K_3={}&
\bmatrix
 x & 0 \\
  y^2 &0 \\
  z^2 &0 \\
 0  &y \\
  0  &z \\
 yz &w \\
\endbmatrix,&
L_0={}&\bmatrix w^2 &x^3+yzw \endbmatrix,\quad \text{and}\\
L_1={}&\bmatrix
 0 &w^2 &0  &zw& 0 & 0  \\
 0 &0  &w^2& 0 & yw& 0  \\
 0 &0  &0 & x^2 &0 & 0  \\
 0 &0  &0 & 0  &x^2& 0  \\
 w &0  &0 & 0  &0 & x^2 \endbmatrix.
\end{align*}
 \end{observation}
The proof is obvious.

\begin{example}
We change the basis to the Macaulay inverse system for   \cite[Ex.~5.6]{MVW} into the form of (\ref{MIS}). 
In this example $\kk$ has characteristic $2$, and the Macaulay Inverse System for $\AA$ is  $$\Phi_3=x^{*(3)}+x^*y^*z^*+x^*y^*w^*+x^*z^*w^*+y^*z^*w^*\in D_3U^*.$$
Use the  
$X^*=x^*$, $Y^*=y^*+w^*$, $Z^*=z^*+w^*$, $W^*=w^*$ for $U^*$. 
Notice that 
$$\Phi_3=X^{*(3)}+X^*Y^*Z^*+Y^*Z^*W^*$$ 
Do another change of basis! Replace  $W^*$ with $\widetilde {W}^*=X^*+W^*$ and keep $X^*$, $Y^*$, and $Z^*$. The Macaulay inverse system for $\AA$ is 
$$X^{*(3)}+Y^*Z^*\widetilde {W}^*.$$
\end{example}

\section{The case when $\SM$ is zero.}

\begin{theorem}
\label{9.1} Adopt the data of {\rm \ref{2.1}} and the matrix $\SM$ of {\rm(\ref{SM})}. If $\SM$ is zero, then $\ann \Phi_3$ has nine minimal generators and the minimal resolution of $\AA$ is given by the complex $(\frak F, \widetilde{\frak f})$ from Theorem~{\rm \ref{a res}}. Furthermore,
 the defining ideal  of $\AA$ is the sum of two linked codimension three perfect ideals.
\end{theorem}
\begin{Remark}
Theorem~\ref{9.1} is implemented in the Macaulay2 script ``SumOfLinkedIdeals9''; see Section~\ref{14.A.4}. \end{Remark}
\begin{proof}
 If $\SM=0$, then $B=0$ and the minimal number of generators of $\ann \Phi_3$ is $9-\rank B=9$. It is already shown in the proof of Theorem~\ref{a res} 
that $(\frak F, \widetilde{\frak f})$ is a minimal resolution of $\AA$.

We show that the defining ideal of $\AA$ is the sum of two linked codimension three perfect ideals. When $B=0$, then a quick look at Theorem~\ref{a res} shows that the minimal resolution $(\frak F, \widetilde{\frak f})$ of $\AA$ is the mapping cone of the following map of complexes:
{\small$$\xymatrix{
0\ar[r]&P\ar[rr]^{\mf_{1,2}^*}\ar[d]^{-x\mf_{1,1}^*}&&{P\t D_2U_0^*}\ar[r]^{\mD^*}\ar[d]^{\mC ^*}& P\t \ker(U_0\t U_0^*) \ar[rr]^{\mA }\ar[d]^{-\mC }&& P\t U_0^*\ar[d]^{x \mf_{1,1}}\\
0\ar[r]&{P\t U_0}\ar[rr]^{\mA ^*}&&P\t\frac{U_0\t U_0^*}{\ev^*(1)}
\ar[r]^{\mD}&P\t\Sym_2U_0\ar[rr]^{\mf_{1,2}}&& P.
}$$}
We complete the proof by observing that 
{\small\begin{equation}\label{9.1.1}\xymatrix{0\ar[r]&{P\t U_0}\ar[rr]^{\mA ^*}&&P\t\frac{U_0\t U_0^*}{\ev^*(1)}
\ar[r]^{\mD}&P\t\Sym_2U_0\ar[rr]^{\mf_{1,2}}&& P}\end{equation}}
is the resolution of a perfect codimension three quotient of $P$. Indeed, the complex (\ref{9.1.1}) is a deformation of the Eagon-Northcott complex $E$ of Observation~\ref{x=0}. Apply Lemma~\ref{9.3}.
\end{proof}

\begin{lemma}
\label{9.3} Let $P=P_0[x]$ be a polynomial ring, $F$ be a complex of finitely generated free $P$-modules, and $\bar P$ be the quotient ring $P/(x)$. If 
$$\HH_i(F\t_P \bar P)\text{ and }\HH_{i+1}(F\t_P \bar P)\text{ are both zero,}$$ then $\HH_i(F
)$ is zero.\end{lemma}

\begin{proof} The long exact sequence of homology that corresponds to the short exact sequence of complexes
$$0\to F(-1) \xrightarrow{x} F\to F\t_P \bar P\to 0$$
yields that multiplication by $x$ from $\HH_{i}(F)_{j-1}$ to $\HH_{i}(F)_{j}$ is an isomorphism for all $j$. (The ``$(-1)$'' refers to a shift of internal degree.)  However $\HH_{i}(F)_{j-1}$ is zero for all $j\ll 0$; consequently, $\HH_i(F)=0$.
\end{proof}

\section{The case when $\SM$ has rank one.}
Adopt the data of {\rm \ref{2.1}} and the matrix $\SM$ of {\rm(\ref{SM})}. Assume that $\SM$ has rank one. 
The usual theory of symmetric matrices guarantees that one  can choose a pair of dual bases for $U_0$ and  $U_0^*$  so that $\SM$ is given in (\ref{rankSM1}).
Recall the resolution $(\widetilde {\frak F},\widetilde{\frak f})$ of $\AA$ which is given in Theorem~\ref{a res}. The differential    $\widetilde{\frak f}_2$ of $\widetilde {\frak F}$ 
is the map $$\begin{matrix} P^8\\\p\\ P^8\end{matrix}\xrightarrow{\widetilde{\frak f}_2=\bmatrix A&B\\x^2C&D\endbmatrix}\begin{matrix} P^3\\\p\\ P^6\end{matrix},$$
where the matrices $A$, $B$, $C$, and $D$ are given in  Example~\ref{record}. When $\SM$ is given in (\ref{rankSM1}), then $B$ is given in (\ref{specialB}).

The minimal resolution $\AA$ is given in Theorem~\ref{7.1}. It is shown in Theorem~\ref{7.2} that the  defining ideal  of $\AA$ is the sum of two linked codimension three perfect ideals.

\begin{definition}
\label{10.1}
Adopt the data of  {\rm \ref{2.1}} and the matrix $\SM$ of {\rm(\ref{SM})}. Assume that $\SM$ has rank one. Assume further that $\SM$ is given in 
{\rm(\ref{rankSM1})}. Recall the matrices \begingroup\allowdisplaybreaks
\begin{align*}A={}&\bmatrix
a_{1,1}&\dots&a_{1,8}\\
a_{2,1}&\dots&a_{2,8}\\
a_{3,1}&\dots&a_{3,8}\endbmatrix,&B={}&\bmatrix
b_{1,1}&\dots&b_{1,8}\\
b_{2,1}&\dots&b_{2,8}\\
b_{3,1}&\dots&b_{3,8}\endbmatrix,&C={}&\bmatrix
c_{1,1}&\dots&c_{1,8}\\
\vdots&\dots&\vdots\\
c_{6,1}&\dots&c_{6,8}\endbmatrix,\ \text{and}\\D={}&\bmatrix
d_{1,1}&\dots&d_{1,8}\\
\vdots&\dots&\vdots\\
d_{6,1}&\dots&d_{6,8}\endbmatrix,\ \end{align*}\endgroup
from Theorem~{\rm\ref{resol}} and Example~{\rm\ref{record}}.
The three rows of $A$ 
are denoted $A_{1,*}$, $A_{2,*}$, $A_{3,*}$. The eight 
 columns of $D$ denoted $D_{*,1},D_{*,2},\dots,D_{*,8}$. The same convention is used to denote the rows and columns of all four matrices.

Define the invertible matrices $\theta_1$ and  $\theta_2$ 
by 
$$\theta_1=\begin{array} {|l|ll|l|}\hline
1&0&0&\\
0&0&\a&0_{3\times 6}\\
0&-\a&0&\\\hline
0_{6\times 1}&D_{*,2}&D_{*,3}&I_6\\\hline\end{array}\quad\text{and}\quad
\theta_2= 
\begin{array}{|l|l|}\hline
I_8&0_{8\times 8}\\\hline
M&I_8\\
\hline\end{array}\,, $$
where $M$ is the $8\times 8$  
matrix
$$M=\frac 1{\alpha}\bmatrix 
0&-a_{3,1}&a_{2,1}&0&0&0&0&0\\
a_{3,1}&0&a_{3,3}&a_{3,4}&a_{3,5}&a_{3,6}&a_{3,7}&a_{3,8}\\
-a_{2,1}&-a_{2,2}&0&-a_{2,4}&-a_{2,5}&-a_{2,6}&-a_{2,7}&-a_{2,8}\\
0&-a_{3,4}&a_{2,4}&0&0&0&0&0\\
0&-a_{3,5}&a_{2,5}&0&0&0&0&0\\
0&-a_{3,6}&a_{2,6}&0&0&0&0&0\\
0&-a_{3,7}&a_{2,7}&0&0&0&0&0\\
0&-a_{3,8}&a_{2,8}&0&0&0&0&0\\
\endbmatrix.$$
Let $J$ be the matrix
$$J=\bmatrix 0_{6\times 6}&I_6\\I_6&0_{6\times 6}\endbmatrix.$$
Define the matrices $E_1$, $E_2$, $E_3$, $E_4$ by
\begin{itemize}
\item $E_1$ is $\widetilde{\frak f}_1 \theta_1$ with columns $2$ and $3$ deleted; 
\item $E_2$ is $\theta_1^{-1}  \widetilde {\frak f}_2 \theta_2$ with  rows $2$ and $3$ and columns $2,3,10,11$ deleted; 
\item $E_3= J  E_2\transpose$; and
\item $E_4=E_1\transpose$.
\end{itemize}
 \end{definition}
\begin{remark}
\label{10.2}
In Section~\ref{calculations} we have collected many consequences of the hypothesis that $B$ has the form of (\ref{specialB}). In particular, in 
(\ref{12.1.1}) we show that
$$a_{1,3}=a_{2,3}=a_{3,2}=a_{1,2}=0\quad \text{and}\quad a_{3,3}=a_{2,2}.$$
It follows   that $M$ is an alternating matrix; row 2 of $M$ is $\frac 1{\alpha} A_{3,*}$; and row 3 of $M$ is $-\frac 1{\alpha} A_{2,*}$. The fact  that $M$ is an alternating matrix ensures that \begin{equation}\label{10.2.1}\theta_2J\theta_2\transpose=J.\end{equation}
\end{remark}

\begin{theorem}
\label{7.1} 
Adopt the data of  {\rm \ref{2.1}} and the matrix $\SM$ of {\rm(\ref{SM})}. Assume that $\SM$ has rank one. Assume further that $\SM$ is given in 
{\rm(\ref{rankSM1})}. Let $E_1$, $E_2$, $E_3$, $E_4$ be the matrices of Definition~{\rm\ref{10.1}}. Then the minimal homogeneous resolution of $\AA$ by free $P$-modules is given by
$$E:\quad 0\to P\xrightarrow{E_4}P^7\xrightarrow{E_3}P^{12}\xrightarrow{E_2}P^7\xrightarrow{E_1} P.$$
\end{theorem}
\begin{Remark}
Theorem~\ref{7.1} is implemented in the Macaulay2 script ``minres7''; see Section~\ref{14.A.2}. \end{Remark}
\begin{proof} 
Recall the resolution $(\frak F,\widetilde{\frak f})$ of Theorem~\ref{a res}. Define $(\frak F', \frak f')$ to be
$$0\to P\xrightarrow{\frak f_4'=\frak f_1'{}\transpose}P^9\xrightarrow{\frak f_3'= J\,\frak f_2'{}\transpose}P^{16}\xrightarrow{\frak f_2'}P^9\xrightarrow{\frak f_1'}P,$$ where
$$ 
\frak f_2'=\theta_1^{-1}\,{} \widetilde{\frak f}_2 {}\theta_2 \quad\text{ and}\quad
\frak f_1'=\widetilde{\frak f}_1\theta_1.$$
Use (\ref{10.2.1}) in order to observe that
$$\xymatrix{
(\frak F',\frak f'):\ar[d]&0\ar[r]&P
\ar@{=}[d]
\ar[rr]^{\frak f_1'{}\transpose} 
&&P^9
\ar[rr]^{J\frak f_2'{}\transpose}
\ar[d]^{\theta_1^{-1}\,\transpose}
&&P^{16}
\ar[rr]^{\frak f'_2}
\ar[d]^{\theta_2}
&&P^9\ar[r] ^{\frak f'_1}  \ar[d]^{\theta_1}
&P\ar@{=}[d]
\\
(\frak F,\widetilde{\frak f}):&0\ar[r]&P\ar[rr]^{\widetilde{\frak f}_1\,\transpose}&&P^9\ar[rr]^{J\,{}\widetilde{\frak f}_2{}\transpose}&&P^{16}\ar[rr]^{\widetilde{\frak f}_2 }&&P^9\ar[r] ^{\widetilde{\frak f}_1}&P,\\
}$$ 
is an isomorphism of complexes. 

\begin{claim-no-advance} The following statements hold.
\label{10.3.1}
\begin{enumerate}[\rm (a)]
\item\label{10.3.1.a} Columns  $2$ and $3$ of $\frak f'_1$ are  
$$\left[ \begin{array}{c} 0_{1\times 2}\end{array}\right].$$ 
\item\label{10.3.1.c} Columns $2$ and $3$ of $\frak f'_2 $ are 
$$\left[ \begin{array}{c}0_{9\times 2}\end{array}\right].$$

\item\label{10.3.1.d} Columns $10$ and $11$ of $\frak f'_2$ are 
\begin{equation}
\label{J20.1}\left[ \begin{array}{c} 0_{1\times 2}\\ I_2 
\\0_{6\times 2}\end{array}\right].\end{equation} 

\item \label{10.3.1.b}Rows $2$ and $3$ of 
$\frak f'_2 $ are
$$\left[\begin{array}{c|c|c} 0_{2\times 9}& I_2
&0_{2\times 5}\end{array}\right],$$ 
\end{enumerate}
\end{claim-no-advance}

\ms\noindent{\it Proof of Claim {\rm \ref{10.3.1}}.} (\ref{10.3.1.a})
Columns 2 and 3 of $\theta_1$ are equal to columns 2 and 3 of $\widetilde{\frak f}_2$; hence columns 2 and 3 of $\frak f'_1=\widetilde{\frak f}_1\theta_1$ are identically zero.

\ms
\noindent (\ref{10.3.1.c}) Column 3 of $\theta_2$ is equal to $\frac 1{\a}$ times column 2 of $\widetilde{f_3}$. Column 2 of $\theta_2$ is equal to $-\frac 1{\a}$ times column 3 of $\widetilde{f_3}$; hence columns 2 and 3 of $\frak f_2'=\theta_1^{-1}\,\widetilde{\frak f}_2\,\theta_2$ are identically zero.

\ms
\noindent (\ref{10.3.1.d}) Observe that
$$\theta_1^{-1}=\begin{array}{|ccc|c|}\hline 1&0&0& \\
0&0&\frac{-1}{\alpha}&0\\
0&\frac 1{\alpha}&0&\\\hline
0&\frac{-1}{\alpha}D_{*,3}&\frac 1{\alpha}D_{*,2}&I_6\\\hline\end{array}\,,\quad 
(\,\,\widetilde{\frak f}_2\,\theta_2)_{*,10}=(\,\,\widetilde{\frak f}_2\,)_{*,10}=\bmatrix 0\\0\\-\alpha\\\hline D_{*,2}\endbmatrix,\quad\text{and}$$$$ (\,\,\widetilde{\frak f}_2\,\theta_2)_{*,11}=(\,\,\widetilde{\frak f}_2\,)_{*,11}=\bmatrix 0\\\alpha\\0\\\hline D_{*,3}\endbmatrix.$$ Conclude that columns 10 and 11 of $\frak f_2'=\theta_1^{-1}\,\widetilde{\frak f}_2\,\theta_2$ are given in (\ref{J20.1}). 

\ms
\noindent (\ref{10.3.1.b})
Rows 2 and 3 of $\frak f_2'=(\theta_1^{-1}\,\widetilde{\frak f}_2)\theta_2$
\begingroup\allowdisplaybreaks\begin{align*}
{}={}&\left[\begin{array}{c|cccc} 
-\frac 1{\alpha}A_{3,*}&0&1&0&0_{1\times 5}\\[3pt]
\phantom{-}\frac 1{\alpha}A_{2,*}&0&0&1&0_{1\times 5}\\
\end{array}\right]
\left[\begin{array}{l|l}
I_8&0_{8\times 8}\\\hline
M&I_8\\
\end{array}\right]\\
{}={}&\left[\begin{array}{c|cccc} 
-\frac 1{\alpha}A_{3,*}+M_{*,2}&0&1&0&0_{1\times 5}\\[3pt]
\phantom{-}\frac 1{\alpha}A_{2,*}+M_{*,3}&0&0&1&0_{1\times 5}\\
\end{array}\right]\\
{}={}&\left[\begin{array}{c|cccc} 
-\frac 1{\alpha}A_{3,*}+\frac 1{\alpha}A_{3,*} &0&1&0&0_{1\times 5}\\[3pt]
\phantom{-}\frac 1{\alpha}A_{2,*}-\frac1{\alpha}A_{*,2}&0&0&1&0_{1\times 5}\\
\end{array}\right]\\
{}={}&\left[\begin{array}{c|cccc} 
0_{1\times 8}&0&1&0&0_{1\times 5}\\
0_{1\times 8}&0&0&1&0_{1\times 5}\\\end{array}\right],
\end{align*}
\endgroup
as claimed.
This completes the proof of Claim\ref{10.3.1}.

\bs

Claim~\ref{10.3.1} establishes that the complex $E$ is a direct summand of the resolution $\frak F'$ and that the complement of $E$ in $\frak F'$ is a split exact complex. It follows that $E$ is a resolution. The resolution $E$ is homogeneous and has the same Betti numbers as the minimal homogeneous resolution of $\AA$; therefore, $E$ is the minimal resolution of $\AA$. 
\end{proof}
\begin{remark} For future reference, we notice that the graded Betti numbers of $E$ are
$$0\to P(-7)\to \begin{matrix} P(-4)\\\p\\P(-5)^6\end{matrix}\to
\begin{matrix} P(-3)^6\\\p\\P(-4)^6\end{matrix}
\to
\begin{matrix} P(-2)^6\\\p\\P(-3)\end{matrix}\to P.$$\end{remark}

\section{Calculations.}\label{calculations} 
We have already applied (\ref{12.1.1}) in the proof of Theorem~\ref{7.1};  furthermore, the 
calculations of this section are used extensively in the proof of Theorem~\ref{7.2}.
\begin{lemma} 
\label{12.1}
Let $\mA $,  $\mB $, $\mC $, and $\mD $ be the maps of Theorem~{\rm\ref{resol}}.  
Then
$$\mA \mB^* :P\t U_0\to P\t U_0^*$$ is an alternating map. In particular, if $B$ is given in  {\rm(\ref{specialB})}, then
$$BA \transpose=\a\bmatrix 0&0&0\\a_{1,3}&a_{2,3}&a_{3,3}\\-a_{1,2}&-a_{2,2}&-a_{3,2}\endbmatrix$$is an alternating matrix. In particular,
\begin{equation}\label{12.1.1}a_{1,3}=a_{2,3}=a_{3,2}=a_{1,2}=0\quad \text{and}\quad a_{3,3}=a_{2,2}.\end{equation} 
In other words,
\begin{align}
\label{i1} [y^2p^{-1}(w^*)](\Phi_3)={}& w^*[p^{-1}(y^2\Phi_3)]=0,\\
\label{i2}[yzp^{-1}(w^*)](\Phi_3)={}& w^*[p^{-1}(yz\Phi_3)]=0,\\
\label{i3} [ywp^{-1}(z^*)](\Phi_3)={}& z^*[p^{-1}(yw\Phi_3)]=0,\\
\label{i4} [y^2p^{-1}(z^*)](\Phi_3)={}& z^*[p^{-1}(y^2\Phi_3)]=0,\\
\label{i5} [yzp^{-1}(z^*)](\Phi_3)={}&[ywp^{-1}(w^*)](\Phi_3),\text{ and}\\ \notag z^*[p^{-1}(yz\Phi_3)]={}&w^*[p^{-1}(yw\Phi_3)].\end{align}
\end{lemma}

\begin{proof} 
Let $u_{1,0}$ be an element of $U_0$.
Use Remark~\ref{4.4+} and Theorem~\ref{resol} to see that
\begingroup\allowdisplaybreaks
\begin{align*}&-\big((\mA \circ \mB ^*)(u_{1,0})\big)(u_{1,0})\\{}={}&
\mA \left(\begin{cases}
- w\t [zp^{-1}(yu_{1,0})(\Phi_3)](\Phi_3)
+z\t [wp^{-1}(yu_{1,0})(\Phi_3)](\Phi_3)\\
+w\t [yp^{-1}(zu_{1,0})(\Phi_3)](\Phi_3)
-y\t [wp^{-1}(zu_{1,0})(\Phi_3)](\Phi_3)\\
-z\t [yp^{-1}(wu_{1,0})(\Phi_3)](\Phi_3)
+y\t [zp^{-1}(wu_{1,0})(\Phi_3)](\Phi_3)\\
\end{cases}\right)(u_{1,0})\\{}={}&S_1+S_2,\end{align*}\endgroup
with 
\begingroup\allowdisplaybreaks
\begin{align*}
S_1={}&
x\t \left(\begin{cases}
\phantom{+}
 \Big[wp^{-1}\Big([zp^{-1}(yu_{1,0})(\Phi_3)](\Phi_3)\Big)\Big](\Phi_3)\\
-\Big[zp^{-1}\Big([wp^{-1}(yu_{1,0})(\Phi_3)](\Phi_3)\Big)\Big](\Phi_3)\\
-\Big[wp^{-1}\Big([yp^{-1}(zu_{1,0})(\Phi_3)](\Phi_3)\Big)\Big](\Phi_3)\\
+\Big[yp^{-1}\Big([wp^{-1}(zu_{1,0})(\Phi_3)](\Phi_3)\Big)\Big](\Phi_3)\\
+\Big[zp^{-1}\Big([yp^{-1}(wu_{1,0})(\Phi_3)](\Phi_3)\Big)\Big](\Phi_3)\\
-\Big[yp^{-1}\Big([zp^{-1}(wu_{1,0})(\Phi_3)](\Phi_3)\Big)\Big](\Phi_3)\\
\end{cases}
\right)(u_{1,0})
\\
{}={}&
x\t \begin{cases}
+\Big[\Big([zp^{-1}(yu_{1,0})(\Phi_3)](\Phi_3)\Big)\Big][p^{-1}(wu_{1,0}\Phi_3)]\\
-\Big[\Big([wp^{-1}(yu_{1,0})(\Phi_3)](\Phi_3)\Big)\Big][p^{-1}(zu_{1,0}(\Phi_3)]\\
-\Big[\Big([yp^{-1}(zu_{1,0})(\Phi_3)](\Phi_3)\Big)\Big][p^{-1}(wu_{1,0}\Phi_3)]\\
+\Big[\Big([wp^{-1}(zu_{1,0})(\Phi_3)](\Phi_3)\Big)\Big][p^{-1}(yu_{1,0}\Phi_3)]\\
+\Big[\Big([yp^{-1}(wu_{1,0})(\Phi_3)](\Phi_3)\Big)\Big][p^{-1}(zu_{1,0}(\Phi_3)]\\
-\Big[\Big([zp^{-1}(wu_{1,0})(\Phi_3)](\Phi_3)\Big)\Big][p^{-1}(yu_{1,0}\Phi_3)]\\
\end{cases}\\
{}={}&
x\t \begin{cases}
+\Big[z[p^{-1}(wu_{1,0}\Phi_3)][p^{-1}(yu_{1,0}\Phi_3)]\Big](\Phi_3)\\
-\Big[w[p^{-1}(zu_{1,0}\Phi_3)][p^{-1}(yu_{1,0}\Phi_3)]\Big](\Phi_3)\\
-\Big[y[p^{-1}(wu_{1,0}\Phi_3)][p^{-1}(zu_{1,0}\Phi_3)]\Big](\Phi_3)\\
+\Big[w[p^{-1}(yu_{1,0}\Phi_3)][p^{-1}(zu_{1,0}\Phi_3)]\Big](\Phi_3)\\
+\Big[y[p^{-1}(zu_{1,0}\Phi_3)][p^{-1}(wu_{1,0}\Phi_3)]\Big](\Phi_3)\\
-\Big[z[p^{-1}(yu_{1,0}\Phi_3)][p^{-1}(wu_{1,0}\Phi_3)]\Big](\Phi_3)
\end{cases}\\
{}={}0
\end{align*}
and
$$S_2= \begin{cases}
- w\cdot [u_{1,0}zp^{-1}(yu_{1,0})(\Phi_3)](\Phi_3)
+z\cdot [u_{1,0}wp^{-1}(yu_{1,0})(\Phi_3)](\Phi_3)\\
+w\cdot [u_{1,0}yp^{-1}(zu_{1,0})(\Phi_3)](\Phi_3)
-y\cdot [u_{1,0}wp^{-1}(zu_{1,0})(\Phi_3)](\Phi_3)\\
-z\cdot [u_{1,0}yp^{-1}(wu_{1,0})(\Phi_3)](\Phi_3)
+y\cdot [u_{1,0}zp^{-1}(wu_{1,0})(\Phi_3)](\Phi_3)\\
\end{cases}=0.$$
\endgroup
\end{proof}

\begin{corollary} 
\label{calculation}
Let $A$,  $B$, $C$, and $D$ be  
the matrices of Example~{\rm\ref{record}}. If $B$ is given in  {\rm(\ref{specialB})}, then
\begin{enumerate}[\rm(a)]
\item\label{12.2.a} $d_{1,1}=d_{4,4}=d_{4,5}=d_{4,6}=d_{6,7} =d_{1,4}+d_{2,1}=d_{1,7}+d_{3,1}=d_{4,8}+d_{5,5}=0$,
$a_{1,1}+d_{3,4}=a_{1,4}+d_{1,7}=a_{1,7}+d_{2,1}=a_{2,8}+d_{4,2}=a_{3,8}+d_{4,3}=c_{4,8}=0$,
\item\label{12.2.b} $d_{5,7}=d_{6,4}=0$,
\item\label{12.2.c} $d_{3,7}=d_{6,1}=d_{6,5}=0$,
\item\label{12.2.d} $a_{1,8}=d_{2,4}=d_{4,1}=d_{5,6}=0$,
\item\label{12.2.e} $d_{4,7}=d_{5,4}=0$, and
\item\label{12.2.f} $a_{1,5}=d_{5,1}=d_{2,7}+d_{3,4}=d_{5,5}-d_{6,6}=0$.
\end{enumerate}
 \end{corollary}

\begin{proof} 
One may read (\ref{12.2.a}) directly from  
 Example~\ref{record}; (\ref{12.2.b}) follows from 
(\ref{i1}); (\ref{12.2.c}) follows from 
(\ref{i2}); (\ref{12.2.d}) follows from 
(\ref{i3}); (\ref{12.2.e}) follows from 
(\ref{i4});          
and (\ref{12.2.f}) follows from (\ref{i5}). 
\end{proof}

\begin{lemma}
\label{11.3} 
Let $A$,  $B$, $C$, and $D$ be  
the matrices of Example~{\rm\ref{record}}. If $B$ is given in  {\rm(\ref{specialB})}, then
the following identities all hold in  
$\kk$. 

\begin{enumerate}
[\rm(a)]

\item\label{12.2.h} 
$$0=\begin{cases}
\phantom{+}[w^2p^{-1}(z^*)](\Phi_3)\cdot [y^*p^{-1}(z^2\Phi_3)]
-[zwp^{-1}(z^*)](\Phi_3)\cdot [y^*p^{-1}(zw\Phi_3)]\\
-\a [y^*p^{-1}(z^*)]
-[y^*p^{-1}(w^2\Phi_3)]\cdot 
[w^*p^{-1}(wz\Phi_3)]
\\
+  [y^*p^{-1}(w^2\Phi_3)]\cdot    
[y^*p^{-1}(yz\Phi_3)]
-[y^*p^{-1}(zw\Phi_3)]\cdot   
[y^*p^{-1}(yw\Phi_3)]
\\
+[y^*p^{-1}(zw\Phi_3)]\cdot 
[w^*p^{-1}(w^2\Phi_3)],
\end{cases}$$

\item\label{20.3.5} 
$$0=\begin{cases}
+\a z^*p^{-1}(z^*)
-[y^*p^{-1}(yw\Phi_3)]\cdot [ywp^{-1}(y^*)](\Phi_3)\\
+[y^*p^{-1}(yw\Phi_3)]\cdot [w^2p^{-1}(w^*)](\Phi_3)
-[y^*p^{-1}(w^2\Phi_3)]\cdot [ywp^{-1}(w^*)](\Phi_3)\\
+[z^*p^{-1}(zw\Phi_3)]\cdot [zwp^{-1}(z^*)](\Phi_3)
-[z^*p^{-1}(zw\Phi_3)]\cdot [w^2p^{-1}(w^*)](\Phi_3)\\
+[z^*p^{-1}(w^2\Phi_3)]\cdot [zwp^{-1}(w^*)](\Phi_3) 
+[y^*p^{-1}(y^2\Phi_3)]\cdot [w^2p^{-1}(y^*)](\Phi_3)\\
-[z^*p^{-1}(zy\Phi_3)]\cdot [w^2p^{-1}(y^*)](\Phi_3)
+[y^*p^{-1}(yz\Phi_3)]\cdot [w^2p^{-1}(z^*)](\Phi_3)\\
-[z^*p^{-1}(z^2\Phi_3)]\cdot [w^2p^{-1}(z^*)](\Phi_3),
\end{cases}$$

\item\label{20.3.3}
$$0=\left\{\begin{array}{ll}-\a [z^*p^{-1}(w^*)]
-[y^*p^{-1}(yw\Phi_3)]\cdot[yzp^{-1}(y^*)](\Phi_3)\\
+[y^*p^{-1}(yw\Phi_3)]\cdot [wzp^{-1}(w^*)](\Phi_3)
-[z^*p^{-1}(zw\Phi_3)]\cdot [wzp^{-1}(w^*)](\Phi_3)\\
+[z^*p^{-1}(w^2\phi_3)]\cdot [z^2p^{-1}(w^*)](\Phi_3)
+[y^*p^{-1}(y^2\Phi_3)]\cdot [zwp^{-1}(y^*)](\Phi_3)\\
+[y^*p^{-1}(yz\Phi_3)]\cdot [zwp^{-1}(z^*)](\Phi_3)
-2[z^*p^{-1}(yz\Phi_3)]\cdot [zwp^{-1}(y^*)](\Phi_3),\\\end{array}\right.
$$

\item\label{22.0.14}
$$0=\begin{cases}
-[x^*p^{-1}(w^2\Phi_3)]
+\a [z^*p^{-1}(z^*)]+[w^*p^{-1}(zw\Phi_3)]\cdot [z^*p^{-1}(w^2\Phi_3)]\\
-[y^*p^{-1}(w^2\Phi_3)]\cdot [z^*p^{-1}(yz\Phi_3)]
-[z^*p^{-1}(w^2\Phi_3)]\cdot [z^*p^{-1}(z^2\Phi_3)]\\
+[z^*p^{-1}(zw\Phi_3)]\cdot [z^*p^{-1}(zw\Phi_3)]
-[w^*p^{-1}(w^2\Phi_3)]\cdot[z^*p^{-1}(zw\Phi_3)],\\
\end{cases}$$

\item \label{22.0.9}
$$0=\left\{\begin{array}{ll}-\a [z^*p^{-1}(w^*)]
-[y^*p^{-1}(zw\Phi_3)][z^*p^{-1}(yz\Phi_3)]-[x^*p^{-1}(zw\Phi_3)] \\
-[z^*p^{-1}(zw\Phi_3)]\cdot [z^*p^{-1}(z^2\Phi_3)]
+[z^*p^{-1}(z^2\Phi_3)]\cdot [z^*p^{-1}(zw\Phi_3)]\\
-[w^*p^{-1}(zw\Phi_3)]\cdot [z^*p^{-1}(zw\Phi_3)]
+[w^*p^{-1}(z^2\Phi_3)]\cdot [z^*p^{-1}(w^2\Phi_3)],\\\end{array}\right.
$$

\item\label{22.0.21}
$$0=\begin{cases}
\phantom{+}[w^*p^{-1}(wy\Phi_3)]\cdot [w^*p^{-1}(yw\Phi_3)]
-[y^*p^{-1}(y^2\Phi_3)]\cdot [w^*p^{-1}(yw\Phi_3)]\\
-[x^*p^{-1}(y^2\Phi_3)],\\
\end{cases}$$
\item 
\label{22.0.25}
$0=
[y^*p^{-1}(yz\Phi_3)]\cdot [w^*p^{-1}(yw\Phi_3)]
+[x^*p^{-1}(yz\Phi_3)],\quad\text{and}
$
\item \label{22.0.19}
$0=
[w^*p^{-1}(wy\Phi_3)]\cdot [y^*p^{-1}(yw\Phi_3)]
+[x^*p^{-1}(yw\Phi_3)].$
\end{enumerate}

\end{lemma}

\begin{proof} All of these identities are consequences of the fact that 
$\left(\frak F, \widetilde {\frak f}\,\,\right)$,   
 from Theorem~\ref{a res}, is a complex and $B$ has the form of (\ref{specialB}).
In particular, 
$$\bmatrix A&B\\x^2C&D\endbmatrix\bmatrix B\transpose&D\transpose\\A\transpose&x^2C\transpose\endbmatrix=0 
\quad\text{and}\quad
\bmatrix xf_{1,1}&f_{1,2}\endbmatrix\bmatrix A&B\\x^2C&D\endbmatrix=0. $$
Consequently, 

\begin{equation}\notag 
x^2 CB\transpose+DA\transpose=0\quad\text{and}\quad 
xf_{1,1}B+f_{1,2}D=0.\end{equation}
The sum 
(\ref{12.2.h}) is minus the coefficient of  $x^2$ in $$(x^2CB\transpose+DA\transpose)_{1,3}=0;$$
(\ref{20.3.5}) is the coefficient of  $x^2$ in $$(x^2CB\transpose+DA\transpose)_{2,3}=0;$$
(\ref{20.3.3}) is the coefficient of $x^2$ in $$(x^2CB\transpose+DA\transpose)_{2,2}=0;$$
(\ref{22.0.14}) is the coefficient of $x^2z$ in 
$$(xf_{1,1}B+f_{1,2}D)_3=0;$$
(\ref{22.0.9}) is the coefficient of $x^2z$ in $$(xf_{1,1}B+f_{1,2}D)_2=0;$$
(\ref{22.0.21}) is the coefficient of $x^2w$ in $$(xf_{1,1}B+f_{1,2}D)_{4}=0;$$
(\ref{22.0.25}) is the coefficient of $x^2w$ in $$(xf_{1,1}B+f_{1,2}D)_1=0;$$
and (\ref{22.0.19}) is the coefficient of $x^2y$ in $$(xf_{1,1}B+f_{1,2}D)_{4}=0.$$
\end{proof}

\section{The seven-generated ideal is a sum of linked perfect ideals.}
In this section we prove the following result.

\begin{theorem}
\label{7.2} 
Adopt the data of {\rm \ref{2.1}} and the matrix $\SM$ of {\rm(\ref{SM})}. Assume that $\SM$ has rank one. 
Choose a pair of dual bases for $U_0$ and  $U_0^*$  so that $\SM$ is given in {\rm (\ref{rankSM1})}. Then the defining ideal  of $\AA$ is the sum of two linked codimension three perfect ideals.
\end{theorem}
\begin{Remark}
Theorem~\ref{7.2} is implemented in the Macaulay2 script ``SumOfLinkedIdeals7''; see Section~\ref{14.A.3}. \end{Remark}
\begin{proof}
 The minimal homogeneous resolution $E$ of $\AA$ is given in Theorem~\ref{7.1}. In Definition~\ref{11.4} we rearrange the rows and columns of the matrices of $E$ and produce the minimal homogeneous resolution $E'$ of $\AA$. In 
Definition~\ref{11.4} we also partition $E'$ in the form 
\begingroup\allowdisplaybreaks
\begin{align} \label{partition} E_1'={}&\left[\begin{array}{l|l}L_0& K_1\end{array}\right],&
E_2'={}&\left[\begin{array}{l|l} K_3\transpose&Z\\\hline
 {L}_1&K_2\\\end{array}\right],\\ 
\notag E_3'={}&\left[\begin{array}{l|l} Z\transpose&K_2\transpose\\\hline
K_3&\fakeht L_1\transpose\end{array}\right],\text{ and}&
E_4'={}&\left[\begin{array}{l} L_0\transpose\\\hline \fakeht
K_1\transpose\end{array}\right], 
\end{align}\endgroup
where the shape of each component matrix is given by
$$\begin{array}{|l|l|}\hline 
\text{matrix}&\text{shape}\\\hline
Z&2\times 6\\\hline K_1&1\times 5\\\hline K_2&5\times 6\\\hline K_3&6\times 2\\  \hline L_0&1\times 2\\\hline
L_1&5\times 6\\\hline\end{array}\ .$$

In Lemma~\ref{Ziszero} we prove that the matrix $Z$ (which appears in both $E_2'$ and $E_3'$)  is identically zero. It follows immediately that
\begin{equation}
\label{moc}
\xymatrix{
0\ar[r]&P^1\ar[r]^{K_1\transpose}\ar[d]^{-L_0\transpose}&P^5\ar[r]^{K_2\transpose}\ar[d]^{ L_1\transpose}&P^6\ar[r]^{K_3\transpose}\ar[d]^{- L_1}&P^2\ar[d]^{L_0}\\
0\ar[r]&P^2\ar[r]^{K_3}&P^6\ar[r]^{K_2}&P^5\ar[r]^{K_1}&P}\end{equation}
is a map of complexes and $E'$ is isomorphic to the  mapping cone  of (\ref{moc}). In Proposition~\ref{perfect} we prove that the bottom row (\ref{moc}) is a minimal homogeneous resolution and the image of $K_1$ is a perfect ideal of grade three. 
It follows immediately that the top row of (\ref{moc}) is the minimal homogeneous resolution of $\Ext^3_P(P/\im K_1,P)$. The fact that the mapping cone of (\ref{moc}) (which is isomorphic to $E'$) is acyclic guarantees that the map from the top row of (\ref{moc}) to the bottom row of (\ref{moc}) induces an isomorphism on homology. In particular, the induced complex  
$$0\to \Ext^3_P(P/\im K_1,P)\to P/\im K_1\to \AA\to 0$$ 
is exact and  
the defining ideal  of $\AA$ is the sum of two linked codimension three perfect ideals; see Example~\ref{GL}. \end{proof}

The complex $E'$ is obtained by rearranging the rows and columns of the matrices of the complex  $E$ of Theorem~\ref{7.1}.
\begin{definition}
\label{11.4}
Adopt the data of Theorem~{\rm\ref{7.2}}.
Let $E$  be the minimal homogeneous  resolution of $\AA$ which is given Theorem~{\rm\ref{7.1}}. We define a new minimal homogeneous resolution 
$$E':\quad 0\to P\xrightarrow{E_4'} P^7 \xrightarrow{E_3'} P^{12}\xrightarrow{E_2'} P^7\xrightarrow{E_1'} P.$$
Let $E_1'$ be the matrix obtained from $E_1$ by arranging the columns in the order \begin{equation}\label{5,1}5,1,7,6,2,3,4;\end{equation} let $E_2'$ be the matrix obtained from $E_2$ by arranging the rows in the order (\ref{5,1})
 and arranging the columns in the order \begin{equation}\label{12,3}12,3,4,1,2,5,6,9,10,7,8,11;\end{equation} let $E_3'$ be the matrix obtained from $E_3$ by arranging the rows in the order (\ref{12,3}) 
and arranging the columns in the order \begin{equation}\label{7,6}7,6,2,3,4,5,1;\end{equation} and let $E_4'$ be the matrix obtained from $E_4$ by arranging the rows in the order (\ref{7,6}). 
Partition the matrices $E_i'$  
in the form (\ref{partition}).  

The component matrices are given by
\begingroup\allowdisplaybreaks
\begin{align*}Z={}&\bmatrix
\frac 1{\a}a_{3,8}d_{4,2}-\frac 1{\a}a_{2,8}d_{4,3}+x^2c_{4,8}& d_{4,5}& d_{4,6}& d_{4,1} &d_{4,4}& d_{4,7} \\
      a_{1,8}&                                 0&       0&       0&       0&       0 \endbmatrix,\\
K_1={}&\text{\tiny $\bmatrix w^2-xp^{-1}(w^2\Phi_3)&zw-xp^{-1}(zw\Phi_3)&y^2-xp^{-1}(y^2\Phi_3)&yz-xp^{-1}(yz\Phi_3)&yw-xp^{-1}(yw\Phi_3)\endbmatrix$}\\
K_2={}&\bmatrix
      \frac 1{\a}a_{3,8}d_{6,2}-\frac 1{\a}a_{2,8}d_{6,3}+x^2c_{6,8} &d_{6,5}& d_{6,6}& d_{6,1}& d_{6,4}& d_{6,7} \\
      \frac 1{\a}a_{3,8}d_{5,2}-\frac 1{\a}a_{2,8}d_{5,3}+x^2c_{5,8}& d_{5,5}& d_{5,6}& d_{5,1}& d_{5,4}& d_{5,7} \\
      \frac 1{\a}a_{3,8}d_{1,2}-\frac 1{\a}a_{2,8}d_{1,3}+x^2c_{1,8}& d_{1,5}& d_{1,6}& d_{1,1}& d_{1,4}& d_{1,7} \\
      \frac 1{\a}a_{3,8}d_{2,2}-\frac 1{\a}a_{2,8}d_{2,3}+x^2c_{2,8}& d_{2,5}& d_{2,6}& d_{2,1}& d_{2,4}& d_{2,7} \\
      \frac 1{\a}a_{3,8}d_{3,2}-\frac 1{\a}a_{2,8}d_{3,3}+x^2c_{3,8}& d_{3,5}& d_{3,6}& d_{3,1}& d_{3,4}& d_{3,7}\\
\endbmatrix,\\
K_3={}&\bmatrix
d_{4,8}                                & 0      \\
\frac 1{\a}a_{3,5}d_{4,2}-\frac 1{\a}a_{2,5}d_{4,3}+x^2c_{4,5}& a_{1,5}\\
\frac 1{\a}a_{3,6}d_{4,2}-\frac 1{\a}a_{2,6}d_{4,3}+x^2c_{4,6}& a_{1,6}\\
\frac 1{\a}a_{3,1}d_{4,2}-\frac 1{\a}a_{2,1}d_{4,3}+x^2c_{4,1} &a_{1,1}\\
\frac 1{\a}a_{3,4}d_{4,2}-\frac 1{\a}a_{2,4}d_{4,3}+x^2c_{4,4}& a_{1,4}\\
\frac 1{\a}a_{3,7}d_{4,2}-\frac 1{\a}a_{2,7}d_{4,3}+x^2c_{4,7}& a_{1,7}\\ 
\endbmatrix,\\
L_0={}&\bmatrix z^2-xp^{-1}(z^2\Phi_3)&x^2p^{-1}(y^*)\endbmatrix,\quad\text
{and} \\
 L_1={}& \bmatrix 
d_{6,8}&h_{6,5}&h_{6,6}&h_{6,1}&h_{6,4}&h_{6,7}\\
d_{5,8}&h_{5,5}&h_{5,6}&h_{5,1}&h_{5,4}&h_{5,7}\\
d_{1,8}&h_{1,5}&h_{1,6}&h_{1,1}&h_{1,4}&h_{1,7}\\
d_{2,8}&h_{2,5}&h_{2,6}&h_{2,1}&h_{2,4}&h_{2,7}\\
d_{3,8}&h_{3,5}&h_{3,6}&h_{3,1}&h_{3,4}&h_{3,7}\endbmatrix,
\end{align*}\endgroup
where $$h_{i,j}=x^2 c_{i,j}+\frac 1{\a}\left|\begin{matrix} 
d_{i,2}&d_{i,3}\\
a_{2,j}&a_{3,j}\end{matrix}\right|.$$
\end{definition}

\begin{definition}
\label{12.3} Adopt the data of Theorem~{\rm\ref{7.2}}.
Let \begingroup\allowdisplaybreaks\begin{align*}
w_1={}&d_{6,6}=[w^*p^{-1}(yw\Phi_3)]\cdot x-y\\
\pi_1={}&d_{3,4}=[y^*p^{-1}(y^2\Phi_3)]\cdot x-[w^*p^{-1}(wy\Phi_3)]\cdot x-y,\\
\pi_2={}&d_{1,7}=[y^*p^{-1}(yz\Phi_3)]\cdot x-z\\  
\pi_3={}&d_{2,1}=[y^*p^{-1}(yw\Phi_3)]\cdot x-w \\
u_{1,1}={}&d_{1,6}=-[y^*p^{-1}(w^2\Phi_3)]\cdot x\\
u_{1,2}={}&d_{2,6}=-[z^*p^{-1}(w^2\Phi_3)]\cdot x\\
u_{1,3}={}&d_{3,6}=[y^*p^{-1}(yw\Phi_3)]\cdot x-[w^*p^{-1}(w^2\Phi_3)]\cdot x+w\\
u_{2,1}={}&d_{1,5}=-[y^*p^{-1}(wz\Phi_3)]\cdot x\\
u_{2,2}={}&d_{2,5}=-[z^*p^{-1}(zw\Phi_3)]\cdot x+w,\text{ and}\\
u_{2,3}={}&d_{3,5}=[y^*p^{-1}(yz\Phi_3)]\cdot x-[w^*p^{-1}(wz\Phi_3)]\cdot x.\\\end{align*}\endgroup
Let $U$ and $\Pi$ be the matrices
$$U=\bmatrix u_{1,1}&u_{1,2}&u_{1,3}\\u_{2,1}&u_{2,2}&u_{2,3}\endbmatrix\quad\text{and}\quad \Pi=\bmatrix \pi_1\\\pi_2\\\pi_3\endbmatrix.$$For $1\le k\le 3$, let $\Delta_k$ be $(-1)^{k+1}$ times the determinant of $U$ with column $k$ deleted.
\end{definition}

\begin{lemma}
\label{Ziszero} Apply the language of Definition~{\rm \ref{12.3}}. The matrices $Z$, $K_2$, and $K_3$ of Definition~{\rm\ref{11.4}} are equal to $Z=0$,
$$K_2=\bmatrix
      \frac 1{\a}a_{3,8}d_{6,2}-\frac 1{\a}a_{2,8}d_{6,3}+c_{6,8}x^2 &0& w_1& 0& 0& 0 \\
      \frac 1{\a}a_{3,8}d_{5,2}-\frac 1{\a}a_{2,8}d_{5,3}+c_{5,8}x^2& w_1& 0& 0& 0& 0 \\
      \frac 1{\a}a_{3,8}d_{1,2}-\frac 1{\a}a_{2,8}d_{1,3}+c_{1,8}x^2& u_{2,1}& u_{1,1}& 0& -\pi_3& \pi_2 \\
      \frac 1{\a}a_{3,8}d_{2,2}-\frac 1{\a}a_{2,8}d_{2,3}+c_{2,8}x^2& u_{2,2}& u_{1,2}& \pi_3& 0& -\pi_1 \\
      \frac 1{\a}a_{3,8}d_{3,2}-\frac 1{\a}a_{2,8}d_{3,3}+c_{3,8}x^2& u_{2,3}& u_{1,3}& -\pi_2& \pi_1& 0\\
\endbmatrix,\text{ and}$$
 $$K_3=\bmatrix
-w_1                                & 0      \\
\frac 1{\a}a_{3,5}d_{4,2}-\frac 1{\a}a_{2,5}d_{4,3}+c_{4,5}x^2& 0\\
\frac 1{\a}a_{3,6}d_{4,2}-\frac 1{\a}a_{2,6}d_{4,3}+c_{4,6}x^2& 0\\
\frac 1{\a}a_{3,1}d_{4,2}-\frac 1{\a}a_{2,1}d_{4,3}+c_{4,1}x^2 &-\pi_1\\
\frac 1{\a}a_{3,4}d_{4,2}-\frac 1{\a}a_{2,4}d_{4,3}+c_{4,4}x^2& -\pi_2\\
\frac 1{\a}a_{3,7}d_{4,2}-\frac 1{\a}a_{2,7}d_{4,3}+c_{4,7}x^2& -\pi_3\\ 
\endbmatrix.$$
\end{lemma}
\begin{proof}
Apply Corollary~\ref{calculation}.
\end{proof}

It follows from Lemma~\ref{Ziszero} that the bottom row of (\ref{moc}) is a complex, which we denote $K$:
\begin{equation}
\label{unknown}K:\quad 0\to P^2\xrightarrow{K_3}P^6\xrightarrow{K_2}P^5\xrightarrow{K_1}P.\end{equation}
We  show  in Proposition~\ref{perfect} that $K$ 
is a (minimal homogeneous) resolution of a perfect module. Indeed we prove that $K$ 
 is a specialization of the resolution of  Anne Brown \cite[Prop.~3.6]{B87}. The precise formulation of the Brown resolution which we use may be found in \cite[Prop.~5.8]{K22} and also in Proposition~\ref{Bres}. 
We now put $K$ into the form of the Brown resolution.

\begin{lemma}
\label{K_1=Br_1} Apply the language of Definition~{\rm \ref{12.3}}.
Let $\Br_1$ be the matrix $$\Br_1=\left[\begin{array}{ccccc}
-(U\Pi)_1& -(U\Pi)_2&  w_1\pi_1& w_1\pi_2& w_1\pi_3\end{array}\right].$$ The matrix $K_1$ of Definition~{\rm\ref{11.4}} satisfies $K_1=\Br_1$. 
\end{lemma}

\begin{proof} 
Use straightforward calculations, including the identities (\ref{i1})-(\ref{i5}) and the fact that if $\ell\in U$, then 
\begin{equation}\notag 
 \ell =[x^*\ell]\cdot x
+[y^*\ell]\cdot z
+[z^*\ell]\cdot y
+[w^*\ell]\cdot w,\end{equation}
to 
  show that $$(K_1)_i-(\Br_1)_i=\blop_i\cdot x^2,$$
for $1\le i\le 5$, with   $\blop_i\in\kk$ as defined below
\begingroup\allowdisplaybreaks
\begin{align*}
\blop_1={}&\begin{cases}
- [y^*p^{-1}(w^2\Phi_3)]\cdot [y^*p^{-1}(y^2\Phi_3)]
+ [y^*p^{-1}(w^2\Phi_3)]\cdot [w^*p^{-1}(wy\Phi_3)]\\
- [z^*p^{-1}(w^2\Phi_3)]\cdot [y^*p^{-1}(yz\Phi_3)]
+ [y^*p^{-1}(yw\Phi_3)]\cdot [y^*p^{-1}(yw\Phi_3)]\\
- [w^*p^{-1}(w^2\Phi_3)]\cdot [y^*p^{-1}(yw\Phi_3)]
-[x^*p^{-1}(w^2\Phi_3)],\\
\end{cases}\\
\blop_2={}&\begin{cases}
-[y^*p^{-1}(wz\Phi_3)]\cdot [y^*p^{-1}(y^2\Phi_3)]
+[y^*p^{-1}(wz\Phi_3)]\cdot [w^*p^{-1}(wy\Phi_3)]\\
-[z^*p^{-1}(zw\Phi_3)]\cdot [y^*p^{-1}(yz\Phi_3)]
+[y^*p^{-1}(yz\Phi_3)]\cdot [y^*p^{-1}(yw\Phi_3)]\\
-[w^*p^{-1}(wz\Phi_3)]\cdot [y^*p^{-1}(yw\Phi_3)]
-[x^*p^{-1}(zw\Phi_3)],\\
\end{cases}\\
\blop_3={}&\begin{cases}
- [w^*p^{-1}(yw\Phi_3)]\cdot [y^*p^{-1}(y^2\Phi_3)]
+ [w^*p^{-1}(yw\Phi_3)]\cdot [w^*p^{-1}(wy\Phi_3)]\\
-[x^*p^{-1}(y^2\Phi_3)],\\
\end{cases}\\
 \blop_4={}& 
-[z^*p^{-1}(yz\Phi_3)]\cdot [y^*p^{-1}(yz\Phi_3)]
-[x^*p^{-1}(yz\Phi_3)]
,\text{ and}\\
\blop_5={}&
-[x^*p^{-1}(yw\Phi_3)]
-[w^*p^{-1}(yw\Phi_3)]\cdot [y^*p^{-1}(yw\Phi_3)].\\
\end{align*}
\endgroup

The identities of Lemma~\ref{11.3}:
\text{(\ref{20.3.5})} minus \text{(\ref{22.0.14})}, \text{(\ref{20.3.3})} minus \text{(\ref{22.0.9})}, \text{(\ref{22.0.21})}, \text{(\ref{22.0.25})}, \text{and}  
\text{(\ref{22.0.19})} show that $\blop_1$, $\blop_2$, $\blop_3$, $\blop_4$, and  $\blop_5$ are zero, respectively.\end{proof}
\begin{proposition}
\label{12.6} Apply the language of Definition~{\rm \ref{12.3}}. Let $\theta$ be the invertible matrix
 $$\theta=
\bmatrix 
1&0&0&0&0&0\\ 
\frac{1}{\a}[y^*p^{-1}(w^2\Phi_3)]\cdot x&1&0&0&0&0\\[3pt] 
\frac {-1}{\a} [y^*p^{-1}(zw\Phi_3)]\cdot x&0&1&0&0&0\\[3pt]
0&0&0&1&0&0\\
0&0&0&0&1&0\\
0&0&0&0&0&1
\endbmatrix,$$
$K'$ be the complex
$$0\to P^2\xrightarrow{K_3'}P^6\xrightarrow{K_2'}P^5\xrightarrow{K_1} P,$$
where $K_2'=K_2\theta$ and $K_3'=\theta^{-1}K_3$, and $\Br_3$ and $\Br_2$ be the matrices
$$\Br_2=\bmatrix 
\frac 1{\a}(U\Pi)_2&0&w_1&0&0&0\\
-\frac1{\a}(U\Pi)_1&w_1&0&0&0&0\\
0&u_{2,1}&u_{1,1}&0&-\pi_3&\pi_2\\
0&u_{2,2}&u_{1,2}&\pi_3&0&-\pi_1\\
0&u_{2,3}&u_{1,3}&-\pi_2&\pi_1&0\endbmatrix\ \text{and}$$
$$\Br_3=\bmatrix -w_1&0\\-\frac1{\a}(U\Pi)_1&0\\\frac1{\a}(U\Pi)_2&0\\-\frac1{\a}\Delta_1&-\pi_1\\-\frac1{\a}\Delta_2&-\pi_2\\-\frac1{\a}\Delta_3&-\pi_3\endbmatrix.$$
Then $K_3'=\Br_3$ and $K_2'=\Br_2$.
\end{proposition}

\begin{proof}
In light of Lemma~\ref{Ziszero}, it suffices to study column one of each matrix.
We compute that $\a(K_2')_{1,1}$ is equal to
$$\begin{cases}
+[w^2p^{-1}(z^*)](\Phi_3)\cdot [w^*p^{-1}(z^2\Phi_3)]\cdot x^2\\
-[zwp^{-1}(z^*)](\Phi_3)\cdot [w^*p^{-1}(zw\Phi_3)]\cdot x^2\\
+[w^*p^{-1}(zw\Phi_3)]\cdot xw
+[zwp^{-1}(z^*)](\Phi_3)\cdot xz\\
-zw
-\a [w^*p^{-1}(z^*)]\cdot x^2\\
-[w^*p^{-1}(yw\Phi_3)]\cdot [y^* p^{-1}(zw\Phi_3)]\cdot x^2
+[y^* p^{-1}(zw\Phi_3)]\cdot xy\\
\end{cases}$$ 
and $\a(\Br_2)_{1,1}$ is equal to
$$ \begin{cases}
-[y^*p^{-1}(wz\Phi_3)]\cdot [y^*p^{-1}(y^2\Phi_3)]\cdot x^2
+[y^*p^{-1}(wz\Phi_3)]\cdot [w^*p^{-1}(wy\Phi_3)]\cdot x^2\\
-[z^*p^{-1}(zw\Phi_3)]\cdot [y^*p^{-1}(yz\Phi_3)]\cdot x^2
-zw\\
+[y^*p^{-1}(yz\Phi_3)]\cdot [y^*p^{-1}(yw\Phi_3)]\cdot x^2
-[w^*p^{-1}(wz\Phi_3)]\cdot [y^*p^{-1}(yw\Phi_3)]\cdot x^2\\
+[y^*p^{-1}(wz\Phi_3)]\cdot xy
+[z^*p^{-1}(zw\Phi_3)]\cdot xz\\
+[w^*p^{-1}(wz\Phi_3)]\cdot xw.
\end{cases}$$
Thus, $\a(K_2')_{1,1}- \a(\Br_2)_{1,1}=\blop_7\cdot x^2$ for 
\begin{equation}\notag 
\blop_7=\begin{cases}
+[w^2p^{-1}(z^*)](\Phi_3)\cdot [w^*p^{-1}(z^2\Phi_3)]
-[zwp^{-1}(z^*)](\Phi_3)\cdot [w^*p^{-1}(zw\Phi_3)]\\
-\a [w^*p^{-1}(z^*)]
+ [y^*p^{-1}(wz\Phi_3)]\cdot [y^*p^{-1}(y^2\Phi_3)]\\
+ [z^*p^{-1}(zw\Phi_3)]\cdot [y^*p^{-1}(yz\Phi_3)]
- [y^*p^{-1}(yz\Phi_3)]\cdot [y^*p^{-1}(yw\Phi_3)]\\
+ [w^*p^{-1}(wz\Phi_3)]\cdot [y^*p^{-1}(yw\Phi_3)]
-2[w^*p^{-1}(yw\Phi_3)]\cdot [y^* p^{-1}(zw\Phi_3)].
\end{cases}\end{equation}
Apply Lemma~\ref{11.3}.(\ref{20.3.3}) to see that  $\blop_7=0$. 

In a similar manner, one computes that $$\a(K_2')_{2,1}-\a(\Br_2)_{2,1}=\blop_8\cdot x^2\quad \text{and}\quad \a(K_2')_{5,1}=\blop_9\cdot x^2.$$ 
The constants  $\blop_8$ and $\blop_9$ are zero by items   
(\ref{20.3.5}) and (\ref{12.2.h}), respectively, of Lemma~\ref{11.3} 
The entries $\a(K_2')_{3,1}$ and $\a(K_2')_{4,1}$ are obviously zero.

One computes that 
\begin{align*}
&\a (K_3')_{2,1}-\a(\Br_3)_{2,1}= 
\blop_{10} \cdot x^2,\\
&\a(K_3')_{3,1}-\a(\Br_3)_{3,1}=\blop_{11} \cdot x^2,\\
&\a(K_3')_{4,1}-\a(\Br_3)_{4,1}=0,\\
&\a(K_3')_{5,1}-\a(\Br_3)_{5,1}=\blop_{12}\cdot x^2,\text{ and}\\
&\a(K_3')_{6,1}-\a(\Br_3)_{6,1}=0.
\end{align*}
Apply items  (\ref{20.3.5}), (\ref{20.3.3}), and (\ref{12.2.h}) of  Lemma~\ref{11.3} to see that $\blop_{10}=0$,  $\blop_{11}=0$, and 
$\blop_{12}=0$, respectively. 
\end{proof}

\begin{observation} 
\label{12.7}
The complex $K$ of {\rm(\ref{unknown})} is isomorphic to
$$0\to P^2\xrightarrow{B_3}P^6\xrightarrow{B_2}P^5\xrightarrow{B_1}P,$$
for
$$B_1=\left[\begin{array}{ccccc}
(U\Pi)_1& (U\Pi)_2&  -w_1\pi_1& -w_1\pi_2& -w_1\pi_3\end{array}\right].$$$$B_2=\bmatrix 
(U\Pi)_2&0&-w_1&0&0&0\\
-(U\Pi)_1&w_1&0&0&0&0\\
0&u_{2,1}&-u_{1,1}&0&-\pi_3&\pi_2\\
0&u_{2,2}&-u_{1,2}&\pi_3&0&-\pi_1\\
0&u_{2,3}&-u_{1,3}&-\pi_2&\pi_1&0\endbmatrix\ \text{and}$$
$$B_3=\bmatrix w_1&0\\(U\Pi)_1&0\\(U\Pi)_2&0\\\Delta_1&\pi_1\\\Delta_2&\pi_2\\\Delta_3&\pi_3\endbmatrix.$$
\end{observation}
\begin{proof} The complex $K$ is isomorphic to the complex $K'$ of 
Proposition~\ref{12.6}. Furthermore,
$$\xymatrix{
0\ar[r]& P^2\ar[r]^{K_3'}\ar[d]^{\beta_3}&P^6\ar[r]^{K_2'}\ar[d]^{\beta_2}&P^5\ar[r]^{K_1}\ar[d]^{1}&P\ar[d]^{-1}\\
0\ar[r]& P^2\ar[r]^{B_3}&P^6\ar[r]^{B_2}&P^5\ar[r]^{B_1}&P,\\
}$$
with
$$\beta_3=\bmatrix -\frac1{\a}&0\\0&-1\endbmatrix\quad\text{and}\quad \beta_2=\bmatrix
\frac1{\a}&0&0&0&0&0\\
0&1&0&0&0&0\\
0&0&-1&0&0&0\\
0&0&0&1&0&0\\
0&0&0&0&1&0\\
0&0&0&0&0&1\endbmatrix,$$
is an isomorphism of complexes.
\end{proof}

\begin{proposition}
\label{Bres}  
Let $P$ be a commutative Noetherian ring, and let $w_1$,  $z_2$, 
$u_{i,j}$ and  $\pi_j$, with $1\le i\le 2$ and $1\le j\le 3$, be elements of $P$. Let $U$ and $\Pi$ be the matrices
$$U=\bmatrix u_{1,1}&u_{1,2}&u_{1,3}\\u_{2,1}&u_{2,2}&u_{2,3}\endbmatrix\quad\text{and}\quad \Pi=\bmatrix \pi_1\\\pi_2\\\pi_3\endbmatrix.$$For $1\le k\le 3$, let $\Delta_k$ be $(-1)^{k+1}$ times the determinant of $U$ with column $k$ deleted.
Let $\B$ be the maps and modules 
$$0\to P^2\xrightarrow{b_3} P^6\xrightarrow{b_2}P^5\xrightarrow{b_1} P,$$
where

\begingroup\allowdisplaybreaks
\begin{align*}
b_3={}&\left[\begin{array}{cc}
w_1&z_2\\
(U\Pi)_1&0\\
(U\Pi)_2&0\\
\Delta_1&\pi_1\\
\Delta_2&\pi_2\\
\Delta_3&\pi_3\\\end{array}\right]\ ,\\
b_2 ={}&\left[\begin{array}{cccccc}
(U\Pi)_2&0&-w_1&-z_2u_{21}&-z_2u_{22}&-z_2u_{23}\\
-(U\Pi)_1&w_1&0&z_2u_{11}&z_2u_{12}&z_2u_{13}\\
0&u_{21}&-u_{11}&0&-\pi_3&\pi_2\\
0&u_{22}&-u_{12}&\pi_3&0&-\pi_1\\
0&u_{23}&-u_{13}&-\pi_2&\pi_1&0\\
\end{array}\right]\ , 
\text{ and}\\ 
b_1={}&\left[\begin{array}{ccccc}
(U\Pi)_1& (U\Pi)_2&  z_2 \Delta_1-w_1\pi_1& 
z_2 \Delta_2-w_1\pi_2& 
z_2\Delta_3-w_1\pi_3\end{array}\right]\ \ .\end{align*}\endgroup
Then $\B$ is a complex and if $3\le \grade I_1(b_1)$, then $\B$ is a resolution of a perfect $P$-module.\end{proposition}
 \begin{proof}See 
\cite[Prop.~3.6]{B87} or  
\cite[Prop.~5.8]{K22}. \end{proof} 
\begin{proposition}
\label{perfect} Adopt the hypotheses of Theorem~{\rm\ref{7.2}}. 
Then the bottom row of {\rm(\ref{moc})} is a homogeneous minimal resolution of the perfect module  $P/\im K_1$.
\end{proposition}
\begin{proof}
Observation~\ref{12.7} shows that the complex $K$ of (\ref{unknown}), which is the bottom row of (\ref{moc}), is the Anne Brown complex of Proposition~\ref{Bres} with $z_2=0$.
 In light of Proposition~\ref{Bres}, it suffices to show that $3\le \grade \im K_1$. The ideal $K_1$ is homogeneous; so it suffices to prove that $3\le \grade (\im K_1)_\maxm$, where $\maxm$ is the maximal homogeneous  ideal of $P$. The ideal $\im K_1+(x)$ contains $\im K_1+(xp^{-1}(y^*),xp^{-1}(z^*))$, which is equal to 
 the defining ideal of $\AA$. The defining ideal of $\AA$ 
 has grade four; hence $4\le \grade(\im K_1+(x))_\maxm$ and $3\le \grade (\im K_1)_\maxm$.  \end{proof}

\section{The case when the rank of $\SM$ is at least two.}

This section is devoted to the proof of the following Theorem. 

\begin{theorem}
\label{6.1} Adopt the data of {\rm \ref{2.1}} and the matrix $\SM$ of {\rm(\ref{SM})}. If $\SM$ has rank at least two, then $\AA$ is a hypersurface section in the sense that the defining ideal of $\AA$ is equal to $(f)+J$ for some homogeneous element $f$ of $P$ and some homogeneous ideal $J$ of $P$ with 
\begin{enumerate}[\rm (a)] 
\item\label{26.1.b} $P/J$ is Gorenstein ring of codimension  three, and
\item\label{26.1.c} $f$ is regular on $P/J$. 
\end{enumerate}
\end{theorem}
\begin{Remark}
Theorem~\ref{6.1} is implemented in the Macaulay2 script ``minres6''; see Section~\ref{14.A.5}. \end{Remark}
The proof of Theorem~\ref{6.1} is given in \ref{Proof-of-step-1}. The explicit matrices in the resolution of $\AA$ are also given in \ref{Proof-of-step-1}. Furthermore, it is shown in \ref{Proof-of-step-1} 
 that the defining ideal of $\AA$ is the generalized sum of two linked codimension three perfect ideals. 

Some of the tools of Section~\ref{section SM} are used in the proof of Theorem~\ref{6.1}. In particular, recall the homomorphism $$\Gamma:U_0^*\to U_0$$ of Definition~\ref{8.1}, the complex  
\begin{equation}
\label{13.0.1}
\mathfrak C:\quad 0\to \kk \xrightarrow{\ \ \ \ \mathfrak C_3\ \ \ } D_2U_0^*\xrightarrow{\ \ \ \mathfrak C_2\ \ \ } \frac{U_0\t U_0^*}{\ev^*(1)}\xrightarrow{\ \ \  \mathfrak C_1\ \ \ } {\ts\bw^2}U_0\to 0\end{equation}
from (\ref{6.4.3}), and the element $\delta$ of $D_2U_0^*$ with \begin{equation}\label{13.0.3/2}\Gamma^{\vee}(u_1)=u_1(\delta),\quad\text{for $u_1\in U_0$},\end{equation} from Definition~\ref{24.2}. (See (\ref{delta-mat}) for a concrete version of $\delta$.)   Recall, also,  from Observation~\ref{8.2} and Lemma~\ref{24.4}, that the hypothesis that $2\le \rank \SM$ guarantees that
\begin{equation}\label{kk}I_2(\SM)=I_1(\mathfrak C_3)=I_5(\mathfrak C_2)=I_3(\mathfrak C_1)=I_3(\mB)=\kk.\end{equation}
\begin{observation}
\label{13.2}
 The hypotheses of Theorem~{\rm\ref{6.1}} guarantee that the minimal homogeneous resolution of $\AA$ has the following graded Betti numbers{\rm:}
\begin{equation}
\label{12.2.1}0\to P(-7)\to P(-5)^6\to \begin{matrix} P(-4)^5\\\p\\P(-3)^5\end{matrix} \to P(-2)^6\to P.\end{equation} Furthermore, the linear syzygies of the defining ideal of $\AA$ are given by the restriction of $\mD$ to the kernel of $\mB$.\end{observation}
\begin{Remark}
The maps $\mf_{1,2}$, $\mB$, and $\mD$ from Theorem~\ref{a res} all appear in the statement or  proof of Observation~\ref{13.2}.
\end{Remark}
\begin{proof}
 Recall
 from (\ref{kk}) that 
$I_3(\mB)=\kk$. 
It follows that the non-minimal resolution 
$\left(\frak F, \widetilde {\frak f}\,\,\right)$ of $\AA$, which is given in Theorem~\ref{a res}, has the form 
$$0\to P(-7)\to \begin{matrix} P(-4)^3\\\p\\P(-5)^6\end{matrix}
\xrightarrow{\bmatrix I_3&*\\ *&*\endbmatrix} \begin{matrix} P(-4)^8\\\p\\P(-3)^8\end{matrix}\xrightarrow{\bmatrix *&I_3\\ *&*\endbmatrix} 
\begin{matrix} P(-3)^3\\\p\\P(-2)^6\end{matrix}\to P.$$
One may split off a split exact subcomplex to obtain a complex with graded Betti numbers given in (\ref{12.2.1}).

 The linear relations on $\mf_{1,2}$ are induced by 
\begin{equation}\label{Dec1}\frac{U_0\t U_0^*}{\ev^*(1)} \xrightarrow{\ \ \ \bmatrix \mB\\\mD\endbmatrix\ \ \ } 
\begin{matrix} P\t U_0^*\\\p\\ P\t \Sym_2U_0.\end{matrix}\end{equation}
The map $\mB$ is a surjection; let $\sigma:P\t U_0^*\to \frac{U_0\t U_0^*}{\ev^*(1)}$ be a splitting map in the sense that $\mB\circ \sigma$ is the identity map on $P\t U_0^*$.
Once the isomorphism $$\im \sigma\xrightarrow{\mB|_{\im\sigma}} P\t U_0^*$$ is split from (\ref{Dec1}) one is left with 
$$\ker \mB\xrightarrow{\mD|_{\ker \mB} } P\t \Sym_2U_0.$$ 
\end{proof}

Observation~\ref{Oct5} describes our strategy for proving Theorem~\ref{6.1}. We look for a matrix $X$ and a generating set $g_1,\dots,g_6$ for the defining ideal of $\AA$ such that the hypotheses of Observation~\ref{Oct5} are satisfied.
If the matrix $X$ happens to be an alternating matrix, then we are finished.

\begin{observation}
\label{Oct5} 
Adopt the hypotheses of  
Theorem~{\rm\ref{6.1}}. Suppose that 
 $g_1,\dots, g_6$ in $P$ is a minimal  homogeneous generating set for the defining ideal of $\AA$ and $X$ is a $5\times 5$ matrix $X$ of linear forms from $P$, with linearly independent columns, such that $$\bmatrix g_1&g_2&g_3&g_4&g_5\endbmatrix X=0.$$
Let $\pmb g$  and $d_1$ be the matrices $$\pmb g=\bmatrix g_1&g_2&g_3&g_4&g_5& g_6\endbmatrix\quad\text{and}\quad d_1=\bmatrix g_1&g_2&g_3&g_4&g_5\endbmatrix,$$
and $J$ be the ideal $J=(g_1,g_2,g_3,g_4,g_5)$ of $P$.
Then the following statements hold.
\begin{enumerate}[\rm(a)]
\item\label{Oct5.a} The columns of
\begin{equation}\label{above2}\left[\begin{array}{c|c}
X&-g_6I_5\\\hline 0_{1\times 5}&d_1\end{array}\right],\end{equation}
where $I_5$ is a $5\times 5$ identity matrix,  
are  a 
minimal generating set for the module of 
  syzygies of $\pmb g$.
\item\label{Oct5.b} The element $g_6$ is regular on $P/J$.
\item\label{Oct5.c} The ideal $J$  is a grade three Gorenstein ideal of $P$.
\item\label{Oct5.d} If $X$ is an alternating matrix, then the following statements also hold.
\begin{enumerate}[\rm (i)]
\item\label{Oct5.di} The complex 
$$0\to P\xrightarrow{d_1\transpose} P^5 \xrightarrow{X} P^5\xrightarrow{d_1} P$$ is a minimal homogeneous resolution of $P/J$. 
\item\label{Oct5.dii} The mapping cone of
$$\xymatrix{0\ar[r]& P\ar[r]^{d_1\transpose}\ar[d]^{g_6}& P^5 \ar[r]^{X}\ar[d]^{g_6}& P^5\ar[r]^{d_1}\ar[d]^{g_6}& P\ar[d]^{g_6}\\
0\ar[r]& P\ar[r]^{d_1\transpose}& P^5 \ar[r]^{X}& P^5\ar[r]^{d_1}& P}$$
is a minimal homogeneous resolution of $\AA$.
\item\label{Oct5.diii} The complex 
$$0\to P\xrightarrow{\pmb g\transpose} P^6\xrightarrow{\mathfrak g_3}P^{10}\xrightarrow{\mathfrak g_2} P^6\xrightarrow{\pmb g}P$$
is a minimal homogeneous resolution of $\AA$, where $\mathfrak g_2$ is given in {\rm(\ref{above2})}, and  $$\mathfrak g_3=\bmatrix -g_6&
d_1\transpose\\-X&0 \endbmatrix.$$
 \end{enumerate}
\end{enumerate}
\end{observation}

\begin{proof}
(\ref{Oct5.a}) Recall that the graded Betti numbers of $\AA$ are given in (\ref{12.2.1}). 
The columns of \begin{equation}\label{above}\left[\begin{array}{ccccc} X\\\hline 0_{1\times 5}\end{array}\right]\end{equation}  
are the beginning of a minimal generating set for the 
linear syzygies of $\pmb g$. It follows from (\ref{12.2.1}) that the columns of (\ref{above}) 
generate all of 
the  linear syzygies of $\pmb g$. In a similar manner, the columns of (\ref{above2})
are the beginning of a minimal generating set for the module of  syzygies of $\pmb g$; hence, it follows from (\ref{12.2.1}) that the columns of (\ref{above2}) are a 
minimal generating set for the module of all  syzygies of $\pmb g$. 

\noindent(\ref{Oct5.b}) The bottom row of (\ref{above2}) generates the ideal $(g_1,\dots,g_5):_P(g_6)$; hence $g_6$ is regular on $P/J$. 

\noindent(\ref{Oct5.c}) If $\mathfrak p$ is any prime ideal of $P$ which contains $I_1(\pmb g)$, then the element $g_6$ is regular on the local ring $(P/J)_{\mathfrak p}$; hence $(P/J)_{\mathfrak p}$ is Gorenstein if and only if 
$(P/(J,g_6))_{\mathfrak p}$ is Gorenstein.

\noindent(\ref{Oct5.di}) We saw in (\ref{Oct5.a}) that 
$$P^5\xrightarrow{X}P^5\xrightarrow{d_1}P$$ is the beginning of a minimal resolution of $P/J$. We saw in (\ref{Oct5.c}) that $J$ is a grade three Gorenstein ideal. It follows that the Betti numbers of $P/J$ are 
$$0\to P^1\to P^5\to P^5\to P.$$ Let $\pmb y$ be a row vector with 
\begin{equation}\label{13.3.3}0\to P^1\xrightarrow{\pmb y\transpose} P^5\xrightarrow{X} P^5\xrightarrow{d_1} P\end{equation}exact. The matrix $X$ is an alternating matrix; hence, $Xd_1\transpose=0$ and $d_1\transpose=\a \pmb y\transpose$ for some $\a\in P$.
The quotient $P/J$ is Gorenstein; so when $\Hom_P(-,P)$ is applied to (\ref{13.3.3}) one obtains another resolution of $P/J$. It follows that $\a$ is a unit.  

\noindent Items (\ref{Oct5.dii}) and (\ref{Oct5.diii}) now follow immediately.
\end{proof}

The map ``$\mL$'' of Definition~\ref{Gamma} will lead us to the matrix ``$X$'' of Observation~\ref{Oct5}.

\begin{definition}
\label{Gamma} Adopt the hypotheses of Lemma~{\rm\ref{6.1}}.
Define the $P$-module homomorphism 
$$\mL:P\t_{\kk} D_2U_0^*\to P\t_{\kk} \Sym_2U_0$$ to be the composition 
$$P\t_{\kk} D_2U_0^*\xrightarrow{\C_2}\frac{U_0\t U_0^*}{\ev^*(1)}
\xrightarrow{\mD}\Sym_2U_0,$$
 where  $\mD$ is given in Theorem~\ref{resol} and 
$\C_2$ is given in (\ref{13.0.1}) and (\ref{6.4.3}).
\end{definition}

\begin{lemma}
\label{Indeed}  
If $\nu_{1,0}$ and $\nu_{1,0}'$ are  in $U_0^*$, then the map $\mL$ of Definition~{\rm\ref{Gamma}} satisfies  
$$[\mL(\nu_{1,0}^{(2)})]({\nu_{1,0}'}^{(2)})=\begin{cases}
\phantom{+}(\nu_{1,0}\w \nu_{1,0}')(\omega_{U_0})
\otimes   [\Gamma(\nu_{1,0})](\nu_{1,0}')\\
-
 x \otimes   \Big[p^{-1} \big(
\big[(\nu_{1,0}\w \nu_{1,0}')(\omega_{U_0})\cdot 
\Gamma(\nu_{1,0})\big](\Phi_3)\big)\Big](\nu_{1,0}').\\
\end{cases}$$ 
\end{lemma}

\begin{proof} We compute 
\begingroup\allowdisplaybreaks
\begin{align*}
&[\mL(\nu_{1,0}^{(2)})]({\nu_{1,0}'}^{(2)})\\
{}={}&[(\mD\circ \C_2)(\nu_{1,0}^{(2)})](\nu'_{1,0}{}^{(2)})\\
{}={}&[\mD(\Gamma(\nu_{1,0})\t \nu_{1,0})](\nu'_{1,0}{}^{(2)})\\
{}={}&\left(\begin{cases}
-
 x\t w\cdot (\proj\circ p^{-1})\Big([\nu_{1,0}(y\w z)\Gamma(\nu_{1,0})](\Phi_3)\Big)\\
+x\t z\cdot (\proj\circ p^{-1})\Big([\nu_{1,0}(y\w w)\Gamma(\nu_{1,0})](\Phi_3)\Big)\\
-x\t y\cdot (\proj\circ p^{-1})\Big([\nu_{1,0}(z\w w)\Gamma(\nu_{1,0})](\Phi_3)\Big)\\
-w\t \nu_{1,0}(y\w z)\Gamma(\nu_{1,0})\\
+z\t \nu_{1,0}(y\w w)\Gamma(\nu_{1,0})\\
-y\t \nu_{1,0}(z\w w)\Gamma(\nu_{1,0})\\
\end{cases}\right)(\nu'_{1,0}{}^{(2)}),\\\intertext{by (\ref{mD}),}
{}={}&\begin{cases}
-x\t w(\nu'_{1,0})\cdot (\nu'_{1,0})(\proj\circ p^{-1})\Big([\nu_{1,0}(y\w z)\Gamma(\nu_{1,0})](\Phi_3)\Big)\\
+x\t z(\nu'_{1,0})\cdot (\nu'_{1,0})(\proj\circ p^{-1})\Big([\nu_{1,0}(y\w w)\Gamma(\nu_{1,0})](\Phi_3)\Big)\\
-x\t y(\nu'_{1,0})\cdot (\nu'_{1,0})(\proj\circ p^{-1})\Big([\nu_{1,0}(z\w w)\Gamma(\nu_{1,0})](\Phi_3)\Big)\\
-w\t (\nu'_{1,0}\w\nu_{1,0})(y\w z)[\Gamma(\nu_{1,0})](\nu'_{1,0})\\
+z\t (\nu'_{1,0}\w\nu_{1,0})(y\w w)[\Gamma(\nu_{1,0})](\nu'_{1,0})\\
-y\t (\nu'_{1,0}\w\nu_{1,0})(z\w w)[\Gamma(\nu_{1,0})](\nu'_{1,0})\\
\end{cases}\\
{}={}&\begin{cases}
-x\t w(\nu'_{1,0})\cdot p^{-1}\Big([\nu_{1,0}(y\w z)\Gamma(\nu_{1,0})](\Phi_3)\Big)(\nu'_{1,0})\\
+x\t z(\nu'_{1,0})\cdot p^{-1}\Big([\nu_{1,0}(y\w w)\Gamma(\nu_{1,0})](\Phi_3)\Big)(\nu'_{1,0})\\
-x\t y(\nu'_{1,0})\cdot  p^{-1}\Big([\nu_{1,0}(z\w w)\Gamma(\nu_{1,0})](\Phi_3)\Big)(\nu'_{1,0})\\
-w\t (\nu'_{1,0}\w\nu_{1,0})(y\w z)[\Gamma(\nu_{1,0})](\nu'_{1,0})\\
+z\t (\nu'_{1,0}\w\nu_{1,0})(y\w w)[\Gamma(\nu_{1,0})](\nu'_{1,0})\\
-y\t (\nu'_{1,0}\w\nu_{1,0})(z\w w)[\Gamma(\nu_{1,0})](\nu'_{1,0})\\
\end{cases}\\
{}={}&\begin{cases}
-
x\t   p^{-1}\Big([(\nu_{1,0}\w \nu'_{1,0})(y\w z\w w)\Gamma(\nu_{1,0})](\Phi_3)\Big)(\nu'_{1,0})\\
-(\nu_{1,0}'\w \nu_{1,0})(y\w z\w w)\t [\Gamma(\nu_{1,0})](\nu'_{1,0})
\end{cases}\\
{}={}&\begin{cases}
-
x\t   p^{-1}\Big([(\nu_{1,0}\w \nu'_{1,0})(\omega_{U_0})\Gamma(\nu_{1,0})](\Phi_3)\Big)(\nu'_{1,0})\\
-(\nu_{1,0}'\w \nu_{1,0})(\omega_{U_0})\t [\Gamma(\nu_{1,0})](\nu'_{1,0})
\end{cases}\\
{}={}&\begin{cases}
-x\t   \Big[p^{-1}\Big([(\nu_{1,0}\w \nu'_{1,0})(\omega_{U_0})\cdot \Gamma(\nu_{1,0})](\Phi_3)\Big)\Big](\nu'_{1,0})\\
+(\nu_{1,0}\w \nu_{1,0}')(\omega_{U_0})\t [\Gamma(\nu_{1,0})](\nu'_{1,0}).
\end{cases} 
\end{align*}
\endgroup
\vskip-24pt\end{proof}

\begin{lemma}
\label{26.8} Adopt the hypotheses of Theorem~{\rm\ref{6.1}}. The following statements hold. 
\begin{enumerate}[\rm(a)]
\item\label{26.8.b} 
The maps
$$ 0\to P\xrightarrow{\ \ \C_3\ \ } P\t_{\kk}D_2U_0^*\xrightarrow{\ \ \mL\ \ }P\t_{\kk}\Sym_2U_0\xrightarrow{\ \ \mf_{1,2}\ \ } P\to \AA\to 0$$ 
form a  complex.
\item\label{26.8.c} 
The homomorphism $\Gamma$, which  is a linear transformation of vector spaces, has rank at least two.
\item\label{26.8.c.5} 
The complex $\C$ of {\rm(\ref{13.0.1})} 
 is split exact.

\item\label{26.8.d} 
The image of $\mL$ contains all of the linear relations on $\mf_{1,2}$.
\item\label{26.8.e} 
The element $\delta $ of  
 $D_2U^*_0$ of {\rm(\ref{13.0.3/2})}
is nonzero and is in the kernel of $\mL$.
If $\xi$ is any element of $\Sym_2U_0$ with $\delta (\xi)=1$, then 
the 
 induced map
$$P\t_{\kk}
\ker \xi
\xrightarrow{\ \ \mL\ \ } P\t_{\kk}\Sym_2U_0$$
maps in a minimal manner onto the linear relations on $ \mf_{1,2}$.
\end{enumerate}
\end{lemma}

\begin{proof}
 Assertions (\ref{26.8.c}) and (\ref{26.8.c.5}) follow from (\ref{kk}) and Lemma~\ref{24.4}.(\ref{24.4.b}). (Keep in mind that $\SM$ is the matrix for $\Gamma$.)
 Recall from Observation~\ref{13.2} that the linear relations on $\mf_{1,2}$ are given by the image of 
\begin{equation}\label{3.9.1}\mD\Big|_{\ker(\mB)}.\end{equation}  Apply (\ref{B=C(Gamma)sub1!}), followed by (\ref{26.8.c.5}), to see that 
$$\ker \mB=\ker \C_1=\im \C_2.$$
Combine the most recent equation with (\ref{3.9.1}) and Lemma~\ref{Indeed} to conclude that the linear relations on $\mf_{1,2}$ are given by $$\im \big(\mD\circ \C_2\big)=\im \mL.$$
Assertion~(\ref{26.8.d}) is established. It is clear from the equation 
$\mL=\mD\circ   \C_2$ that $\mL\circ \C_3=0$; hence (\ref{26.8.b}) is established.

The complex $\C$ is  split exact; so it is clear that the image of $\C_3$ is a free $\kk$-module of rank $1$. This image is generated by the element $\delta \in D_2U_0^*$ of (\ref{13.0.3/2}). 
 If $\xi$ is any element of $\Sym_2U_0$ with $\delta (\xi)=1$, then  $$D_2U^*_0=\ker \xi \p \kk \delta $$and 
$\ker \xi$ is a free $\kk$-module of rank $5$. 
Combine (\ref{26.8.d})  
and  the fact that $\mL(\delta )=0$  
in order to conclude  that
the image of 
\begin{equation}\label{3.11.2}P\t_{\kk}\ker \xi \xrightarrow{\ \ \mL\ \ } P\t_{\kk}\Sym_2U_0\end{equation} 
contains all of the linear syzygies on $\mf_{1,2}$. The map (\ref{3.11.2}) is minimal because $\mf_{1,2}$ has a five-dimensional vector space of linear syzygies; see   
(\ref{12.2.1}).  Assertion (\ref{26.8.e}) is established and the proof is complete.
\end{proof}

As promised in Observation~\ref{Oct5}, the rest of the argument proceeds quickly once Lemma~\ref{alternating} is established.

\begin{lemma}
\label{alternating}
The $P$-module homomorphism $\mL:P\t_{\kk} D_2U_0^*\to P\t_{\kk} \Sym_2U_0$ 
of Definition~{\rm\ref{Gamma}} is an alternating map.
\end{lemma}

\begin{proof} 
Fix arbitrary elements $\nu_{1,0}$ and $\nu_{1,0}'$ in $U_0^*$. It suffices to show that $$[\mL(\nu_{1,0}^{(2)})]({\nu_{1,0}'}^{(2)})+[\mL({\nu_{1,0}'}^{(2)})](\nu_{1,0}^{(2)})=0.$$ Apply Lemma~\ref{Indeed}. One quickly sees that it suffices to show that 
\begin{equation}
\label{STS1}
0=\begin{cases}\phantom{+}\Big[p^{-1} \big(
\big[(\nu_{1,0}\w \nu_{1,0}')(\omega_{U_0})\cdot 
\Gamma(\nu_{1,0})\big](\Phi_3)\big)\Big](\nu_{1,0}')\vspace{5pt}\\+ \Big[p^{-1} \big(
\big[(\nu_{1,0}'\w \nu_{1,0})(\omega_{U_0})\cdot 
\Gamma(\nu_{1,0}')\big](\Phi_3)\big)\Big](\nu_{1,0}).\end{cases}\end{equation}
It is clear that (\ref{STS1}) holds if $\nu_{1,0}$ and $\nu_{1,0}'$ are linearly dependent. So, it suffices to establish (\ref{STS1}) when $\nu_{1,0}$ and $\nu_{1,0}'$ are linearly independent elements of $U_0^*$. We may as well 
take $y,z,w$ and $y^*,z^*,w^*$ to be a pair of dual bases for $U_0$ and $U_0^*$ with  $y\w z\w w=\omega_{U_0}$, $w^*\w z^*\w y^*=\omega_{U_0^*}$, $\nu_{1,0}=y^*$, and $\nu_{1,0}'=z^*$. The right side of (\ref{STS1}) is now equal to

\begingroup\allowdisplaybreaks
\begin{align*}
{}&\begin{cases}\phantom{+}\big[p^{-1} \big(
\big[(y^*\w z^*)(y\w z\w w) \cdot
\Gamma(y^*)\big](\Phi_3)\big)\big](z^*)\\+ \big[p^{-1} \big(
\big[(z^*\w y^*)(y\w z\w w) \cdot
\Gamma(z^*)\big](\Phi_3)\big)\big](y^*)\end{cases}\\
{}={}&-\big[p^{-1} \big(
\big[ w
\Gamma(y^*)\big](\Phi_3)\big)\big](z^*)+ \big[p^{-1} \big(
\big[w 
\Gamma(z^*)\big](\Phi_3)\big)\big](y^*)\\
{}={}&-\big[ 
\big[ w p^{-1}(z^*)
\big](\Phi_3)\big][\Gamma(y^*)]+ 
\big[ 
\big[ w p^{-1}(y^*)
\big](\Phi_3)\big][\Gamma(z^*)]\\\end{align*}
\endgroup
$$=\begin{cases} -\big[ 
\big[ ywp^{-1}(z^*)
\big](\Phi_3)\big]\cdot y^*[\Gamma(y^*)]+ 
\big[ 
\big[y wp^{-1}(y^*)
\big](\Phi_3)\big]\cdot y^*[\Gamma(z^*)] \\
-\big[ 
\big[ zwp^{-1}(z^*)
\big](\Phi_3)\big]\cdot z^*[\Gamma(y^*)]+ 
\big[ 
\big[z wp^{-1}(y^*)
\big](\Phi_3)\big]\cdot z^*[\Gamma(z^*)]\\
-\big[ 
\big[ w^2p^{-1}(z^*)
\big](\Phi_3)\big]\cdot w^*[\Gamma(y^*)]+ 
\big[ 
\big[w^2p^{-1}(y^*)
\big](\Phi_3)\big]\cdot w^*[\Gamma(z^*)].
\end{cases}$$ 
Apply Definition~\ref{8.1}, especially (\ref{SM}), to see that 
\begin{align*}y^*\Gamma(y^*)={}&
[p^{-1}(z^2\Phi_3)](w^2\Phi_3)-[p^{-1}(zw\Phi_3)](zw\Phi_3)\\
y^*\Gamma(z^*)={}&
-[p^{-1}(yz\Phi_3)](w^2\Phi_3)+[p^{-1}(zw\Phi_3)](yw\Phi_3)\\
z^*\Gamma(z^*)={}&
[p^{-1}(y^2\Phi_3)](w^2\Phi_3)-[p^{-1}(yw\Phi_3)](yw\Phi_3)\\
w^*\Gamma(z^*)={}&
-[p^{-1}(y^2\Phi_3)](zw\Phi_3)+[p^{-1}(yz\Phi_3)](yw\Phi_3)\\
w^*\Gamma(y^*)={}&
[p^{-1}(yz\Phi_3)](zw\Phi_3)-[p^{-1}(z^2\Phi_3)](yw\Phi_3).
\end{align*}
It follows that the right side of  (\ref{STS1}) is equal to $Y_1+Y_2+Y_3$, with 
\begingroup \allowdisplaybreaks
\begin{align*}
Y_1={}&\begin{cases}
\phantom{+}
[p^{-1}(yw(\Phi_3))](z^*)
\cdot 
\begin{cases}-[p^{-1}(z^2(\Phi_3))](w^2(\Phi_3))\\+[p^{-1}(zw(\Phi_3))](zw(\Phi_3))\end{cases}
 \\
+
[p^{-1}(yw(\Phi_3))](y^*)
\cdot 
\begin{cases}-[p^{-1}(yz(\Phi_3))](w^2(\Phi_3))\\+[p^{-1}(zw(\Phi_3))](yw(\Phi_3)),\end{cases}
\end{cases}
\\
Y_2={}&\begin{cases}
\phantom{+}[p^{-1}(zw(\Phi_3))](z^*)\cdot 
\begin{cases}\phantom{+}[p^{-1}(yz(\Phi_3))](w^2(\Phi_3))\\-[p^{-1}(yw(\Phi_3))](zw(\Phi_3))\end{cases}
\\
+[p^{-1}(zw(\Phi_3))](y^*)\cdot 
\begin{cases}\phantom{+}[p^{-1}(y^2(\Phi_3))](w^2(\Phi_3))\\
-[p^{-1}(yw(\Phi_3))](yw(\Phi_3)),\end{cases} 
\end{cases}\\
\intertext{and}
Y_3={}&\begin{cases}
\phantom{+}[p^{-1}(w^2(\Phi_3))](z^*)\cdot
\begin{cases}-[p^{-1}(yz(\Phi_3))](zw(\Phi_3))\\
+[p^{-1}(z^2(\Phi_3))](yw(\Phi_3))\end{cases} 
\\
+[p^{-1}(w^2(\Phi_3))](y^*)
\cdot
\begin{cases}-[p^{-1}(y^2(\Phi_3))](zw(\Phi_3))\\
+[p^{-1}(yz(\Phi_3))](yw(\Phi_3)).\end{cases} 
\end{cases}\end{align*}\endgroup
We reconfigure $Y_3$ 
to obtain
\begingroup\allowdisplaybreaks
\begin{align*}Y_3&{}=\begin{cases}
-[p^{-1}(w^2(\Phi_3))](z^*)\cdot[p^{-1}(yz(\Phi_3))](zw(\Phi_3))\\
-[p^{-1}(w^2(\Phi_3))](y^*)\cdot[p^{-1}(y^2(\Phi_3))](zw(\Phi_3))\\
+[p^{-1}(w^2(\Phi_3))](z^*)\cdot[p^{-1}(z^2(\Phi_3))](yw(\Phi_3)) \\
+[p^{-1}(w^2(\Phi_3))](y^*)\cdot[p^{-1}(yz(\Phi_3))](yw(\Phi_3))\end{cases} 
\\
{}&{}= 
-[p^{-1}(yu_1'(\Phi_3))](zw(\Phi_3))
+[p^{-1}(u_1'z(\Phi_3))](yw(\Phi_3)), 
\end{align*} \endgroup 
for $u_1'=[p^{-1}(w^2(\Phi_3))](z^*)\cdot z
+[p^{-1}(w^2(\Phi_3))](y^*)\cdot y$.
The elements $x,y,z,w$ and $x^*,y^*,z^*,w^*$ form a pair of dual bases for $U$ and $U^*$, respectively; consequently, 
if $u_{1}\in U$, then $$u_{1}=u_1(x^*)\cdot x+u_1(y^*)\cdot y+u_1(z^*)\cdot z+
u_1(w^*)\cdot w.$$
 In particular, when $u_1$ is taken to be $p^{-1}(w^2(\Phi_3))\in U$, then 
$$u_1'=
[p^{-1}(w^2(\Phi_3))] 
-[p^{-1}(w^2(\Phi_3))](x^*)\cdot x 
-[p^{-1}(w^2(\Phi_3))](w^*)\cdot w, 
$$ and
 $Y_3=Y_3'+Y_3''+Y_3'''$, with
\begin{align*}
Y_3'={}&\begin{cases}-[p^{-1}(y  [p^{-1}(w^2(\Phi_3))](\Phi_3))](zw(\Phi_3))\\
+ [p^{-1}(z  [p^{-1}(w^2(\Phi_3))](\Phi_3))](yw(\Phi_3))\end{cases}\\
{}={}&\begin{cases} -\big[y \cdot [p^{-1}(zw(\Phi_3))]\cdot [p^{-1}(w^2(\Phi_3))]\big](\Phi_3) \\
+\big[z \cdot [p^{-1}(yw(\Phi_3))] \cdot [p^{-1}(w^2(\Phi_3))]\big](\Phi_3),\end{cases}\\
Y_3''={}&\begin{cases} \phantom{+}
\big[p^{-1}\big(y  \big(\big[[p^{-1}(w^2(\Phi_3))](x^*)\big]\cdot x  \big)(\Phi_3)\big)\big](zw(\Phi_3))\\- \big[p^{-1}\big(z  \big(\big[[p^{-1}(w^2(\Phi_3))](x  ^*)\big]\cdot x  \big)(\Phi_3)\big)\big](yw(\Phi_3))\end{cases}\\
{}={}&\big[[p^{-1}(w^2(\Phi_3))](x  ^*)\big]\cdot\big(
[p^{-1}(y   x  (\Phi_3))](zw(\Phi_3))- [p^{-1}(z   x  (\Phi_3))](yw(\Phi_3))\big)\\
{}={}&\big[[p^{-1}(w^2(\Phi_3))](x^*)\big]\cdot\big(
(yzw(\Phi_3))- (zyw(\Phi_3))\big)=0,\\\intertext{because $p^{-1}(u_1x\Phi_3)=p^{-1}(p(u_1))=u_1$ for $u_1\in U$, and}
Y_3'''={}&\begin{cases}\phantom{+}\big[p^{-1}\big(y \big(\big[[p^{-1}(w^2(\Phi_3))](w^*)\big]\cdot w\big)(\Phi_3)\big)\big](zw(\Phi_3))\\-\big[p^{-1}\big(z
\big(\big[[p^{-1}(w^2(\Phi_3))](w^*)\big]\cdot w\big)(\Phi_3)\big)\big](yw(\Phi_3))\end{cases} \\
{}={}&\big[[p^{-1}(w^2(\Phi_3))](w^*)\big]\cdot\begin{cases} 
\phantom{+}[p^{-1}(y  w(\Phi_3))](zw(\Phi_3))\\-[p^{-1}(z
 w(\Phi_3))](yw(\Phi_3))\end{cases}\\
{}={}&0.
\end{align*}
Thus, $$Y_3=Y_3'=\begin{cases}-\big[y \cdot [p^{-1}(zw(\Phi_3))]\cdot [p^{-1}(w^2(\Phi_3))]\big](\Phi_3)\\ +\big[z \cdot [p^{-1}(yw(\Phi_3))] \cdot [p^{-1}(w^2(\Phi_3))]\big](\Phi_3).\end{cases}$$
In a similar manner 
\begingroup\allowdisplaybreaks\begin{align*}Y_1&= \begin{cases}
-\big[z\cdot p^{-1}[(w^2)(\Phi_3)]\cdot p^{-1}[(yw)(\Phi_3)]\big](\Phi_3)\\
+\big[w\cdot p^{-1}[(zw)(\Phi_3)]\cdot p^{-1}[(yw)(\Phi_3)]\big](\Phi_3)\end{cases}
 \\
Y_2&=\begin{cases}\phantom{+}\big[y\cdot p^{-1}[(w^2)(\Phi_3)]\cdot p^{-1}[(zw)(\Phi_3)]\big](\Phi_3)\\
-\big[w\cdot p^{-1}[(yw)(\Phi_3)]\cdot p^{-1}[(zw)(\Phi_3)]\big](\Phi_3).\end{cases}
\end{align*}\endgroup
It is now apparent that $Y_1+Y_2+Y_3=0$; thus, (\ref{STS1}) is established and the proof is complete.
\end{proof}

\begin{chunk}
\label{Proof-of-step-1}
{\it Proof of Theorem~{\rm\ref{6.1}}.} 
Start with the complex 
 \begin{equation}\label{(*)} P\t D_2U_0^*\xrightarrow{\mL}  P\t \Sym_2U_0\xrightarrow{\mf_{1,2}}P\to  \AA 
\to0\end{equation} of Lemma~\ref{26.8}.  
The image of $\mf_{1,2}$ is equal to the defining ideal of $\AA$  
 and 
$\mL$ maps onto the linear syzygies of $\mf_{1,2}$.
Use the basis \begin{equation}\label{basis}y^{*(2)},y^*z^*,y^*w^*,z^{*(2)},z^*w^*,w^{*(2)}\end{equation} to identify $D_2U_0^*$ with $P^6$ and the basis $y^2,yz,yw,z^2,zw,w^2$ to identify $\Sym_2U_0$ with $P^6$. Convert (\ref{(*)}) into the complex 
$$P(-3)^6\xrightarrow{L}P(-2)^6\xrightarrow{f_{1,2}}P\to \AA\to 0,$$where $L$ is the matrix of $\mL$ and $f_{1,2}$ is given in Example~\ref{record}.
It is shown in Lemma~\ref{alternating}   that $L$ is an alternating matrix.

The nonzero element $\delta \in D_2U_0^*$ 
with $\delta \in \ker \mL$ is identified in   
Lemma~\ref{26.8}. When $\delta$ is written in terms of the basis (\ref{basis}), it looks like
\begin{equation}\label{delta-mat}\bmatrix \CA_{(1,1)}\\\CA_{(1,2)}\\\CA_{(1,3)}\\\CA_{(2,2)}\\\CA_{(2,3)}\\\CA_{(3,3)}\endbmatrix,\end{equation} where $\CA$ is the classical adjoint of $\SM$. Let $\blip$ be a $6\times 6$ invertible matrix with entries from $\kk$ and last column equal to (\ref{delta-mat}). Observe that
$$P(-3)^6\xrightarrow{\blip\transpose L\blip}P(-2)^6\xrightarrow{f_{1,2}(\blip\transpose)^{-1}} P\to \AA\to 0$$ is a complex which is exact at the positions $P$ and $\AA$. Observe also that the alternating matrix $\blip\transpose L\blip$ maps onto the linear syzygies of $f_{1,2}(\blip\transpose)^{-1}$.
  Write $f_{1,2}(\theta\transpose)^{-1}$ as $\bmatrix d_1&g_6\endbmatrix$, where $d_1=\bmatrix g_1&g_2&g_3&g_4&g_5\endbmatrix$, and write
$$\theta\transpose L\theta=\left[\begin{array}{c|c} X&0_{5\times 1}\\\hline 0_{1\times 5}&0\end{array}\right],$$ where $X$ is a $5\times 5$ alternating matrix. None of the linear syzygies of $\bmatrix d_1&g_6\endbmatrix$ involve $g_6$.
The hypotheses of Observation~\ref{Oct5} are satisfied and  the proof is complete. Indeed, the minimal homogeneous resolution of $\AA$ is the mapping cone of
\begin{equation}\label{MC}\xymatrix{
0\ar[r]&P(-7)\ar[rr]^{d_1\transpose}
\ar[d]^{g_6}&&P(-5)^5\ar[rr]^{X}\ar[d]^{g_6}&&P(-4)^5\ar[rr]^{d_1}\ar[d]^{g_6}&&P(-2)\ar[d]^{g_6}\\
0\ar[r]&P(-5)\ar[rr]^{d_1\transpose}&&P(-3)^5\ar[rr]^{X}&&P(-2)^5\ar[rr]^{d_1}&&P.
}\end{equation}
The minimal homogeneous resolution of $\AA$ is also given by 
$$0\to P\xrightarrow{\mathfrak g_4}P^6\xrightarrow{\mathfrak g_3}P^{10}\xrightarrow{
\mathfrak g_2
}P^6\xrightarrow{\mathfrak g_1}P.$$
with \begin{align*}\mathfrak g_1={}&\bmatrix d_1&g_6\endbmatrix,
&&\mathfrak g_2=\bmatrix 
X&-g_6\\
0_{1\times 5}&d_1
\endbmatrix,\\
\mathfrak g_3={}&\bmatrix
-g_6&d_1\transpose\\
 -X&0_{5\times 1}
\endbmatrix,\text{ and}&&
\mathfrak g_4=\bmatrix d_1\transpose\\g_6\endbmatrix.\end{align*}

\hfil \qed
\end{chunk}

\begin{theorem}
\label{6.2} Adopt the data of {\rm \ref{2.1}} and the matrix $\SM$ of {\rm(\ref{SM})}. If $\SM$ has rank at least two, 
then the defining ideal  of $\AA$ is the generalized sum of two linked codimension three perfect ideals.
\end{theorem}
\begin{proof}
See (\ref{MC}) and Example~\ref{7.4}.
\end{proof}

\section{Macaulay2 scripts.}

The package ``SocDeg3.m2'' contains five Macaulay2 scripts: BasicSocDeg3,  minres7, SumOfLinkedIdeals7, SumOfLinkedIdeals9, and minres6. In this section, we describe each of these scripts, print the package  SocDeg3.m2, run a few examples, and give a list of Phi's which are in the ``correct form'' in the sense that $\ann \Phi\cap P_1=\{0\}$, the first listed variable is a weak Lefschetz element on P/ann Phi and if $\rank \SM=1$, then $\SM$ is given by (\ref{form}).  
Recall that $\SM$ is given by (\ref{form}) if and only if $B$ is given by
(\ref{specialB}).

\begin{remark}
\label{14.1}
 Recall that if $P$ is the polynomial ring $\pmb k[x_1,\dots,x_n]=\Sym_\bullet P_1$, then the graded dual of $P$ is the Divided Power Algebra $D_{\bullet}P_1^*$. The monomials $\binom{x_1,\dots, x_n}{i}$ form a basis for $\Sym_i P_1$ (see Convention~\ref{2.5}) and the elements
$$\{x_1^{*(e_1)}\cdots x_n^{*(e_n)}\mid e_1+\dots+e_n=i\}$$
form a basis for $D_iP_1^*$. The structure of $D_\bullet P_1^*$ as a $\Sym_\bullet P_1$- module is given by 
\begin{equation}\label{contract}
x_1^{f_1}\cdots x_n^{f_n}(x_1^{*(e_1)}\cdots x_n^{*(e_n)})=\begin{cases}
 x_1^{*(e_1-f_1)}\cdots x_n^{*(e_n-f_n)},&\text{if $f_j\le e_j$ for all $j$},\\
0,&\text{otherwise}.\end{cases}\end{equation}

The notation that Macaulay2 uses for the element $$x_1^{*(e_1)}\cdots x_n^{*(e_n)}$$ of $D_\bullet P_1^*$ is $x_1^{e_1}\cdots x_n^{e_n}$. 
Consequently, in this section, when we write ``Let Phi be the element $yx^2+xzw+xw^2+z^3+w^3$ of $D_3P_1^*$'', we are thinking, ``Let Phi be the element $y^*x^{*(2)}+x^*z^*w^*+x^*w^{*(2)}+z^{*(3)}+w^{*(3)}$ of $D_3P_1^*$''.

The Macaulay2 command contract($x_1^{f_1}\cdots x_n^{f_n},x_1^{e_1}\cdots x_n^{e_n})$ produces (\ref{contract}).\end{remark}

\subsection{The scripts in SocDeg3.m2.}

$ $

\subsubsection{The script ``BasicSocDeg3''.}\label{14.A.1}
Let $P$ be a polynomial ring in four variables  over a field k.
For the sake of this introduction $P=k[x,y,z,w]$ (although any field can be used  and the variables can be given any names). Let Phi be an element of $D_3 P_1^*$, where $P_1$ is the subspace of linear forms in $P$, $P_1^*$ is $\Hom_k(P_1,k)$, and $D_\bullet P_1^*$ is the Divided Power algebra which is the graded dual of the polynomial ring $P=\Sym_\bullet P_1$
The  Phi must be chosen  so that zero is the only linear element in ann(Phi) and the variable $x$ (that is, the first variable in the  list of variables) is a weak Lefschetz element on P/ann Phi.
{\bf Run this script by typing:} 
$$\text{(f11,f12,A,B,C,D,SM)=BasicSocDeg3(Phi)}.$$
The output is the seven matrices $f_{1,1}, f_{1,2},A, B, C, D, \SM$. The first six matrices are described in  Example~\ref{record}. The last matrix is described in (\ref{SM}). According to Theorem~\ref{resol}, the matrices
\begin{equation}\notag 
0\to P(-6)\xrightarrow{\ f_4\ }\begin{matrix}P(-4)^3\\\p\\ P(-4)^6\end{matrix}\xrightarrow{\ f_3\ }\begin{matrix}P(-3)^{8}\\\p\\ P(-3)^8\end{matrix}\xrightarrow{\ f_2\ }
\begin{matrix}P(-2)^3\\\p\\ P(-2)^6\end{matrix}\xrightarrow{\ f_{1}\ }P,\end{equation}
form a minimal homogeneous resolution of $P/\ann(x\Phi)$, where
\begin{align*}
 f_4={}&\bmatrix f_{1,1}\transpose\\f_{1,2}\transpose\endbmatrix,&
 f_3={}&\bmatrix xB\transpose&D \transpose\\A \transpose&xC \transpose\endbmatrix,\\
 f_2={}&\bmatrix A &xB \\xC &D \endbmatrix,\text{ and}&
 f_1={}&\bmatrix f_{1,1}&f_{1,2}\endbmatrix.\end{align*}
According to Theorem~\ref{a res}, 
the matrices
\begin{equation}
\label{14.0.1}0\to P(-7)\xrightarrow{\ F_4\ }\begin{matrix}P(-4)^3\\\p\\ P(-5)^6\end{matrix}\xrightarrow{\ F_3\ }\begin{matrix}P(-4)^{8}\\\p\\ P(-3)^8\end{matrix}\xrightarrow{\ F_2\ }
\begin{matrix}P(-3)^3\\\p\\ P(-2)^6\end{matrix}\xrightarrow{\ F_{1}\ }P,\end{equation}
form a  homogeneous  (not necessarily minimal) resolution of $P/\ann(x\Phi)$, where
\begin{align*}
 F_4={}&\bmatrix xf_{1,1}\transpose\\f_{1,2}\transpose\endbmatrix,&
F_3={}&\bmatrix B\transpose&D \transpose\\A \transpose&x^2C \transpose\endbmatrix,\\
 F_2={}&\bmatrix A &B \\x^2C &D \endbmatrix,\text{ and}&
 F_1={}&\bmatrix xf_{1,1}&f_{1,2}\endbmatrix.\end{align*} 
The rank of the $3\times 3$ symmetric matrix $\SM$ determines the number of minimal generators in the defining ideal of $\AA$. 
If $\SM$ has rank zero, then the defining ideal of $\AA$ has nine minimal generators and the minimal resolution of $\AA$ is given by (\ref{14.0.1}).
If $\SM$ has rank one, then the defining ideal of $\AA$ has seven minimal generators. One may change variables to put $\SM$ in the form
\begin{equation}
\notag
\bmatrix \a&0&0\\0&0&0\\0&0&0\endbmatrix,\end{equation}for some nonzero $\a\in\kk$. Once this is done, then
 the minimal resolution of $\AA$ is given in Theorem \ref{7.1}.   
This procedure is implemented by the script ``minres7''.  
If $\SM$ has rank at least two, then the defining ideal of $\AA$ has six minimal generators, $\AA$ is a hypersurface section of a codimension three Gorenstein ring, and the minimal resolution of $\AA$ is given in Theorem \ref{6.1}. This procedure is implemented in the script ``minres6''.

\subsubsection{The script ``minres7''.}\label{14.A.2}
Let $P$ be a polynomial ring in four variables over a field $k$ 
 and Phi be an element of $D_3P_1^*$. Choose Phi so that zero is the 
 only linear element in ann(Phi) and the first listed variable is a weak
 Lefschetz element on P/ann Phi.
 Assume further that the matrix $B$ of Example~\ref{record} has rank $2$, and has the form
 $$\bmatrix 0& 0&       0&     0& 0& 0& 0& 0\\
 0& 0&       \alpha& 0& 0& 0& 0& 0\\
 0& -\alpha & 0&     0& 0& 0& 0& 0\endbmatrix,$$
 where $\alpha$ is a unit.

  {\bf Run this script by typing}
$$\text{(E1,E2,E3,E4)=minres7(Phi)}.$$  
  The output is the four matrices $E_1$, $E_2$, $E_3$, and $E_4$ in the minimal resolution
 of P/ann (Phi), as described in Theorem~\ref{7.1}. 

\subsubsection{The script ``SumOfLinkedIdeals7''.}\label{14.A.3}
Let $P$ be a polynomial ring in four variables over a field $k$ 
 and Phi be an element of $D_3P_1^*$. Choose Phi so that zero is the 
 only linear element in ann(Phi) and the first listed variable is a weak
 Lefschetz element on P/ann Phi.
 Assume further that the matrix $B$ from Example~\ref{record} has rank $2$, and has the form
 $$B=\bmatrix 0& 0&       0&     0& 0& 0& 0& 0\\
 0& 0&       \alpha& 0& 0& 0& 0& 0\\
 0& -\alpha & 0&     0& 0& 0& 0& 0\endbmatrix,$$
 where $\alpha$ is a unit. {\bf Run this script by typing}
$$\text{(B1,B2,B3,L0,L1)=SumOfLinkedIdeals7(Phi)}.$$  
  The output is five matrices $B_1,B_2,B_3,\mathcal L_0,\mathcal L_1$ which 
sit in the following map of complexes:
\begin{equation}\label{14.0.2}\xymatrix{
0\ar[r]&P(-7)\ar[r]^{B_1\transpose}\ar[d]^{-\mathcal L_0\transpose}&P(-5)^5\ar[r]^{B_2\transpose}\ar[d]^{\mathcal L_1\transpose}&{\begin{matrix} P(-3)\\\p\\P(-4)^5\end{matrix}}\ar[r]^{B_3\transpose}\ar[d]^{-\mathcal L_1}&{\begin{matrix} P(-2)\\\p\\P(-3)\end{matrix}}\ar[d]^{\mathcal L_0}\\
0\ar[r]& {\begin{matrix}P(-5)\\\p\\P(-4)\end{matrix}}\ar[r]^{B_3}&{\begin{matrix}P(-4)\\\p\\P(-3)^5\end{matrix}}\ar[r]^{B_2}&P(-2)^5\ar[r]^{B_1}&P.}\end{equation}
The bottom row of (\ref{14.0.2}) is the Brown resolution  of $P/I_1(B_1)$, where $I_1(B_1)$ is a perfect grade 3 ideal; see Proposition~\ref{Bres}. The top row is the resolution of the canonical module, 
$\Ext^3_P(P/I_1(B_1),P)$, of $P/I_1(B_1)$. The mapping cone of (\ref{14.0.2}) is the minimal resolution of $P/\ann \Phi$. Thus, the defining ideal of $\AA$ is expressed as the sum of linked perfect ideals  as described in Theorem~\ref{7.2}.

\subsubsection{The script ``SumOfLinkedIdeals9''.}\label{14.A.4}
Let $P$ be a polynomial ring in four variables over a field $k$ 
 and Phi be an element of $D_3P_1^*$. Choose Phi so that zero is the 
 only linear element in ann(Phi) and the first listed variable is a weak
 Lefschetz element on P/ann Phi.
 Assume further that the matrix $B$ from Example~\ref{record} is identically zero. 
 {\bf Run this script by typing}
$$\text{(K1,K2,K3,L0,L1)=SumOfLinkedIdeals9(Phi)}.$$  
  The output is five matrices $K_1,K_2,K_3, L_0, L_1$ which 
sit in the following map of complexes:
\begin{equation}\label{14.0.3}\xymatrix{
0\ar[r]&P(-7)\ar[r]^{K_1\transpose}\ar[d]^{-L_0\transpose}&P(-5)^6\ar[r]^{K_2\transpose}\ar[d]^{L_1\transpose}&P(-4)^8
\ar[r]^{K_3\transpose}\ar[d]^{-L_1}&P(-3)^3
\ar[d]^{L_0}\\
0\ar[r]& P(-4)^3 
\ar[r]^{K_3}&P(-3)^8
\ar[r]^{K_2}&P(-2)^6\ar[r]^{K_1}&P.}\end{equation}
The bottom row of (\ref{14.0.3}) is a deformation of the Eagon-Northcott resolution of $P/(y^2,yz,yw,z^2,zw,w^2)$.
 The top row is the resolution of the canonical module, 
$\Ext^3_P(P/I_1(K_1),P)$, of $P/I_1(K_1)$. The mapping cone of (\ref{14.0.3}) is the minimal resolution of $P/\ann \Phi$. Thus, the defining ideal of $\AA$ is expressed as the sum of linked perfect ideals  as described in Theorem~\ref{9.1}.

\subsubsection{The script ``minres6''.}\label{14.A.5}
Let $P$ be a polynomial ring in four variables over a field $k$ 
 and Phi be an element of $D_3P_1^*$. Choose Phi so that zero is the 
 only linear element in ann(Phi) and the first listed variable is a weak
 Lefschetz element on P/ann Phi.
 Assume further that the matrix $\SM$ from (\ref{SM}) is invertible. {\bf Run this script by typing}
$$\text{(d1,f,X)=minres6(Phi)}.$$ 
The output is a $1\times 5$ matrix $d_1$ with entries from $P$, an element $f$ of $P$, and a $5\times 5$ alternating matrix $X$ with entries from $P$.
Let $J$ be the ideal generated by the entries of $d_1$. Then the ideal $(J,f)$ is equal to  the ideal ann (Phi) of $P$; the ring $P/J$ is a codimension three Gorenstein ring with minimal homogeneous resolution
$$0\to P(-5)\xrightarrow{d_1\transpose}P(-3)^5\xrightarrow{X}P(-2)^5\xrightarrow{d_1}P,$$
and $f$ is regular on $P/J$. Thus P/ann(Phi) is a hypersurface section of a codimension three Gorenstein ring as described in Theorems \ref{6.1} and \ref{6.2}.

\subsection{The scripts.}

$ $

{\tiny
\begin{verbatim}


-- **************************************************
-- ** The script BasicSocDeg3                      **
-- ** Type (f11,f12,A,B,C,D,SM)=BasicSocDeg3(Phi); **
-- **************************************************  
BasicSocDeg3=(Phi) -> (
Rl:=ring(Phi);
Phil:=Phi;

if not isHomogeneous(Phil) then
error "Phi must be homogeneous.";

if not (degree(Phil))_0==3 then
error "Phi must have degree 3.";

xl:=Rl_0;
yl:=Rl_1;
zl:=Rl_2;
wl:=Rl_3;
vl:=matrix{{xl,yl,zl,wl}};
vsql:=matrix{{yl^2,yl*zl,yl*wl,zl^2,zl*wl,wl^2}};

pl:=contract (xl*(transpose vl)*vl,Phil);

if det pl==0 then
error "One must choose Phi so that the only linear element in ann(Phi)
 is zero and the variable P_0 is a weak Lefschetz element on P/ann Phi.";

pinversel:=inverse (pl);
blop:=(u)-> (contract(matrix{{u}},matrix{{Phil}})//vl);
proj:=matrix{{0_Rl,1,0,0},{0,0,1,0},{0,0,0,1}};
ip:=(u,u')-> (transpose( blop(u))*pinversel*blop(u'));

-- If u is an element of Sym_2U, then blop(u) is  u(Phi_3), written
-- as a column vector.

f11l:=xl*vl*pinversel*matrix{{0,0,0},{1,0,0},{0,1,0},{0,0,1}};
f12l:=(vsql-xl*vl*pinversel*( matrix{{contract(vsql,Phil)}}//vl));


c1b212:=
matrix{{0},{ip(yl*wl,zl^2)-ip(yl*zl,zl*wl)},{ip(yl*wl,zl*wl)-ip(yl*zl,wl^2)}};

c2b212:=
matrix{{ip(yl*zl,zl*wl)-ip(yl*wl,zl^2)},{0},{ip(zl*wl,zl*wl)-ip(zl^2,wl^2)}};

c3b212:=
matrix{{ip(yl*zl,wl^2)-ip(yl*wl,zl*wl)},{ip(zl^2,wl^2)-ip(zl*wl,zl*wl)},{0}};

c4b212:=
matrix{{0},{ip(yl^2,zl*wl)-ip(yl*wl,yl*zl)},{ip(yl^2,wl^2)-ip(yl*wl,yl*wl)}}; 

c5b212:=
matrix{{ip(yl*wl,yl*zl)-ip(yl^2,zl*wl)},{0},{ip(zl*yl,wl^2)-ip(zl*wl,yl*wl)}};

c6b212:=
matrix{{ip(yl*wl,yl*wl)-ip(yl^2,wl^2)},{ip(zl*wl,yl*wl)-ip(wl^2,yl*zl)},{0}};

c7b212:=
matrix{{0},{ip(yl*zl,yl*zl)-ip(yl^2,zl^2)},{ip(yl*zl,yl*wl)-ip(yl^2,wl*zl)}};

c8b212:=
matrix{{ip(yl^2,zl^2)-ip(yl*zl,yl*zl)},{0},{ip(zl^2,yl*wl)-ip(zl*yl,zl*wl)}};

b212:=c1b212|c2b212|c3b212|c4b212|c5b212|c6b212|c7b212|c8b212;

kb222:=matrix{
{ 0,  0,  0,  wl,  0,  0,  -zl, 0  },
{ -wl, 0,  0,  0,  wl,  0,  yl,  -zl },
{ zl,  0,  0,  -yl, 0,  wl,  0,  0  },
{ 0,  -wl, 0,  0,  0,  0,  0,  yl  },
{ 0,  zl,  -wl, 0,  -yl, 0,  0,  0  },
{ 0,  0,  zl,  0,  0,  -yl , 0,  0  }};

-- If l and qu are homogeneous forms in Sym_2U_0 of degree 1 and 2
-- respectively, then j12(l,qu) is the quadratic form 
-- l times (proj circ p^{-1})(qu Phi) in Sym_2U_0. Of course the 
-- answer is expressed 
-- as a column vector. The entries of the column vector are the 
-- coefficients of y^2,yz,yw,z^2,zw,w^2 in that order.

prej12:= (l,qu) ->( 
t1:=pinversel*blop(qu);
return(l*matrix{{yl,zl,wl}}*proj*t1));

j12:= (l,qu) ->( matrix{{prej12(l,qu)}}//vsql);

 b222:=
(-j12(zl,yl*wl)+j12(wl,yl*zl)| 
j12(wl,zl^2)-j12(zl,wl*zl)|    
j12(wl,zl*wl)-j12(zl,wl^2)|     
-j12(wl,yl^2)+j12(yl,wl*yl)|   
-j12(wl,yl*zl)+j12(yl,wl*zl)|   
-j12(wl,yl*wl)+j12(yl,wl^2)|   
j12(zl,yl^2)-j12(yl,yl*zl)|    
j12(zl,yl*zl)-j12(yl,zl^2));  

kb211:=matrix{{yl,  0,  0,  zl, 0,   0,  wl, 0}, 
 { 0,   yl, 0,  0,  zl,  0,  0,  wl },
  {-wl, 0,  yl, 0,  -wl, zl, 0,  0}};  

j14:=(mu,nu) ->( above=
mu*matrix{{yl,zl,wl}}*proj*pinversel*(matrix{{nu}}//vl); 
return(matrix{{above}}//vsql)
);

b221:=
(j14(yl,yl)-j14(wl,wl))|
j14(yl,zl)|
j14(yl,wl)|
j14(zl,yl)|
(j14(zl,zl)-j14(wl,wl))|
j14(zl,wl)|
j14(wl,yl)|
j14(wl,zl);

-- j14'(mu,nu) is the element proj[ (mu (p^{inverse} nu))(Phi)] of U_0 
-- for mu in U and nu in U^*. The answer is written as a column vector.
-- The input is given in terms of elements of U and U^*.

j14':=(mu,nu) ->(
above:=contract(mu*vl*pinversel*(matrix{{nu}}//vl),Phil);
return(proj*(matrix{{above}}//vl)));

b211:=
(j14'(yl,yl)-j14'(wl,wl))| 
j14'(yl,zl)| 
j14'(yl,wl)| 
j14'(zl,yl)| 
(j14'(zl,zl)-j14'(wl,wl))| 
j14'(zl,wl)| 
j14'(wl,yl)| 
j14'(wl,zl);


Al:=(-xl*b211+kb211);
Bl:=b212;
Cl:=-b221;
Dl:=(kb222-xl*b222);
SMl:=   matrix{{ip(zl^2,wl^2)-ip(zl*wl,zl*wl),
ip(zl*wl,yl*wl)-ip(zl*yl,wl^2),
   ip(yl*zl,zl*wl)-ip(zl^2,yl*wl)},
    {ip(zl*wl,yl*wl)-ip(yl*zl,wl^2),
ip(yl*yl,wl^2)-ip(yl*wl,yl*wl),
ip(yl*zl,yl*wl)-ip(yl*yl,zl*wl)},
  {ip(yl*zl,zl*wl)-ip(yl*wl,zl^2),
ip(yl*wl,yl*zl)-ip(yl*yl,zl*wl),ip(yl*yl,zl*zl)-ip(yl*zl,yl*zl)}};

return(f11l,f12l,Al,Bl,Cl,Dl,SMl)
);

-- **************************************************
-- ** The script minres7                           **
-- ** Type (E1,E2,E3,E4)=minres7(Phi);             **
-- **************************************************



minres7=(Phi)->(
Run1:=BasicSocDeg3(Phi);
f11l1:=(Run1)_0;
f12l1:=(Run1)_1;
Al1:=(Run1)_2;
Bl1:=(Run1)_3;
Cl1:=(Run1)_4;
Dl1:=(Run1)_5;

Pl:=ring(Phi);
xl:=Pl_0;
alphal:=Bl1_(1,2);

if  alphal==0 then
error "The entry of B in row 2, column 3 must be a unit.";

realB:=matrix {{0,0,0,0,0,0,0,0},{0,0,alphal,0,0,0,0,0},
{0,-alphal,0,0,0,0,0,0}};

if not (Bl1-realB)==0 then error "The matrix B has the wrong form.";

zmat=(r,c)->(map(Pl^r,Pl^c,(i,j)->0));


theta1:=(matrix{{1_Pl,0,0},{0,0,alphal},{0,-alphal,0}}|zmat(3,6))||
(zmat(6,1)|submatrix(Dl1,,{1,2})|id_(Pl^6));

NewM:=(1//alphal)*matrix{
{0,-Al1_(2,0),Al1_(1,0),0,0,0,0,0},
{Al1_(2,0),0,Al1_(2,2),Al1_(2,3),Al1_(2,4),Al1_(2,5),Al1_(2,6),Al1_(2,7)},
{-Al1_(1,0),-Al1_(2,2),0,-Al1_(1,3),-Al1_(1,4),-Al1_(1,5),-Al1_(1,6),-Al1_(1,7)},
{0,-Al1_(2,3),Al1_(1,3),0,0,0,0,0},
{0,-Al1_(2,4),Al1_(1,4),0,0,0,0,0},
{0,-Al1_(2,5),Al1_(1,5),0,0,0,0,0},
{0,-Al1_(2,6),Al1_(1,6),0,0,0,0,0},
{0,-Al1_(2,7),Al1_(1,7),0,0,0,0,0}};
Newtheta2:=(id_(Pl^8)|zmat(8,8))||(NewM|id_(Pl^8));
F1l:=map(Pl^1, ,(xl*f11l1)|f12l1);
F2l:= map(source F1l,,(Al1|Bl1)||(xl^2*Cl1|Dl1));
E1l:=map(Pl^1,,submatrix'(F1l*theta1,,{1,2}));
TF2l:=(inverse theta1)*F2l*Newtheta2;
TE2l:=submatrix'(TF2l,{1,2},{1,2,9,10});
E2l:=map(source E1l,,TE2l);
Jl:=(zmat(6,6)|id_(Pl^6))||(id_(Pl^6)|zmat(6,6));
E3l:= map(source E2l,,Jl*transpose E2l);
E4l:=map(source E3l,,transpose E1l);
return(E1l,E2l,E3l,E4l)
);


-- ********************************************************
-- ** The script SumOfLinkedIdeals7                      **
-- ** Type (B1,B2,B3,L0,L1)=SumOfLinkedIdeals7(Phi); **
-- ********************************************************   


SumOfLinkedIdeals7=(Phi)->(
Run2:=minres7(Phi);
Run3:=BasicSocDeg3(Phi);
Bl1:=(Run3)_3;
Dl1:=(Run3)_5;

E1l1:=(Run2)_0;
E2l1:=(Run2)_1;
E3l1:=(Run2)_2;
E4l1:=(Run2)_3;
E1l':=submatrix(E1l1,,{4,0,6,5,1,2,3});
E2l':=submatrix(E2l1,{4,0,6,5,1,2,3},{11,2,3,0,1,4,5,8,9,6,7,10});
E3l':=submatrix(E3l1,{11,2,3,0,1,4,5,8,9,6,7,10},{6,5,1,2,3,4,0});
E4l':=submatrix(E4l1,{6,5,1,2,3,4,0},);     
L0l:=submatrix(E1l',,{0,1});
K1l:=submatrix(E1l',,{2..6});
K2l:=submatrix'(E2l',{0,1},{0..5});
L1l:=submatrix(E2l',{2..6},{0..5});
K3l:=transpose(submatrix(E2l',{0,1},{0..5}));
 alphal:=Bl1_(1,2); 
 thetal:=matrix{
 {1,0,0,0,0,0},
 {(1//alphal)*Dl1_(1,2),1,0,0,0,0},
 {(-1//alphal)*Dl1_(1,1),0,1,0,0,0},
 {0,0,0,1,0,0},
 {0,0,0,0,1,0},
 {0,0,0,0,0,1}};
 K2l':=K2l*thetal;
 K3l':=(inverse thetal)*K3l;
 beta3:=matrix{{-1//alphal,0},{0,-1}};  
 beta2:=matrix{
{1//alphal,0,0,0,0,0},
{0,1,0,0,0,0}, 
{0,0,-1,0,0,0},
{0,0,0,1,0,0},
{0,0,0,0,1,0},
{0,0,0,0,0,1}};
Bl1':=-K1l;
Bl2:=K2l'*(inverse (beta2));
Bl3:=beta2*K3l'*( inverse (beta3));    
newL0l:=L0l*(transpose beta3);
newL1l:=-L1l*(transpose(inverse(thetal)))*(transpose(beta2));
return(Bl1',Bl2,Bl3,newL0l,newL1l)
);

-- ********************************************************
-- ** The script SumOfLinkedIdeals9                      **
-- ** Type (K1,K2,K3,L0,L1)=SumOfLinkedIdeals9(Phi);     **
-- ********************************************************   



SumOfLinkedIdeals9=(Phi)->(
Run4:=BasicSocDeg3(Phi);
f11l1:=(Run4)_0;
f12l1:=(Run4)_1;
Al1:=(Run4)_2;
Bl1:=(Run4)_3;
Cl1:=(Run4)_4;
Dl1:=(Run4)_5;
Pl=ring(Phi);
xl:=Pl_0;
if not (Bl1==0) then error "The matrix B must be zero to use this script.";

K1l:=f12l1;
K2l:=Dl1;
K3l:=transpose Al1;
L1l:=xl^2*Cl1;
L0l:=xl*f11l1;
return(K1l,K2l,K3l,L0l,L1l)
);




-- ********************************************************
-- ** The script minres6                                 **
-- ** Type (d1,f,X)=minres6(Phi);                        **
-- ********************************************************  

minres6=(Phi)->(

ds':=(X,r,c)->(return(det submatrix'(X,{r},{c})));
CA:=(X)->(return (transpose(matrix{{ds'(X,0,0),-ds'(X,0,1),ds'(X,0,2)},
{-ds'(X,1,0),ds'(X,1,1),-ds'(X,1,2)},
{ds'(X,2,0),-ds'(X,2,1),ds'(X,2,2)}}))
);

Run5:=BasicSocDeg3(Phi);

f12l1:=(Run5)_1;
Dl1:=(Run5)_5;
SMl1:=(Run5)_6;
 Pl:=ring(Phi);
 if  (minors(2,SMl1)==0) then 
error "The matrix SM must have rank at least two to use this script.";

 SMi:=CA(SMl1);
 predelta:=matrix{{SMi_(0,0)},{SMi_(0,1)},{SMi_(0,2)},{SMi_(1,1)}};
 delta:=predelta||matrix{{SMi_(1,2)},{SMi_(2,2)}};
C2:=matrix{
{SMl1_(0,0),SMl1_(0,1),0,0,-SMl1_(2,1),-SMl1_(2,2)},
{SMl1_(1,0),SMl1_(1,1),SMl1_(1,2),0,0,0},
{SMl1_(2,0),SMl1_(2,1),SMl1_(2,2),0,0,0},
{0,SMl1_(0,0),0,SMl1_(0,1),SMl1_(0,2),0},
{0,SMl1_(1,0),-SMl1_(2,0),SMl1_(1,1),0,-SMl1_(2,2)},
{0,SMl1_(2,0),0,SMl1_(2,1),SMl1_(2,2),0},
{0,0,SMl1_(0,0),0,SMl1_(0,1),SMl1_(0,2)},
{0,0,SMl1_(1,0),0,SMl1_(1,1),SMl1_(1,2)}}; 
L:=Dl1*C2;
M:=id_(Pl^6)|delta;
i:=#(for i from 0 to 6 when (det submatrix'(M,,{i})==0) list i do i=i+1);
above:=submatrix'(id_(Pl^6),,{i})|delta;
X:=submatrix' (transpose(above) * L*above,{5},{5});
newf12l:=f12l1*inverse(transpose above);
d1:=submatrix'(newf12l,,{5});
f:=newf12l_(0,5);
return(d1,f,X)
   );







\end{verbatim}
}

\subsection{We run a few examples.}\label{14.C}

$ $

{\tiny
\begin{verbatim}
ii2 : load "SocDeg3.m2"

ii3 : P=QQ[x,y,z,w];

ii4 : Phi=x^2*y+z^3+x*z*w+x*w^2+w^3;

ii5 : (f11,f12,A,B,C,D,SM)=BasicSocDeg3(Phi);

ii6 : -- One can look at the output, for example:
      
      A

oo6 = | y  0 0 z  0   0  w  0  |
      | 0  y 0 0  x+z -x -x w  |
      | -w 0 y -x -w  z  -x -x |

              3       8
oo6 : Matrix P  <--- P

ii7 : -- We form the resolution of P/ann(x Phi):
      
      f1=map(P^1, ,f11|f12);

              1       9
oo7 : Matrix P  <--- P

ii8 : f2= map(source f1,,(A|x*B)||(x*C|D));

              9       16
oo8 : Matrix P  <--- P

ii9 : f3top=(x*transpose B)|(transpose D);

              8       9
oo9 : Matrix P  <--- P

ii10 : f3bot=(transpose A)|(x*transpose C);

               8       9
oo10 : Matrix P  <--- P

ii11 : f3= map(source f2,,(f3top||f3bot));

               16       9
oo11 : Matrix P   <--- P

ii12 : f4=map(source f3,,(transpose f11)||(transpose f12));

               9       1
oo12 : Matrix P  <--- P

ii13 : -- We verify that the f_i form a resolution of P/ann(x(Phi)).
       
       ideal(f1)==ideal(fromDual((contract(x,Phi))))

oo13 = true

ii14 : (image f2) == (ker f1)

oo14 = true

ii15 : (image f3)==(ker f2)

oo15 = true

ii16 : (image f4)==(ker f3)

oo16 = true

ii17 : -- We form the resolution of P/(ann Phi):
       
       F1=map(P^1, ,(x*f11)|f12);

               1       9
oo17 : Matrix P  <--- P

ii18 : F2= map(source F1,,(A|B)||(x^2*C|D));

               9       16
oo18 : Matrix P  <--- P

ii19 : F3top=((transpose B)|(transpose D));

               8       9
oo19 : Matrix P  <--- P

ii20 : F3bot=((transpose A)|(x^2*transpose C));

               8       9
oo20 : Matrix P  <--- P

ii21 : F3= map(source F2,,(F3top||F3bot) );

               16       9
oo21 : Matrix P   <--- P

ii22 : F4=map(source F3,,(transpose (x*f11))||(transpose f12));

               9       1
oo22 : Matrix P  <--- P

ii23 : -- We verify that the F_i form a resolution of P/ann(Phi).
       
       (ideal F1)==ideal(fromDual(Phi))

oo23 = true

ii24 : (image F2)==(ker F1)

oo24 = true

ii25 : (image F3)==(ker F2)

oo25 = true

ii26 : (image F4)==(ker F3)

oo26 = true

ii27 : SM

oo27 = {-1} | 1 0 0 |
       {-1} | 0 0 0 |
       {-1} | 0 0 0 |

               3       3
oo27 : Matrix P  <--- P

ii28 : -- The rank of SM is 1, so ann(Phi) is 7-generated.
       -- The form of SM is correct; so we may apply the script minres7
       -- to obtain the minimal resolution of P/ann(Phi)
       
       (E1,E2,E3,E4)=minres7(Phi);

ii29 : -- We verify that the E_i form a resolution of P/ann(Phi).
       
       ideal(E1)==ideal(fromDual(Phi))

oo29 = true

ii30 : (image E2)==(ker E1)

oo30 = true

ii31 : (image E3)==(ker E2)

oo31 = true

ii32 : (image E4)==(ker E3)

oo32 = true

ii33 : -- The Betti numbers of this resolution are:
       
       (numColumns(E4),numColumns(E3),numColumns(E2),numColumns(E1),numRows(E1)) 

oo33 = (1, 7, 12, 7, 1)

oo33 : Sequence

ii34 : -- These are the expected Betti numbers.
              
       -- We apply the script SumOfLinkedIdeals7 to write ann Phi as a sum of
       -- linked ideals.
              
       (B1,B2,B3,newL0,newL1)=SumOfLinkedIdeals7(Phi);

ii35 :        -- Notice that the B's form a Brown resolution of a perfect
              -- grade three ideal and the mapping cone of the map of
              -- complexes from Hom_P(B,P) to B, given by
              -- -newL0 transpose, newL1 transpose, -newL1, newL0,
              -- is the minimal homogeneous resolution of P/ann (Phi). 
              
              -- In the next example the ideal ann Phi has nine minimal
              -- generators. The matrices SM and B are both identically
              -- zero. One can read the minimal homogeneous resolution  
              -- of P/ann Phi by running
              -- (f11,f12,A,B,C,D,SM)=BasicSocDeg3(Phi). We apply the
              -- script SumOfLinkedIdeals9 in order to write ann(Phi)
              -- as a sum of linked perfect ideals.
              
              Phi=x*y^2+x*z^2+x^2*w;

ii36 : (K1,K2,K3,L0,L1)=SumOfLinkedIdeals9(Phi);

ii37 : -- In the next example the ideal ann Phi has six minimal generators.
       -- The matrix SM is invertible.
       -- The script minres6 exhibits P/ann Phi as a hypersurface section of
       -- a codimension three Gorenstein ring.
           
       Phi=-3*x^3-2*x^2*y+2*x^2*z+x*z^2-x^2*w-x*y*w+x*z*w+z^2*w;

ii38 : (d1,f,X)=minres6(Phi);

ii39 : ideal(d1):f==ideal(d1)

oo39 = true

ii40 : (ideal(d1)+ideal(f))==(ideal fromDual Phi)

oo40 = true

ii41 : d1=map(P^1,,d1);

               1       5
oo41 : Matrix P  <--- P

ii42 : d2=map(source d1,,X);

               5       5
oo42 : Matrix P  <--- P

ii43 : d3=map(source d1,,transpose d1);

               5       1
oo43 : Matrix P  <--- P

ii44 : image d2==ker d1

oo44 = true

ii45 : image d2==ker d1

oo45 = true

ii46 : (d1,f,X)=minres6(Phi);

ii47 : ideal(d1):f==ideal(d1)

oo47 = true

ii48 : -- Thus f is regular on P/J, where J is generated by the
       -- entries of d1.
       
       (ideal(d1)+ideal(f))==(ideal fromDual Phi)

oo48 = true

ii49 : -- Thus (J,f) = ann Phi
       
       d1=map(P^1,,d1);

               1       5
oo49 : Matrix P  <--- P

ii50 : d2=map(source d1,,X);

               5       5
oo50 : Matrix P  <--- P

ii51 : d3=map(source d1,,transpose d1);

               5       1
oo51 : Matrix P  <--- P

ii52 : image d2==ker d1

oo52 = true

ii53 : image d2==ker d1

oo53 = true

ii54 : -- Thus  (d1^T,  X,  d1) are the maps in a resolution of P/J. 
       -- In particular, P/J is a codimension three Gorenstein ring.
\end{verbatim}
}
\subsection{A few suitable $\Phi_3$'s.}

$ $

Let $P=\pmb k[x,y,z,w]$ be a polynomial ring in four variables over the field $\pmb k$. We list a few examples of $\Phi_3\in D_\bullet P_1^*$ with the following properties:
\begin{enumerate}[\rm(a)]
\item $\ann \Phi_3\cap P_1=(0)$,
\item $x$ is a weak Lefschetz element on $P/\ann \Phi_3$,
\item if $\SM$, from (\ref{SM}), has rank one, then
$$\SM=\bmatrix \a&0&0\\0&0&0\\0&0&0\endbmatrix$$ for some nonzero $\a$ in $\pmb k$. \end{enumerate}
Many of these $\Phi_3$ were obtained from examples in \cite{MVW} by changing the variables in order to make (a) -- (c) hold. 
The $\Phi$'s (\ref{9.X}), (\ref{9.P}), and (\ref{std}) were used in Section~\ref{14.C}. Of course, we had to convert these $\Phi$'s into the notation used by Macaulay2; see Remark~\ref{14.1}.

The base field is $\mathbb Q$ in all of the following examples.

\subsubsection{The corresponding Gorenstein ideal has 9 minimal generators.}

\begin{enumerate}[\rm(i)]
\item\label{9.X}$x^*y^{*(2)}+x^*z^{*(2)}+x^{*(2)}w^*$
\item\label{9.W}$3x^{*(3)}+2x^{*(2)}y^*+x^*y^{*(2)}+x^*z^*w^*+y^*z^*w^*+z^*w^{*(2)}$
\item\label{10.E}$\begin{cases}\phantom{+}5x^{*(3)}+2x^{*(2)}y^*+x^*y^{*(2)}+x^{*(2)}z^*+x^*y^*z^*+y^{*(2)}z^*+x^{*(2)}w^*\\+x^*w^{*(2)}+w^{*(3)}\end{cases}$
\item\label{10.A}$\begin{cases}\phantom{+}4x^{*(3)}+x^{*(2)}y^*+2x^{*(2)}z^*+x^*y^*z^*+x^*z^{*(2)}\\+y^*z^{*(2)}+x^{*(2)}w^*+x^*y^*w^*+y^{*(2)}w^*\end{cases}$
\item\label{10.B}$x^{*(2)}y^*+x^*z^*w^*+x^*w^{*(2)}+w^{*(3)}$
\item\label{10.C}$6x^{*(3)}+2x^{*(2)}y^*+x^*y^{*(2)}+x^{*(2)}z^*+2x^{*(2)}w^*+x^*z^*w^*+x^*w^{*(2)}+z^*w^{*(2)}$
\item\label{10.F}$x^{*(3)}+x^{*(2)}z^*+x^*y^*z^*+y^{*(2)}z^*+x^*w^{*(2)}+y^*w^{*(2)}$
\setcounter{nameOfYourChoice}{\value{enumi}}
\end{enumerate}

\subsubsection{The corresponding Gorenstein ideal has 7 minimal generators.}

\begin{enumerate}[\rm(i)]
\setcounter{enumi}{\value{nameOfYourChoice}}
\item\label{9.P}$x^{*(2)}y^*+z^{*(3)}+x^*z^*w^*+x^*w^{*(2)}+w^{*(3)}$
\item\label{9.V}$2x^{*(3)}+x^{*(2)}y^*+x^*y^{*(2)}+y^{*(3)}+x^*z^{*(2)}+y^*z^{*(2)}-x^*w^{*(2)}-y^*w^{*(2)}$
\item\label{9.Q}$x^{*(2)}y^*+x^*z^{*(2)}+x^*z^*w^*+z^*w^{*(2)}$
\item\label{9.S}$\begin{cases}\phantom{+}x^{*(3)}+x^*z^{*(2)}+y^*z^{*(2)}-2x^{*(2)}w^*-2x^*y^*w^*\\-2y^{*(2)}w^*+3x^*w^{*(2)}+3y^*w^{*(2)}-3w^{*(3)}\end{cases}$
\item\label{9.T}$-x^{*(2)}y^*-x^*y^{*(2)}-y^{*(3)}+x^*z^*w^*+y^*z^*w^*+w^{*(3)}$
\item\label{10.D}$\begin{cases}\phantom{+}x^{*(3)}+x^*z^{*(2)}+y^*z^{*(2)}-x^{*(2)}w^*-x^*y^*w^*-y^{*(2)}w^*\\+2x^*w^{*(2)}+2y^*w^{*(2)}-3w^{*(3)}\end{cases}$
\item\label{9.U}$\begin{cases}\phantom{+}7x^{*(3)}+6x^{*(2)}y^*+6x^*y^{*(2)}+6y^{*(3)}+2x^{*(2)}z^*+2x^*y^*z^*\\+2y^{*(2)}z^*+2x^{*(2)}w^*+2x^*y^*w^*+
2y^{*(2)}w^*+x^*z^*w^*+y^*z^*w^*\end{cases}$

\setcounter{nameOfYourChoice}{\value{enumi}}
\end{enumerate}

\subsubsection{The corresponding Gorenstein ideal has 6 minimal generators.}

\begin{enumerate}[\rm(i)]
\setcounter{enumi}{\value{nameOfYourChoice}}               
\item\label{std}      $-3x^{*(3)}-2x^{*(2)}y^*+2x^{*(2)}z^*+x^*z^{*(2)}-x^{*(2)}w^*-x^*y^*w^*+x^*z^*w^*+z^{*(2)}w^*$
\item\label{12.F}$\begin{cases}-2x^{*(3)}-2x^{*(2)}y^*+2x^{*(2)}z^*+x^*z^{*(2)}-x^{*(2)}w^*-x^*y^*w^*+x^*z^*w^*\\+z^{*(2)}w^*+x^*w^{*(2)}+w^{*(3)}\end{cases}$
 \end{enumerate}


\begin{thebibliography}{99}



\bibitem{AS22} N.~Abdallah and H.~Schenck, {\em Free Resolutions And Lefschetz Properties Of
Some Artin Gorenstein Rings Of Codimension Four}, 
J. Symbolic Comput. {\bf 121} (2024), Paper No. 102257, 13 pp.




\bibitem{B87} A.~Brown, {\em A structure theorem for a class of grade three perfect ideals}
J. Algebra {\bf 105} (1987),  308--327.




\bibitem{Eis} D.~Eisenbud, {\em Commutative algebra with a view toward algebraic geometry} Graduate Texts in Mathematics, {\bf 150}. Springer-Verlag, New York, 1995.



\bibitem{EKK1}S.~El Khoury and A.~Kustin, {\em Artinian Gorenstein algebras with linear resolutions}, J. Algebra {\bf 420} (2014), 402--474. 

\bibitem{EKK2}  S.~El Khoury and A.~Kustin, {\em The explicit minimal resolution constructed from a Macaulay inverse system}, J. Algebra {\bf 440} (2015), 145--186.

\bibitem{EKK3}  S.~El Khoury and A.~Kustin, {\em  The structure of Gorenstein-linear resolutions of Artinian algebras}, J. Algebra {\bf 453} (2016), 492--560. 

\bibitem{EKK-qp} S.~El Khoury and A.~Kustin, {\em Quadratically presented grade three Gorenstein ideals}, J. Algebra {\bf 622} (2023), 258--290.



\bibitem{M2}D.~Grayson and M.~Stillman,
          {\em Macaulay2, a software system for research in algebraic geometry},
          {Available at \tt{http://www2.macaulay2.com}
}

\bibitem{GL} T.~Gulliksen and G.~Levin, {\em Homology of Local Rings}, Queen’s Papers in Pure and Applied Mathematics, vol.20, Queen’s University, Kingston, Ont., 1969.

\bibitem{Ho75} M.~Hochster, {\em
Topics in the homological theory of modules over commutative rings},
CBMS Reg. Conf. Ser. Math., vol. 24, Amer. Math. Soc., Providence, RI, 1975.

\bibitem{J78} T.~J\'ozefiak, {\em Ideals generated by minors of a symmetric matrix}, Comment. Math. Helv. {\bf 53} (1978),  595--607.

\bibitem{Ko91} B.~Kotzev, {\em Determinantal ideals of linear type of a generic symmetric matrix}, J. Algebra {\bf 139} (1991),  484--504.

\bibitem{K22} A.~Kustin, {\em
Perfect modules with Betti numbers $(2,6,5,1)$},
J. Algebra {\bf 600} (2022), 71--124.

\bibitem{K23}A.~Kustin, {\em The weak Lefschetz property for  standard graded, Artinian Gorenstein  algebras of embedding dimension four and  socle degree three}, submitted for publication.

\bibitem{K74} R.~Kutz, {\em Cohen-Macaulay rings and ideal theory in rings of invariants of algebraic groups}, Trans. Amer. Math. Soc. {\bf 194} (1974), 115--129.

\bibitem{MVW} P.~Macias Marques, O.~Veliche, and J.~Weyman {\em Artinian Gorenstein Algebras Of Embedding Dimension Four And Socle Degree Three}, 
J. Algebra {\bf 638} (2024), 788--839.





\bibitem{PS} C.~Peskine and L.~Szpiro, {\em Liaison des vari\'et\'es alg\'ebriques. I},
Invent. Math. {\bf 26} (1974), 271--302.

\bibitem{U90} B.~Ulrich, {\em Sums of linked ideals},
Trans. Amer. Math. Soc. {\bf 318} (1990),  1--42. 

 \end{thebibliography}
\end{document}